\documentclass[10pt]{amsart}
\usepackage{graphicx, amsmath, fullpage, amssymb, amsthm, amsfonts, color, algorithm}
\usepackage{enumerate}
\newtheorem{theorem}{Theorem}[section]
\newtheorem{lemma}[theorem]{Lemma}

\theoremstyle{definition}
\newtheorem{definition}[theorem]{Definition}

\usepackage[toc,page]{appendix} 

\theoremstyle{remark}
\newtheorem{remark}[theorem]{Remark}
\numberwithin{equation}{section}
\usepackage[colorlinks,
            linkcolor=red,
            anchorcolor=red,
            citecolor=blue
            ]{hyperref}
 
\renewcommand{\epsilon}{\varepsilon}

\begin{document}

\title{Community detection in the sparse hypergraph stochastic block model}

    % Information for first author
\author{Soumik Pal}

\address{Department of Mathematics,
University of Washington, Seattle, WA 98195}

\email{soumikpal@gmail.com}

%    Information for second author
\author{Yizhe Zhu}
\address{Department of Mathematics, University of California, San Diego, La Jolla, CA 92093}
\email{yiz084@ucsd.edu}
\thanks{S.P. was supported in part by NSF grant DMS-1612483. Y.Z. was supported in part by NSF grant DMS-1949617}

\date{\today}

\begin{abstract}
We consider the community detection problem in sparse random hypergraphs.  Angelini et al. in \cite{angelini2015spectral} conjectured the existence of a sharp threshold on model parameters for community detection in sparse hypergraphs generated by a hypergraph stochastic block model. We solve the positive part of the conjecture for the case of two blocks: above the threshold, there is a spectral algorithm which asymptotically almost surely constructs a partition of the hypergraph correlated with the true partition. Our method is a generalization to random hypergraphs of the method developed by Massouli\'{e} in \cite{massoulie2014community} for sparse random graphs.
\end{abstract}

\maketitle

\section{Introduction}
Clustering is an important topic in  network analysis, machine learning, and computer vision \cite{hennig2015handbook}. Many clustering algorithms are based on graphs, which represent  pairwise relationships among data.  Hypergraphs can be used to represent higher-order relationships among objects, including co-authorship and citation networks, and they have been shown empirically to have advantages over graphs \cite{zhou2007learning}.  Recently hypergraphs have been used as the data model in machine learning, including recommender system \cite{tan2011using}, image retrieval \cite{liu2011hypergraph,agarwal2005beyond} and bioinformatics \cite{tian2009hypergraph}. The stochastic block model (SBM) is a generative model for random graphs with community structures, which serves as a useful benchmark for  clustering algorithms on  graph data. It is natural to have an analogous model for random hypergraphs to model higher-order relations. In this paper, we  consider a higher-order SBM called the hypergraph stochastic block model (HSBM). Before describing HSBMs, let's recall clustering on graph SBMs.

\subsection{The Stochastic block model for graphs}
In this section, we summarize the state-of-the-art results for graph SBM with two blocks of roughly equal size.
Let $\Sigma_n$ be the set of all pairs $(G, \sigma)$, where $G=([n],E)$ is a graph with vertex set $[n]$ and edge set $E$, $\sigma=(\sigma_1,\dots, \sigma_n)\in \{+1,-1\}^n$ are spins on $[n]$, i.e., each vertex $i\in[n]$ is assigned with a spin $\sigma_i\in \{-1,+1\}$. From this finite set $\Sigma_n$, one can generate a random element $(G,\sigma)$ in two steps. 
\begin{enumerate}
\item First generate i.i.d random variables $\sigma_i\in\{-1,+1\}$ equally likely for all $i\in [n]$.
\item 	Then given $\sigma=(\sigma_1,\dots, \sigma_n)$, we generate a random graph $G$ where each edge $\{i,j\}$ is included independently with probability $p$ if $\sigma_i=\sigma_j$ and with probability $q$ if $\sigma_i\not=\sigma_j$.
\end{enumerate}
The law of this pair $(G,\sigma)$ will be denoted by $\mathcal G(n,p,q)$.  In particular, we are interested in the model $\mathcal G(n,p_n,q_n)$ where $p_n,q_n$ are parameters depending on $n$. We use the shorthand notation $\mathbb P_{\mathcal G_n}$ to emphasize that the integration is taken under the law $\mathcal G(n,p_n,q_n)$.

Imagine  $C_1=\{i:\sigma_i=+1\}$ and $C_2=\{i:\sigma_i=-1\}$ as   two communities in the graph $G$. Observing only $G$  from a  sample $(G,\sigma)$ from the distribution $\mathcal G(n,p_n,q_n)$, the goal of community detection is to estimate the unknown vector $\sigma$ up to a sign flip. Namely, we  construct  label estimators $\hat{\sigma}_i\in \{\pm 1\}$ for each $i$ and consider the empirical overlap between $\hat{\sigma}$ and unknown $\sigma$ defined by 
	\begin{align}\label{ov}
	ov_n(\hat{\sigma},\sigma):=\frac{1}{n}\sum_{i\in [n]}\sigma_i\hat{\sigma}_{i}.\end{align} 
  We may ask the following questions about the estimation as $n$ tends to infinity: 
\begin{enumerate}
	\item Exact recovery (strong consistency): 
	\begin{align*}
	\lim_{n\to\infty}\mathbb P_{\mathcal G_n}\left(\{ov_n(\hat{\sigma},\sigma)=1\}\cup \{ov_n(\hat{\sigma},\sigma)=-1\}\right)=1.	
	\end{align*}

		\item Almost exact recovery (weak consistency):  for any $\epsilon>0$,
		\begin{align*} 
			\lim_{n\to\infty}\mathbb P_{\mathcal G_n}\left(\{|ov_n(\hat{\sigma},\sigma)-1|>\epsilon\}\cap \{|ov_n(\hat{\sigma},\sigma)+1|>\epsilon\}\right)=0.
		\end{align*}
	\item Detection: Find a partition which is correlated with the true partition. More precisely, there exists a constant $r>0$ such that it satisfies the following: for any $\epsilon>0$,
 \begin{align}
 	\lim_{n\to\infty}\mathbb P_{\mathcal G_n}(\{|ov(\hat{\sigma},\sigma)-r|>\epsilon\}\cap \{|ov(\hat{\sigma},\sigma)+r|>\epsilon\} )=0.
 \end{align}
\end{enumerate}

There are many works on these questions using different tools, we list some of them. A conjecture of  \cite{decelle2011asymptotic} based on non-rigorous ideas from statistical physics predicts a threshold of  detection in the SBM, which is  called the Kesten-Stigum threshold. In particular, if $p_n=\frac{a}{n}$ and $q_n=\frac{b}{n}$ where $a,b$ are positive constants independent of $n$, then the detection is possible if and only if $(a-b)^2>2(a+b)$. This conjecture was confirmed in \cite{mossel2015reconstruction,mossel2018proof,massoulie2014community,bordenave2018nonbacktracking} where \cite{mossel2018proof,massoulie2014community,bordenave2018nonbacktracking} provided efficient algorithms to achieve the threshold. Very recently, two alternative spectral algorithms were proposed based on distance matrices \cite{stephan2018robustness} and a graph powering method in \cite{abbe2018graph}, and they both achieved the detection threshold.

Suppose $p_n=\frac{a\log n}{n}, q_n=\frac{b\log n}{n}$ where $a,b$ are constant independent of $n$. Then the exact recovery is possible if and only if $(\sqrt{a}-\sqrt{b})^2>2$, which was solved in \cite{abbe2016exact,hajek2016achieving} with efficient algorithms achieving the threshold. 
Besides the phase transition behavior, various algorithms were proposed and analyzed in different regimes and more general settings beyond the $2$-block SBM \cite{brito2016recovery,chien2018community,guedon2016community,abbe2018proof,krzakala2013spectral,mossel2016belief,cole2020exact,stephan2020non,bordenave2020detection}, including spectral methods, semidefinite programming, belief-propagation, and approximate message-passing algorithms. We recommend \cite{abbe2018community} for further details.

\subsection{Hypergraph stochastic block models}

The hypergraph stochastic block model (HSBM) is a generalization of the SBM for graphs, which was first studied in \cite{ghoshdastidar2014consistency}, where the authors consider hypergraphs generated by the stochastic block models that are dense and uniform. A faithful representation of a hypergraph is its adjacency tensor (see Definition \ref{def:tensor}). However, most of the computations involving tensors are NP-hard \cite{hillar2013most}. Instead,
they considered spectral algorithms for exact recovery using hypergraph Laplacians. Subsequently, they extended their results to sparse, non-uniform hypergraphs \cite{ghoshdastidar2015provable,ghoshdastidar2015spectral,ghoshdastidar2017consistency}. 
For exact recovery, it was shown that the phase transition occurs in the regime of logarithmic average degrees in \cite{lin2017fundamental,chien2018minimax,chien2018community} and the exact threshold was given in \cite{kim2018stochastic}, by a generalization of the techniques in \cite{abbe2016exact}.  Almost exact recovery for HSBMs was studied in \cite{chien2018community,chien2018minimax,ghoshdastidar2017consistency}.

For detection of the HSBM with two blocks, the authors of \cite{angelini2015spectral} proposed a conjecture that the phase transition occurs in the regime of constant average degree, based on the performance of the belief-propagation algorithm. Also, they conjectured a spectral algorithm based on non-backtracking operators on hypergraphs could reach the threshold. In \cite{florescu2016spectral}, the authors showed  an algorithm for detection when the  average degree is bigger than some constant by reducing it to a bipartite stochastic block model. They also mentioned a barrier to further improvement.  We confirm the positive part of the conjecture in \cite{angelini2015spectral} for the case of two blocks: above the threshold, there is a spectral algorithm which asymptotically almost surely constructs a partition of the hypergraph  correlated with the true partition.

 Now we specify our $d$-uniform hypergraph stochastic block model with two clusters. Analogous to $\mathcal G(n,p_n,q_n)$, we define $\mathcal H(n,d,p_n,q_n)$ for $d$-uniform hypergraphs. Let $\Sigma_n$ be  the set of  all  pair $(H, \sigma)$, where $H=([n],E)$ is a $d$-uniform hypergraph (see Definition \ref{def:hypergraph} below) with vertex set $[n]$ and hyperedge set $E$, $\sigma=(\sigma_1,\dots, \sigma_n)\in \{+1,-1\}^n$ are the spins on $[n]$. From this finite set $\Sigma_n$, one can generate a random element $(H,\sigma)$ in two steps. 
\begin{enumerate}
\item First generate i.i.d random variables $\sigma_i\in\{-1,+1\}$ equally likely for all $i\in [n]$.
\item 	Then given $\sigma=(\sigma_1,\dots, \sigma_n)$, we generate a random hypergraph $H$ where each hyperedge $\{i_1, \dots i_d\}$ is included independently with probability $p_n$ if $\sigma_{i_1}=\dots =\sigma_{i_d}$ and with probability $q_n$ if the spins  $\sigma_{i_1},\dots \sigma_{i_d}$ are not the same.\end{enumerate}
The law of this pair $(H,\sigma)$ will be denoted by $\mathcal H(n,d,p_n,q_n)$. We use the shorthand notation $\mathbb P_{\mathcal H_n}$  and  $\mathbb E_{\mathcal H_n}$ to emphasize that integration is taken under the law $\mathcal H(n,d,p_n,q_n)$.  Often we drop the index $n$ from our notation, but it will be clear from  $\mathbb P_{\mathcal H_n}$.

\begin{figure}
\includegraphics[width= 0.25\linewidth]{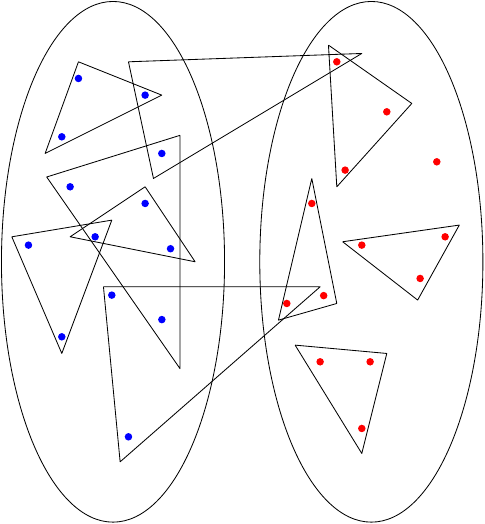}	
\caption{A HSBM with $d=3$. Vertices in blue and red have spin $+$ and $-$, respectively.}
\end{figure}

  \subsection{Main results}
 We consider the detection problem of the HSBM in the constant expected degree regime.
  Let \[p_n:=\frac{a}{{{n}\choose {d-1}}}, \quad   q_n:=\frac{b}{{{n}\choose {d-1}}}\] for some constants $a\geq b>0$ and a constant integer $d\geq 3$. Let
   \begin{align}\label{alphabeta}
\alpha:=(d-1)\frac{a+(2^{d-1}-1)b}{2^{d-1}},\quad \beta:=(d-1)\frac{a-b}{2^{d-1}}.\end{align}
  Here $\alpha$ is a constant which measures the expected degree of any vertex, and $\beta$ measures the discrepancy between the number of neighbors with $+$ sign and $-$ sign of any vertex.  For $d=2$,  $\alpha,\beta$ are the same parameters for the graph case  in \cite{massoulie2014community}.  Now we are able to state our main result which is an extension of the result of for graph SBMs in \cite{massoulie2014community}.  Note that with the definition of $\alpha,\beta$, we have $\alpha>\beta$. The condition $\beta^2>\alpha$ in the statement of Theorem \eqref{main}  below implies	 $\alpha,\beta>1$, which will be assumed for the rest of the paper. 
  \begin{theorem}\label{main}
	Assume $\beta^2>\alpha$. Let $(H,\sigma)$ be a random labeled hypergraph sampled from $\mathcal H(n,d,p_n,q_n)$ and  $B^{(l)}$ be its $l$-th self-avoiding matrix (see Definition \ref{SAM} below). Set $l= c\log(n)$ for a constant $c$ such that $c\log(\alpha)<1/8$.  Let $x$ be a $l_2$-normalized eigenvector corresponding to the second largest  eigenvalue of $B^{(l)}$. There exists a constant $t$ such that, if we define the label estimator $\hat{\sigma}_i$ as 
	\begin{align*}
		\hat{\sigma}_i=\begin{cases}
			+1 & \text{if $x_i\geq t/\sqrt{n}$,}\\
			-1 & \text{otherwise},
		\end{cases}
	\end{align*}
	then detection is possible. More precisely, there exists a constant $r>0$ such that the empirical overlap between $\hat{\sigma}$ and $\sigma$ defined similar to \eqref{ov} satisfies the following: 
	for any $\epsilon>0$,
 \begin{align*}
 	\lim_{n\to\infty}\mathbb P_{\mathcal H_n}\left(\{|ov_n(\hat{\sigma},\sigma)-r|>\epsilon\}\bigcap \{|ov_n(\hat{\sigma},\sigma)+r|>\epsilon\} \right)=0.
 \end{align*}
\end{theorem}

\begin{remark} If we take $d=2$, the condition $\beta^2>\alpha$ is the threshold for detection in graph SBMs proved in \cite{massoulie2014community,mossel2015reconstruction,mossel2018proof}.  When $d\geq 3$,
  the conjectured detection threshold for HSBMs is given in Equation (48) of \cite{angelini2015spectral}. With our notations, in the 2-block case, Equation (48) in \cite{angelini2015spectral} can be written as
$\frac{\alpha-\beta}{\alpha+\beta}=\frac{\sqrt{\alpha}-1}{\sqrt{\alpha}+1}$, which says $\beta^2=\alpha$ is the conjectured detection threshold for HSBMs. This is an analog of the Kesten-Stigum threshold proved in the graph case \cite{decelle2011asymptotic,mossel2015reconstruction,mossel2018proof,massoulie2014community,bordenave2018nonbacktracking}. Our Theorem \ref{main} proves the positive part of the conjecture.
\end{remark}

Our algorithm can be summarized in two steps. The first step is  a dimension reduction:  $B^{(l)}$ has $n^2$ many entries from the original adjacency tensor $T$ (see Definition \ref{def:tensor}) of $n^d$ many entries. Since the $l$-neighborhood of any vertex contains at most one cycle with high probability (see Lemma \ref{tangle}), by breadth-first search, the matrix $B^{(l)}$ can be constructed in polynomial time. The second step is a simple spectral clustering according to leading eigenvectors as the common clustering algorithm in the graph case.   

Unlike graph SBMs, in the HSBMs, the random hypergraph $H$ we observe is essentially a random tensor. Getting the spectral information of a tensor is NP-hard \cite{hillar2013most} in general,  making the corresponding problems in HSBMs very different from graph SBMs. It
is not immediately clear which operator  to associate to
$H$ that encodes the community structure in the bounded expected degree regime. The novelty of our method is a way to project the random tensor into  matrix forms (the self-avoiding matrix $B^{(l)}$ and the adjacency matrix $A$) that give us the community structure from their leading eigenvectors. In practice, the hypergraphs we  observed are usually not $d$-uniform, which can not be  represented as a tensor. However, we can still construct the matrix $B^{(l)}$ since the definition of self-avoiding walks does not depend on the uniformity assumption. In this paper, we focus on the $d$-uniform case to simplify the presentation, but our proof techniques can be applied to the non-uniform case.

The analysis of HSBMs is  harder than the original  graph SBMs due to the extra dependency in the hypergraph structure and the lack of linear algebra tools for tensors. To overcome these difficulties,  new techniques are developed in this paper to establish the desired results.

There are multiple ways to define self-avoiding walks on hypergraphs, and our definition  (see Definition \ref{def:walk}) is the only one that works for us when applying the moment method. We  develop a moment method suitable for sparse random hypergraphs in Section \ref{sec:matrixmomentmethod} that controls the spectral norms by counting concatenations of self-avoiding walks on hypergraphs.  The combinatorial counting argument  in the proof of Lemma \ref{expectboundmoment}  is more involved as we need to consider labeled vertices and labeled hyperedges. The moment method for hypergraphs developed here could be of independent interest for other random hypergraph problems.

The growth control of the size of the local neighborhood (Section \ref{sec:localanalysis}) for HSBMs turns out to be more challenging compared to graph SBMs in \cite{massoulie2014community}  due to the dependency between the number of vertices with spin $+$ and $-$, and overlaps between different hyperedges. We use a new second-moment estimate to obtain a matching lower bound and upper bound for the size of the neighborhoods in the proof of Theorem \ref{growthbound}.
The issues mentioned above do not appear in the sparse random graph case.

To analyze the local structure of HSBMs, we prove a new  coupling result between a typical neighborhood of a vertex in the sparse random hypergraph $H$ and a multi-type Galton-Watson hypertree described in Section \ref{sec:coupling}, which is a stronger version of local weak convergence of sparse random hypergraphs (local weak convergence for hypergraphs was recently introduced in \cite{delgosha2018load}). Compared to the classical 2-type Galton-Watson tree in the graph case, the vertex $\pm$ labels in a hyperedge is not assigned independently. We carefully designed the probability of different types of hyperedges that appear in the hypertree to match the local structure of the HSBM.  Combining all the new ingredients, we obtain the weak Ramanujan property of $B^{(l)}$ for sparse HSBMs in Theorem \ref{rama3} as a generalization of the results in \cite{massoulie2014community}.  We conclude the proof of  our  Theorem \ref{main} in Section \ref{sec:maintheorem}. 

Our Theorem \ref{main} deals with the positive part of the phase transition conjecture in \cite{angelini2015spectral}.
To have a complete characterization of the phase transition,  one needs to show an impossibility result when $\beta^2<\alpha$. Namely, below this threshold, no algorithms (even with exponential running time) will solve the detection problem  with high probability. For graph SBMs, the impossibility result was proved in \cite{mossel2015reconstruction} based on a reduction to the broadcasting problem  on Galton-Watson trees analyzed in \cite{evans2000broadcasting}. To answer the corresponding problem in the HSBMs, one needs to establish a similar information-theoretical lower bound for the broadcasting problem on hypertrees and relate the problem to the detection problem on HSBMs. To the best of our knowledge, even for the very first step,  the broadcasting problem on hypertrees has not been studied yet. The  multi-type Galton-Watson hypertrees described in Section \ref{sec:coupling} can be used as a model to study this type of problem on hypergraphs. We leave it as a future direction.

\section{Preliminaries}\label{sec:prep}

 \begin{definition}[hypergraph]\label{def:hypergraph}
 	A  \textit{hypergraph} 
 	$H$ is a pair $H=(V,E)$ where $V$ is a set of vertices and $E$ is the set of non-empty subsets of $V$ called \textit{hyperedges}. If any hyperedge $e\in E$ is a set of $d$ elements of $V$, we call $H$ \textit{$d$-uniform}. In particular, $2$-uniform hypergraph is an ordinary graph. A $d$-uniform hypergraph is complete if any set of $d$ vertices is a hyperedge and we denote a  complete $d$-uniform hypergraph on $[n]$ by $K_{n,d}$. The \textit{degree} of a vertex $i\in V$ is the number of hyperedges in $H$ that contains $i$.
 \end{definition}
 
 \begin{definition}[adjacency tensor]\label{def:tensor}
 Let $H=(V,E)$ be a $d$-uniform hypergraph with $V=[n]$. We define $T$ to be the adjacency tensor of $H$ such that
 	for any set of  vertices $ \{ i_1,i_2,\dots ,i_d\}$,
 	 \begin{align*}
	T_{i_1,\dots,i_d}=\begin{cases}
		1 &  \text{if } \{i_1,\dots, i_d\}\in E,\\
		0 & \text{otherwise.}
	\end{cases}
\end{align*}
We set $T_{\sigma(i_1),\sigma({i_2}),\dots, \sigma(i_d)} = T_{i_1,\dots, i_d}$ for any permutation $\sigma$.    We may  write $T_e$ in place of $T_{i_1, \ldots, i_d}$ where $e = \{i_1, \ldots, i_d\}$.   \end{definition}

 \begin{definition}[adjacency matrix] The adjacency matrix $A$ of a $d$-uniform hypergraph $H=(V,E)$ with vertex set $[n]$ is a $n\times n$ symmetric matrix such that for any $i\not=j$, $A_{ij}$ is the number of hyperedges in $E$ which contains $i,j$ and $A_{ii}=0$ for $i\in [n]$. Equivalently, we have
 \begin{align*}
	A_{ij}=\begin{cases}
		\sum_{e: \{i,j\}\in e} T_e &  \text{if } i\not=j,\\
		0 & \text{if } i=j.
	\end{cases}
\end{align*}  

 \end{definition}
 
\begin{definition}[walk]\label{def:walk}
	A \textit{walk} of length $l$ on a hypergraph $H$ is a sequence $(i_0,e_1,i_1,\cdots ,e_l, i_l)$ such that $i_{j-1}\not=i_{j}$ and $\{i_{j-1},i_j\}\subset e_j$ for all $1\leq j\leq l$. A walk is closed if $i_0=i_l$ and we call it a \textit{circuit}.  A \textit{self-avoiding walk} of length $l$ is a walk $(i_0,e_1,i_1,\cdots, e_l,i_l)$ such that
  \begin{enumerate}
  	\item $|\{i_0,i_1,\dots ,i_l\}|=l+1$.
  	\item Any consecutive hyperedges $e_{j-1},e_{j}$ satisfy $e_{j-1}\cap e_{j}=\{i_{j-1}\}$ for $2\leq j\leq l$.
  	\item Any two hyperedges $e_j,e_k$ with $1\leq j<k\leq l, k\not=j+1$ satisfy  $e_j\cap e_k=\emptyset$.
  \end{enumerate} 
  \begin{figure}
      \centering
      \includegraphics[width=0.25\linewidth]{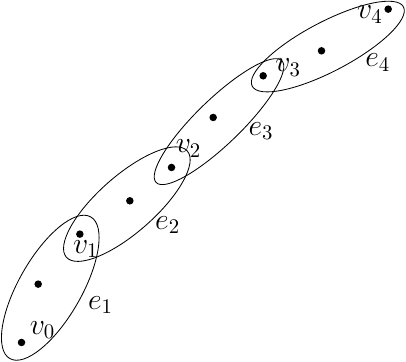}
      \caption{a self-avoiding walk of length $4$ denoted by $(v_0,e_1,v_1,e_2,v_2,e_3,v_3,e_4,v_4)$}
      \label{fig:SAW}
  \end{figure}
  
  See Figure \ref{fig:SAW} for an example of a self-avoiding walk in a $3$-uniform hypergraph.
  Recall that a self-avoiding walk of length $l$ on a graph is a walk $(i_0,\dots, i_l)$ without repeated vertices. Our definition is a generalization of the self-avoiding walk to hypergraphs.
\end{definition}

\begin{definition}[cycle and hypertree]\label{hypertree}
	A \textit{cycle} of length $l$ with $l\geq 2$ in a hypergraph $H$  is a walk $(i_0,e_1,\dots, i_{l-1},e_l,i_0)$ such that $i_0,\dots i_{l-1}$ are distinct vertices and $e_1\dots e_l$ are distinct hyperedges. A \textit{hypertree} is a hypergraph which contains no cycles.
\end{definition}
Let ${[n]\choose d}$ be the collection of all subsets of $[n]$ with size $d$. 
For any subset $e\in {[n]\choose d}$ and $i\not=j\in [n]$, we define
\begin{align*}
	A_{ij}^e=\begin{cases}
		1 &  \text{if } \{i,j\}\in e\text{ and } e\in E,\\
		0 & \text{otherwise,} 
	\end{cases}
\end{align*} and  we define $A_{ii}^e=0$ for all $i\in [n]$. With our notation above, $ A_{ij}=\sum_{e\in {[n]\choose d}} A_{ij}^e.$ We have the following expansion of the trace of $A^k$ for any integer $k\geq 0$:
\begin{align*}
\textnormal{tr}A^k&=\sum_{ i_0,i_2,\dots, i_{k-1}\in [n]}A_{i_0i_1}A_{i_2i_3}\cdots A_{i_{k-1}i_0}=\sum_{\substack{ i_0,i_1,\dots, i_{k-1}\in [n]\\ e_1,\dots, e_k\in {[n]\choose d}}}A_{i_0i_1}^{e_1}\cdots A_{i_{k-2}i_{k-1}}^{e_{k-1}}A_{i_{k-1}i_0}^{e_{k}}.
\end{align*}
Therefore, $\textnormal{tr}A^k$ counts the number of circuits $(i_0,e_1,i_1,\dots, i_{k-1},e_k,i_0)$ in the hypergraph $H$ of length $k$. This connection was used in \cite{lu2012loose} to study the spectra of the Laplacian of random hypergraphs.
 From our definition of self-avoiding walks on hypergraphs, we associate a self-avoiding adjacency matrix to the hypergraph.
  \begin{definition}[self-avoiding matrix]\label{SAM}
   Let $H=(V,E)$ be a hypergraph with $V=[n]$. For any $l\geq 1$,
a $l$-th \textit{self-avoiding matrix} $B^{(l)}$ is a $n\times n$ matrix where for $i\not=j\in [n]$, $B_{ij}^{(l)}$ counts the number of self avoiding walks of length $l$ from $i$ to $j$ and $B_{ii}^{(l)}=0$ for $i\in [n]$.
   \end{definition}
$B^{(l)}$ is a symmetric matrix since a time-reversing self avoiding walk from $i$ to $j$ is  a self avoiding walk from $j$ to $i$.
Let $\text{SAW}_{ij}$ be the set of all self-avoiding walks of length $l$ connecting $i$ and $j$ in the complete $d$-uniform hypergraph on vertex set $[n]$. We denote a walk of length $l$  by  $w=(i_0,e_{i_1},\dots, i_{l-1},e_{i_l},i_l).$  Then for any $i,j\in [n]$,
\begin{align}\label{Bij}
	B_{ij}^{(l)}=\sum_{w\in \text{SAW}_{ij}}\prod_{t=1}^l A_{i_{t-1}i_t}^{e_{i_t}}.
\end{align}

\section{Matrix expansion and spectral norm bounds}\label{sec:matrixexpan}
Consider a random labeled $d$-uniform hypergraph $H$ sampled from $\mathcal H\left(n,d,p_n,q_n\right)$ with adjacency matrix $A$ and self-avoiding matrix $B^{(l)}$. Let 
$  \overline{A}:=\mathbb E_{\mathcal H_n}[A\mid \sigma].
$
 Let 
$  \rho(A):=\sup_{x: \|x\|_2=1}\|Ax\|_2$ be the spectral norm of a matrix $A$.  Recall \eqref{Bij}, define
\begin{align}\label{Delta_ijl}
\Delta_{ij}^{(l)}:=\sum_{w\in \text{SAW}_{ij}}\prod_{t=1}^l (A_{i_{t-1}i_t}^{e_{i_t}}-\overline{A}_{i_{t-1}i_t}^{e_{i_t}}),
\end{align}
where $\overline{A}_{i_{t-1}i_t}^{e_{i_t}}=\mathbb E_{\mathcal H_n}[A_{i_{t-1}i_t}^{e_{i_t}} \mid \sigma]$. 
$\Delta^{(l)}$ can be regarded as a centered version of $B^{(l)}$. We will apply the classical moment method to estimate the spectral norm of $\Delta^{(l)}$, since this method works well for centered random variables.  Then we can relate the spectrum of $\Delta^{(l)}$ to the spectrum of $B^{(l)}$ through a matrix expansion formula
which connects $\overline{A}$, $B^{(l)}$ and $\Delta^{(l)}$ in the following theorem.  Recall the definition of $\alpha$ in \eqref{alphabeta}. 

\begin{theorem}\label{matrixexpansionthm} Let $H$ be a random hypergraph sampled from $\mathcal H\left(n,d,p_n,q_n\right)$ and $B^{(l)}$ be its $l$-th self avoiding matrix.
Then the following holds.
\begin{enumerate}
\item There exist some matrices  $\{\Gamma^{(l,m)}\}_{m=1}^l$ such that for any $l\geq 1$, $B^{(l)}$ satisfies the identity
\begin{align}
B^{(l)}=\Delta^{(l)}+\sum_{m=1}^l (\Delta^{(l-m)}\overline{A}B^{(m-1)})-\sum_{m=1}^l \Gamma^{(l,m)}.	\label{matrixexpansion}
\end{align}
\item  For any sequence $l_n=O(\log n)$ and any fixed $\epsilon>0$, 
	\begin{align}\label{specrho}
	\lim_{n\to\infty}\mathbb P_{\mathcal H_n}\left(\rho(\Delta^{(l_n)})\leq n^{\epsilon}\alpha^{l_n/2}\right) &=1,\\
	 \lim_{n\to\infty}\mathbb P_{\mathcal H_n}\left(\bigcap_{m=1}^{l_n}\left\{\rho(\Gamma^{(l_n,m)})\leq n^{\epsilon-1}\alpha^{(l_n+m)/2}\right\}\right) &=1.	\label{specrho2}
	\end{align}
	\end{enumerate}
\end{theorem}

Theorem \ref{matrixexpansionthm} is one of the main ingredients to show $B^{(l)}$ has a spectral gap. Together with the local analysis in Section \ref{sec:localanalysis}, we will show in Theorem \ref{rama3} that the bulk eigenvalues of $B^{(l)}$ are separated from the first and second eigenvalues. The proof of Theorem \ref{matrixexpansionthm} is deferred to Section \ref{sec:matrixmomentmethod}.
The matrices $\{\Gamma^{(l,m)}\}_{m=1}^l$ in Theorem \ref{matrixexpansionthm} record concatenations of self-avoiding walks with different weights, which will be carefully analyzed in Lemma \ref{lemma:bound_on_gamma} of Section \ref{sec:matrixmomentmethod}.

\section{Local analysis}\label{sec:localanalysis}

In this section, we  study the structure of the local neighborhoods in the HSBM. Namely, what the neighborhood of a typical vertex in the random hypergraph looks like.

\begin{definition}\label{DefVt}
In a hypergraph $H$, we define the \textit{distance} $d(i,j)$ between two vertices $i,j$ to be the minimal length of walks between $i$ and $j$. Define the \textit{$t$-neighborhood} $V_t(i)$ of a fixed vertex $i$ to be the set of vertices which have distance $t$ from $i$. Define $
V_{\leq t}(i):=\bigcup_{k\leq t} V_k(i)
$ to be the set all of vertices which have distance at most $t$ from $i$ and $V_{>t}=[n]\setminus V_{\leq t}$. Let $V_t^{\pm}(i)$ be the vertices in $V_t(i)$ with spin $\pm$ and define it similarly for $V_{\leq t}^{\pm}(i)$.
\end{definition}
For $i\in [n]$, define \[ S_t(i):=|V_t(i)|, \quad 
D_t(i):=\sum_{j: d(i,j)=t}\sigma_j.	\]

Let $\mathbf{1}=(1\dots, 1)\mathbb \in \mathbb R^n$ and recall $\sigma\in \{-1,1\}^n$. 
We will show that when $l=c\log n$ with $c\log\alpha<1/8$, $S_l(i), D_l(i)$ are close to the corresponding quantities $(B^{(l)}\mathbf{1})_i, (B^{(l)}\sigma)_i$ (see Lemma \ref{rama1}). In particular, the vector $(D_l(i))_{1\leq i\leq n}$ is asymptotically aligned with the second  eigenvector of $B^{(l)}$, from which we get the information on the partitions.  
We give the following growth estimates of $S_t(i)$ and $D_t(i)$. The proof of Theorem \ref{quasi} is given in Section \ref{sec:thm43}.
\begin{theorem}\label{quasi}
	Assume $\beta^2>\alpha> 1$ and $l=c\log n$, for a constant $c$ such that $c\log\alpha<1/4$. There exists constants $C,\gamma>0$ such that for sufficiently large $n$, with probability at least $1-O(n^{-\gamma})$ the following holds for all $i\in [n]$  and  $1\leq t\leq l$:
	\begin{align}
	S_t(i)&\leq C\log (n)\alpha^t\label{claim1},\\
	|D_t(i)|&\leq C\log (n)\beta^t\label{claim2},\\
	S_t(i)&=\alpha^{t-l}S_l(i)+O(\log(n)\alpha^{t/2})\label{claim3},\\
	D_t(i)&=\beta^{t-l}D_l(i)+O(\log (n)\alpha^{t/2})\label{claim4}.
	\end{align}
\end{theorem}

The approximate independence of neighborhoods of distinct vertices is given in the following lemma. It will be used later  to analyze the martingales constructed on the Galton-Watson hypertree defined in Section \ref{sec:coupling}. The proof of Lemma \ref{approxind} is given in Appendix \ref{sec:appendix_approx}.

\begin{lemma}\label{approxind}
	For any two fixed vertices $i\not=j$, let $l=c\log(n)$ with constant $c\log(\alpha)<1/4$. Then the total variation distance between the joint law  $\mathcal L((U_{k}^{\pm}(i))_{k\leq l}, (U_k^{\pm}(j))_{k\leq l})$ and the law with the same marginals and independence between them, denoted by $\mathcal L((U_{k}^{\pm}(i))_{k\leq l}\otimes (U_k^{\pm}(j))_{k\leq l})$, is $O(n^{-\gamma})$ for some $\gamma>0$.
\end{lemma}

Now we consider number of cycles in $V_{\leq l}(i)$ of any vertex $i\in [n]$.
	We say $H$ is \textit{$l$-tangle-free} if for any $i\in[n]$, there is  no more than one cycle in $V_{\leq l}(i)$. 
\begin{lemma}\label{tangle}
	Assume $l=c\log n$ with $c\log(\alpha)<1/4$. Let $(H,\sigma)$ be a sample from $\mathcal H\left(n,d,p_n,q_n\right)$. Then  
	\begin{align*}
	&\lim_{n\to\infty}\mathbb P_{\mathcal H_n}\left(|\{i\in [n]: V_{\leq l}(i) \text{ contains at least one cycle}\}|\leq \log^4(n)\alpha^{2l}\right)=1,\\	
	 &\lim_{n\to\infty}\mathbb P_{\mathcal H_n} \left( H \text{ is } l\text{-tangle-free} \right) =1.
	 \end{align*} \end{lemma}

The proof of Lemma \ref{tangle} is given in Appendix \ref{A4}.
In the next lemma, we translate the local analysis of the neighborhoods to the control of vectors $B^{(m)}\mathbf{1}, B^{(m)}\sigma$. The proof is similar to the proof of Lemma 4.3 in \cite{massoulie2014community}, and we include it in Appendix \ref{sec:TangleCount}. For any event $A_n$, we say $A_n$ happens \textit{asymptotically almost surely} if $\lim_{n\to\infty}\mathbb P_{\mathcal H_n}(A_n)=1$.

\begin{lemma}\label{TangleCount}
	Let $\mathcal B$ be the set of vertices $i$ whose $l-$neighborhood contains a cycle. For $l=c\log n$ with $c\log(\alpha)<1/4$, asymptotically almost surely the following holds:
	\begin{enumerate}
		\item for all $m\leq l$ and all $i\not\in \mathcal B$ the following holds
	\begin{align}
	(B^{(m-1)}\mathbf{1})_i&=\alpha^{m-1-l}(B^{(l)}\mathbf{1})_i+O(\alpha^{(m-1)/2}\log n),\label{Be}\\
	(B^{(m-1)}\sigma)_i&=\beta^{m-1-l}(B^{(l)}\sigma)_i+O(\alpha^{(m-1)/2}\log n).\label{Bs}	
	\end{align}
	\item For all $i\in \mathcal B$:
	\begin{align}\label{tanglecount}
	|(B^{(m)}\sigma)_i	|\leq |(B^{(m)}\mathbf{1})_i	|\leq 2\sum_{t=0}^m S_t(i)=O(\alpha^m \log n).
	\end{align}
	\end{enumerate}	
\end{lemma}

Combining  Theorem \ref{matrixexpansionthm}, Theorem \ref{quasi}, and Lemma \ref{TangleCount},  we are able to prove the following theorem. 

\begin{theorem}\label{thm:endsection}
	Assume $\beta^2>\alpha>1$ and $l=c\log n$ with $c\log(\alpha)<1/8$. Then the following holds: for any $\epsilon>0$
	\begin{align*}
		\lim_{n\to\infty}\mathbb P_{\mathcal H_n}\left(\sup_{\|x\|_2=1,x^{\top}(B^{(l)}\mathbf{1})=x^{\top}(B^{(l)}\sigma)=0}\|B^{(l)}x\|_2\leq n^{\epsilon}\alpha^{l/2}\right)=1.
	\end{align*} 
\end{theorem} 
Theorem \ref{thm:endsection} is a key ingredient to prove the bulk eigenvalues of $B^{(l)}$ are $O(n^{\epsilon}\alpha^{l/2})$ in Theorem \ref{rama3}.  The proof of Theorem \ref{thm:endsection} is given in in Section \ref{sec:given}.

\section{Coupling with multi-type Poisson hypertrees}\label{sec:coupling}

Recall the definition of a hypertree from Definition \ref{hypertree}.
We construct a hypertree growth process in the following way. The hypertree is designed to obtain a coupling with the local neighborhoods of the random hypergraph $H$.
\begin{itemize}
	\item Generate a root $\rho$ with spin $\tau(\rho)=+$, then generate $ \textnormal{Pois}\left(\frac{\alpha}{d-1}\right)$ many hyperedges that only intersects at $\rho$. Call the vertices in these hyperedges except $\rho$ to be the \textit{children} of $\rho$ and of generation $1$. Call $\rho$ to be their \textit{parent}.
	\item For $0\leq r\leq d-1$, we define a hyperedge is of type $r$ if $r$ many children in the hyperedge has spin $\tau(\rho)$ and $(d-1-r)$ many children has spin $-\tau(\rho)$. We first assign a type for each hyperedge independently. Each hyperedge will be of type $(d-1)$ with probability $ \frac{(d-1)a}{\alpha 2^{d-1}}$ and of type $r$ with probability $ \frac{(d-1)b{ d-1 \choose r}}{\alpha 2^{d-1}}$ for $0\leq r\leq d-2$. Since 
$ \frac{(d-1)a}{\alpha 2^{d-1}}+\sum_{r=0}^{d-2}\frac{(d-1)b{ d-1 \choose r}}{\alpha 2^{d-1}}=1,
$  the probabilities of being various types of hyperedges add up to $1$. Because the type is chosen i.i.d for each hyperedge, by Poisson thinning, the number of hyperedges of different types are independent and Poisson. 
\item We draw the hypertree in a plane and label each child from left to right. For each type $r$ hyperedge, we uniformly randomly pick $r$ vertices among $d-1$ vertices in the first generation to put spins $\tau(\rho)$, and the rest $d-1-r$ many vertices are assigned with spins $-\tau(\rho)$.

\item After defining the first generation, we keep constructing subsequent generations by induction. For each children $v$ with spin $\tau(v)$ in the previous generation, we generate $ \textnormal{Pois}\left(\frac{\alpha}{d-1}\right)$ many hyperedges that pairwise intersects at $v$ and assign a type  to each hyperedge by the same rule with $\tau(\rho)$ replaced by $\tau(v)$.  We call such random hypergraphs with spins  a \textit{multi-type Galton-Watson hypertree}, denoted by $(T,\rho,\tau)$ (see Figure \ref{GWTD}).
\end{itemize}   

	\begin{figure}
	\includegraphics[width=0.35 \linewidth]{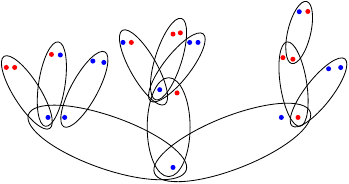}
\caption{A Galton-Watson hypertree with $d=3$. The vertices with spin $+$ are in blue and vertices with spin $-$ are in red.}\label{GWTD}
\end{figure}

Let  $W_t^{\pm}$ be the number of vertices with $\pm$ spins at the $t$-th generation and $W_t^{(r)}$ be the number of  hyperedges which contains exactly $r$ children with spin $+$  in the $t$-th generation. Let
$\mathcal G_{t-1}:=\sigma(W^{\pm}_k, 1\leq k\leq t-1)$ be the $\sigma$-algebra generated by $W_{k}^{\pm}, 1\leq k\leq t-1$.
From our definition,
$W_0^+=1,W_0^-=0$ and  $\{W_{t}^{(r)}\}_{0\leq r\leq d-1}$ are independent conditioned on $\mathcal G_{t-1}$, and the conditioned laws of $W_{t}^{(r)}$ are given by
\begin{align}
\mathcal L(W_t^{(d-1)}|\mathcal G_{t-1})&=\text{Pois}\left(\frac{a}{2^{d-1}}W_{t-1}^{+} + \frac{b}{2^{d-1}}W_{t-1}^{-}\right),\label{Pois1}\\
\mathcal L(W_t^{(0)}| \mathcal G_{t-1})&=\text{Pois}\left(\frac{a}{2^{d-1}}W_{t-1}^{-} +\frac{b}{2^{d-1}}W_{t-1}^{+}\right),\\
\mathcal L(W_t^{(r)}|\mathcal G_{t-1})&=\text{Pois}\left(\frac{b{{d-1}\choose{r}} }{2^{d-1}}(W_{t-1}^{-}+W_{t-1}^{+})\right), \quad  1\leq r\leq d-2.\label{Pois3}
\end{align}
We also have 
\begin{align}
W_t^+=\sum_{r=0}^{d-1} rW_{t}^{(r)}, \quad 
W_t^-=\sum_{r=0}^{d-1} (d-1-r)W_{t}^{(r)}.	\label{Pois5}
\end{align}

\begin{definition}\label{def:spin_preserving}
	A \textit{rooted hypergraph} is a hypergraph $H$ with a distinguished vertex $i\in V(H)$, denoted by $(H,i)$. We say two rooted hypergraphs $(H,i)$ and  $(H',i')$ are \textit{isomorphic} and  if and only if there is a bijection $\phi: V(H)\to V(H')$ such that $\phi(i)=i'$ and  $e\in E(H)$ if and only if $\phi(e):=\{\phi(j): j\in e\}\in E(H')$.

	Let $(H,i,\sigma)$ be a rooted hypergraph with root $i$ and each vertex $j$ is given a spin $\sigma(j)\in\{-1,+1\}$. Let $(H',i',\sigma')$ be a rooted hypergraph with root $i'$ where for each vertex $j\in V(H')$, a spin $\sigma'(j)\in \{-1,+1\}$ is given.  We say  $(H,i,\sigma)$ and $(H',i',\sigma')$ are \textit{spin-preserving isomorphic} and denoted by $(H,i,\sigma)\equiv (H',i',\sigma')$ if and only if there is an isomorphism $\phi: (H,i)\to (H',i')$ with $\sigma(v)=\sigma'(\phi(v))$ for each $v\in V(H)$. 
 \end{definition}

Let $(H,i,\sigma)_{t}, (T,\rho,\tau)_t$ be the rooted hypergraphs $(H,i,\sigma), (T,\rho,\tau)$  truncated at distance $t$ from $i, \rho$ respectively, and let $(T,\rho,-\tau)$ be the corresponding hypertree growth process where the root $\rho$ has spin $-1$.
We prove a local weak convergence of a typical neighborhood of a vertex in the hypergraph $H$ to the hypertree process $T$ we described above. In fact, we prove the following stronger statement. The proof of Theorem \ref{coupling2} is given in Section \ref{sec:coupling}.

\begin{theorem}\label{coupling2} Let $(H,\sigma)$ be a random hypergraph $H$ with spin $\sigma$ sampled from $\mathcal H_n$.
	Let $i\in [n]$ be fixed with spin $\sigma_i$. Let $l=c\log(n)$ with $c\log(\alpha)<1/4$, the following holds for sufficiently large $n$.
	\begin{enumerate}
	\item 	If $\sigma_i=+1$, there exists a coupling between $(H,i,\sigma)$ and $(T,\rho,\tau)$ such that $(H, i,\sigma)_l \equiv (T, \rho ,\tau)_l$ with probability at least $1-n^{-1/5}$.
	\item  If $\sigma_i=-1$, there exists a coupling between $(H,i,\sigma)$ and $(T,\rho,-\tau)$ such that $(H, i,\sigma)_l \equiv (T, \rho ,-\tau)_l$ with probability at least $1-n^{-1/5}$.
	\end{enumerate}
\end{theorem}

Now we construct two martingales from the Poisson hypertree growth process. Define two processes
	\begin{align*}
		M_t:&=\alpha^{-t}(W_t^+ + W_t^-),\quad 
		\Delta_t :=\beta^{-t}(W_t^+ - W_t^-).
	\end{align*}
	
\begin{lemma}\label{lem:martingale}
The two processes $\{M_t\}, \{\Delta_t\}$ are $\mathcal G_t$-martingales. If $\beta^2>\alpha>1$, $\{M_t\}$ and $\{\Delta_t\}$ are  uniformly integrable.
The martingale $\{\Delta_t\}$ converges almost surely and in $L^2$ to a unit mean random variable $\Delta_{\infty}$. Moreover, $\Delta_{\infty}$ has a finite variance  and 	\begin{align}\label{l2}
	\lim_{t\to\infty}\mathbb E|\Delta_t^2-\Delta_{\infty}^2|= 0.	
	\end{align}
\end{lemma}

  The following Lemma will be used in the proof of Theorem \ref{main} to analyze the correlation between the estimator we construct and the correct labels of vertices based on the random variable $\Delta_{\infty}$. The proof  is similar to the proof of  Theorem 4.2 in \cite{massoulie2014community}, and we include it in Appendix \ref{sec:thresh}.
\begin{lemma}\label{thresholding} Let $l=c\log n$ with $c\log\alpha<1/8$.
	For any $\epsilon>0$,
	\begin{align}\label{inprob}
		\lim_{n\to\infty}\mathbb P_{\mathcal H_n}\left(\left|\frac{1}{n}\sum_{i=1}^n \beta^{-2l}D_l^2(i)-\mathbb E[\Delta_{\infty}^2]\right|>\epsilon\right)=0.
	\end{align}

Let $y^{(n)}\in\mathbb R^n$ be a random sequence of $l_2$-normalized vectors defined by
\begin{align*}
	y_i^{(n)}:=\frac{D_l(i)}{\sqrt{\sum_{j=1}^n D_l(j)^2}}, 1\leq i\leq n.
\end{align*}
Let $x^{(n)}$ be any sequence of random vectors in $\mathbb R^n$ such that for any $\epsilon>0$,
$$\lim_{n\to\infty}\mathbb P_{\mathcal H_n}(\|x^{(n)}-y^{(n)}\|_2>\epsilon)=0.
$$
For all $\tau\in\mathbb R$ that is a point of continuity of the distribution of both $\Delta_{\infty}$ and $-\Delta_{\infty}$, for any $\epsilon>0$, one has the following 
\begin{align}\label{indicator}
&\lim_{n\to\infty}\mathbb P_{\mathcal H_n}\left(\left|\frac{1}{n}\sum_{i\in[n]:\sigma_i=+}\mathbf{1}\left\{x_i^{(n)}\geq\tau/\sqrt{n\mathbb E[\Delta_{\infty}^2]}\right\}-\frac{1}{2}\mathbb P(\Delta_{\infty}\geq \tau)\right|>\epsilon\right)=0,\\
&\lim_{n\to\infty}\mathbb P_{\mathcal H_n}\left(\left|\frac{1}{n}\sum_{i\in[n]:\sigma_i=-}\mathbf{1}\left\{x_i^{(n)}\geq\tau/\sqrt{n\mathbb E[\Delta_{\infty}^2]}\right\}-\frac{1}{2}\mathbb P(-\Delta_{\infty}\geq \tau)\right|>\epsilon\right)=0.\notag 
	\end{align}
\end{lemma}

\section{Proof of the main result}\label{sec:maintheorem}

Let $\vec{S}_l:=(S_l(1),\dots, S_l(n))$ and  $\vec{D}_l:=(D_l(1),\dots, D_l(n))$. We say the the sequence of  vectors $\{v_n\}_{\geq 1}$ is \textit{asymptotically aligned} with the sequence of  vectors $\{w_n\}_{n\geq 1}$ if 
	\[ \lim_{n\to\infty} \frac{|\langle v_n, w_n\rangle|}{\|v_n\|_2\cdot \|w_n\|_2}=1.
	\]
	
With all the ingredients in Sections \ref{sec:matrixexpan}-\ref{sec:coupling},  we establish the following weak Ramanujan property of $B^{(l)}$. The proof of Theorem \ref{rama3} is given in Section \ref{sec:rama3}.
\begin{theorem}\label{rama3}
	For $l=c\log(n)$ with $c\log(\alpha)<1/8$, asymptotically almost surely the two leading eigenvectors of $B^{(l)}$ are asymptotically aligned with vectors $\vec{S_l}, \vec{D}_l$, where the first  eigenvalue is  of order  $\Theta(\alpha^l)$ up to some logarithmic factor and the second eigenvalue is of order $\Omega(\beta^l)$. All other eigenvalues are of order $O(n^{\epsilon }\alpha^{l/2})$ for any $\epsilon>0$.
\end{theorem}

Theorem \ref{rama3} connects the  leading eigenvectors  of $B^{(l)}$ with the local structures of the random hypergraph $H$ and shows that the bulk eigenvalues of $B^{(l)}$ are separated from the two top eigenvalues.
Equipped with  Theorem \ref{rama3} and Lemma \ref{thresholding}, we are ready to prove our main result. 
\begin{proof}[Proof of Theorem \ref{main}]
	Let $x^{(n)}$ be the $l_2$-normalized second eigenvector of $B^{(l)}$, by Theorem \ref{rama3}, $x^{(n)}$ is asymptotically aligned with the $l_2$-normalized vector
$$y_i^{(n)}=\frac{D_l(i)}{\sqrt{\sum_{j=1}^n D_l(j)^2}}, 1\leq i\leq n$$ asymptotically almost surely.  So we have  $\|x^{(n)}-y^{(n)}\|_2\to 0$ or $\|x^{(n)}+y^{(n)}\|_2\to 0$ asymptotically almost surely. We first assume $\|x^{(n)}-y^{(n)}\|_2\to 0$. Since $\mathbb E\Delta_{\infty}=1$, from the proof of Theorem 2.1 in \cite{massoulie2014community}, there exists a point $\tau\in\mathbb R$, in the set  of  continuity points of both $\Delta_{\infty}$ and $-\Delta_{\infty}$, that satisfies  
$r:=\mathbb P(\Delta_{\infty}\geq \tau)-\mathbb P(-\Delta_{\infty}\geq \tau)>0.
$  Take $t=\tau/\sqrt{\mathbb E(\Delta_{\infty}^2)}$ and let $\mathcal N^+,\mathcal N^-$ be the set of vertices with spin $+$ and $-$, respectively.  From the definition of $\hat{\sigma}$, we have 
\begin{align}
\frac{1}{n}\sum_{i\in[n]}\sigma_i\hat{\sigma}_i
=&\frac{1}{n}\sum_{i\in[n]}\sigma_i\left(\mathbf{1}_{\left\{x_i^{(n)}\geq t/\sqrt{n}\right\}}-\mathbf{1}_{\left\{x_i^{(n)}< t/\sqrt{n}\right\}}\right) \label{final}\\
=&-\frac{1}{n}\sum_{i\in[n]}\sigma_i+\frac{2}{n}
\sum_{i\in\mathcal N^+}\mathbf{1}_{\left\{x_i^{(n)}\geq \tau/\sqrt{n\mathbb E\Delta_{\infty}^2}\right\}}-\frac{2}{n}\sum_{i\in\mathcal N^-}\mathbf{1}_{\left\{x_i^{(n)}\geq \tau/\sqrt{n\mathbb E\Delta_{\infty}^2}\right\}}.\notag 
\end{align}

Note that $\frac{1}{n}\sum_{i\in [n]}\sigma_i\to 0$ in probability by the law of large numbers. From  \eqref{indicator} in  Lemma \ref{thresholding}, we have \eqref{final} converges in probability to $\mathbb P(\Delta_{\infty}\geq \tau)-\mathbb P(-\Delta_{\infty}\geq \tau)=r$.
If $\|x^{(n)}+y^{(n)}\|_2\to 0$, similarly we have   
$\frac{1}{n}\sum_{i\in[n]}\sigma_i\hat{\sigma}_i $ converges to $-r$  in probability.
From these two cases, for any $\epsilon>0$,
 \begin{align*}
 	\lim_{n\to\infty}\mathbb P_{\mathcal H_n}\left(\{|ov_n(\hat{\sigma},\sigma)-r|>\epsilon\}\bigcap \{|ov_n(\hat{\sigma},\sigma)+r|>\epsilon\} \right)=0.
 \end{align*}
This concludes the proof of Theorem \ref{main}.
\end{proof}

\section{Proof of Theorem \ref{matrixexpansionthm}}\label{sec:matrixmomentmethod}

\subsection{Proof of \eqref{matrixexpansion} in Theorem \ref{matrixexpansionthm}}

 For ease of notation, we drop the index $n$ from $l_n$ in the proof, and it will be clear from the law $\mathcal H_n$. 
For any sequences of real numbers $\{a_t\}_{t=1}^l,\{b_t\}_{t=1}^l$, we have the following expansion identity for $l\geq 2$ (see for example, Equation (15) in \cite{massoulie2014community} and Equation (27) in \cite{bordenave2018nonbacktracking}):
\begin{align*}
\prod_{t=1}^l(a_t-b_t)=\prod_{t=1}^l a_t -\sum_{m=1}^{l} \left(\prod_{t=1}^{l-m}(a_t-b_t) \right)b_{l-m+1}\prod_{t=l-m+2}^l a_t.	\end{align*}
 Therefore  the following identity holds.
\begin{align*}
 \prod_{t=1}^l (A_{i_{t-1}i_t}^{e_{i_t}}-\overline{A}_{i_{t-1}i_t}^{e_{i_t}})
=&\prod_{t=1}^l A_{i_{t-1}i_{t}}^{e_{i_t}}-\sum_{m=1}^l\left(\prod_{t=1}^{l-m} (A_{i_{t-1}i_t}^{e_{i_t}}-\overline{A}_{i_{t-1}i_t}^{e_{i_t}})\right)\overline{A}_{i_{l-m}i_{l-m+1}}^{e_{i_{l-m+1}}}\prod_{t=l-m+2}^lA_{i_{t-1}i_t}^{e_{i_t}}.
\end{align*}

Summing over all $w\in \text{SAW}_{ij}$,   $\Delta_{ij}^{(l)}$ can be written as 
\begin{align}\label{expand}
B_{ij}^{(l)}-\sum_{m=1}^l\sum_{w\in \text{SAW}_{ij}}	\left(\prod_{t=1}^{l-m} (A_{i_{t-1}i_t}^{e_{i_t}}-\overline{A}_{i_{t-1}i_t}^{e_{i_t}})\right)\overline{A}_{i_{l-m}i_{l-m+1}}^{e_{i_{l-m+1}}}\prod_{t=l-m+2}^lA_{i_{t-1}i_t}^{e_{i_t}}.
\end{align}

Introduce the set $Q_{ij}^m$ of walks $w$ defined by concatenations of two self-avoiding walks $w_1, w_2$ such that $w_1$ is a self-avoiding walk of length $l-m$ from $i$ to some vertex $k$, and $w_2$ is a self-avoiding walk of length $m$ from $k$ to $j$ for all possible $1\leq m\leq l$ and $k\in [n]$. Then $\text{SAW}_{ij}\subset Q_{ij}^m$ for all $1\leq m\leq l$. Let $R_{ij}^m= Q_{ij}^m\setminus \text{SAW}_{ij}$. Define the matrix $\Gamma^{(l,m)}$ as 
\begin{align}\label{Gammaij}
\Gamma_{ij}^{(l,m)}:=	\sum_{w\in R_{ij}^m}	\prod_{t=1}^{l-m} (A_{i_{t-1}i_t}^{e_{i_t}}-\overline{A}_{i_{t-1}i_t}^{e_{i_t}})\overline{A}_{i_{l-m}i_{l-m+1}}^{e_{t_{l-m+1}}}\prod_{t=l-m+2}^lA_{i_{t-1}i_t}^{e_{i_t}}.
\end{align}
 From \eqref{expand}, $\Delta_{ij}^{(l)}$ can be expanded as 
 \begin{align*}
 &B_{ij}^{(l)}-\sum_{m=1}^l\sum_{w\in  Q_{ij}^m \setminus R_{ij}^m}	\left(\prod_{t=1}^{l-m} (A_{i_{t-1}i_t}^{e_{i_t}}-\overline{A}_{i_{t-1}i_t}^{e_{i_t}})\right)\overline{A}_{i_{l-m}i_{l-m+1}}^{e_{i_{l-m+1}}}\prod_{t=l-m+2}^lA_{i_{t-1}i_t}^{e_{i_t}}.
  \end{align*}
  It can be further written as 
$$ B_{ij}^{(l)}- \sum_{m=1}^l\sum_{w\in Q_{ij}^m}	\prod_{t=1}^{l-m} (A_{i_{t-1}i_t}^{e_{i_t}}-\overline{A}_{i_{t-1}i_t}^{e_{i_t}})\overline{A}_{i_{l-m}i_{l-m+1}}^{e_{i_{l-m+1}}}\prod_{t=l-m+2}^lA_{i_{t-1}i_t}^{e_{i_t}} +\sum_{m=1}^l \Gamma_{ij}^{(l,m)}.$$

From the definition of matrix multiplication, we have
\begin{align}
 &\sum_{w\in Q_{ij}^m}	\prod_{t=1}^{l-m} (A_{i_{t-1}i_t}^{e_{i_t}}-\overline{A}_{i_{t-1}i_t}^{e_{i_t}})\overline{A}_{i_{l-m}i_{l-m+1}}^{e_{i_{l-m+1}}}\prod_{t=l-m+2}^lA_{i_{t-1}i_t}^{e_{i_t}} \notag\\
=&\sum_{1\leq u,v\leq n}\Delta_{iu}^{(l-m)}\overline{A}_{uv} B^{(m-1)}_{vj}
=\left(\Delta^{(l-m)}\overline{A}B^{(m-1)}\right)_{ij}.\label{matrixij}
\end{align}

Combining the expansion of $\Delta_{ij}^{(l)}$ above and \eqref{matrixij}, we obtain 
\begin{align}
	\Delta_{ij}^{(l)}=&B_{ij}^{(l)}-\sum_{m=1}^l(\Delta^{(l-m)}\overline A B^{(m-1)})_{ij}+\sum_{m=1}^l \Gamma_{ij}^{(l,m)}\label{entrywise}.
\end{align}
Since \eqref{entrywise} is true for any $i,j\in [n]$, it implies \eqref{matrixexpansion}.

\subsection{Proof of  \eqref{specrho} in Theorem \ref{matrixexpansionthm}}

We first prove the following spectral norm bound on $\Delta^{(l)}$.

\begin{lemma}\label{expectboundmoment}
	For $l=O(\log n)$ and fixed $k$, we have
	\begin{align}
	\mathbb E_{\mathcal H_n}[\rho(\Delta^{(l)})^{2k}]=O( n\alpha^{kl}\log ^{6k}n).\label{rhodelta}
	\end{align}
\end{lemma}
 
\begin{proof} 
	Note that
	$\mathbb E_{\mathcal H_n}[\rho(\Delta^{(l)})^{2k}]\leq \mathbb E_{\mathcal H_n}[\textnormal{tr}(\Delta^{(l)})^{2k}]$. The estimation is based on a coding argument, and we modify the proof in \cite{massoulie2014community} to count circuits in hypergraphs.  
	Let $W_{2k,l}$ be the set of all circuits of length $2kl$ in the complete hypergraph $K_{n,d}$ which are concatenations of $2k$ many self-avoiding walks of length $l$. For any circuits $w\in W_{2k,l}$, we  denote it by  $w=(i_0,e_{i_1},i_1,\dots e_{i_{2kl}},i_{2kl})$, with $i_{2kl}=i_0$. From \eqref{Delta_ijl}, we have 
	\begin{align}\label{deltasum}
		\mathbb E_{\mathcal H_n}\left[\textnormal{tr}(\Delta^{(l)})^{2k}\right]&=\sum_{j_1,\dots, j_{2k}\in [n]}\mathbb E_{\mathcal H_n} \left[\Delta^{(l)}_{j_1j_2}\Delta^{(l)}_{j_2j_3}\cdots \Delta^{(l)}_{j_{2k}j_1}\right]=\sum_{w\in W_{2k,l}}\mathbb E_{\mathcal H_n}\left[\prod_{t=1}^{2kl} (A_{i_{t-1}i_t}^{e_{i_t}}-\overline{A}_{i_{t-1}i_t}^{e_{i_t}})\right].
		\end{align}
		  For each circuit, the weight it contributes to the sum is the product of $(A_{ij}^{e}-\overline{A_{ij}^{e}})$ over all the hyperedges $e$ traversed in the circuits. In order to have an upper bound on $\mathbb E_{\mathcal H_n}[\textnormal{tr}(\Delta^{(l)})^{2k}]$, we need to estimate how many such circuits are included in the sum and what are the weights they contribute.

		 We also write $w=(w_1,w_2,\dots w_{2k})$, where each $w_i$ is a self-avoiding walk of length $l$.  Let $v$ and $h$ be the number of distinct vertices and hyperedges  traversed by the circuit respectively. The idea is to bound the number of all possible circuits $w$ in \eqref{deltasum} with given $v$ and $h$, and then sum over all possible $(v,h)$ pairs.
		 
		 Fix $v$ and $h$, for any circuit $w$  we form a labeled multigraph  $G(w)$ with labeled vertices  $\{1,\dots,v\}$ and  labeled multiple edges $\{e_1,\dots, e_h\}$ by the following rules:
		 \begin{itemize}
          \item Label the vertices in $G(w)$ by the order they first appear in $w$, starting from $1$. For any pair vertices $i,j\in [v]$, we add an edge between $i,j$ in $G(w)$ whenever a  hyperedge appears  between the $i$-th  and $j$-th distinct vertices in the circuit $w$. $G(w)$ is a multigraph since it's possible that for some  $i,j$,  there exists two distinct hyperedges connecting  the $i$-th  and $j$-th distinct vertices in $w$, which corresponds to two distinct edges in $G(w)$ connecting $i,j$. 
         \item Label the edges in $G(w)$ by the order in which the corresponding hyperedge first appears in $w$  from $e_1$ to $e_h$. Note that that the number of edges in $G(w)$ is at least $h$, since distinct edges in $G(w)$ can get the same hyperedge labels.
         At the end we obtain a multigraph $G(w)=(V(w), E(w))$ with vertex set $\{1,\dots,v\}$ and edge set $E(w)$ with hyperedge labels in $\{e_1,\dots e_h\}$.	 
		 \end{itemize}
		
		It's crucial to see that the labeling of vertices and edges in $G(w)$ is in order, and it tells us how the circuit $w$ is traversed. Consider any edge in $G(w)$ such that its right endpoint (in the order of the traversal of $w$) is a new vertex that has not been traversed by $w$. We call it a \textit{tree edge}. Denote by $T(w)$ the tree spanned by those edges. It's clear for the construction that $T(w)$ includes all vertices in $G(w)$, so $T(w)$ is a spanning tree of $G(w)$. Since the labels of vertices and edges are given in $G(w)$, $T(w)$ is uniquely defined. See Figure \ref{Gw} for an example.

	For a given $w\in W_{2k,l}$ with distinct hyperedges $e_1,\dots, e_h$, define $\textnormal{end}(e_i)$ to be the set of vertices in $V(w)$ such that they are the endpoints of edges with label $e_i$ in $G(w)$. For example, consider a hyperedge $e_1=\{1,2,3,4\}$ such that $\{1,2\}, \{1,3\}$ are all the edges in $G(w)$ with labels $e_1$, then $\textnormal{end}(e_1)=\{1,2,3\}$.  We consider circuits $w$ in three different cases and estimate their contribution to \eqref{deltasum} separately.

\textbf{Case (1)}. We first consider $w\in W_{2k,l}$ such that
\begin{itemize}
    \item each hyperedge label in $\{e_i\}_{1\leq i\leq h}$ appears exactly once on the edges of $G(w)$; 
    \item vertices in $e_i\setminus \textnormal{end}(e_i)$ are all distinct for $1\leq i\leq h$, and they are not vertices with labels in $V(w)$.
\end{itemize} 
The first condition implies the number of edges in $G(w)$ is $h$. The second condition implies that there are exactly $(d-2)h+v$ many distinct vertices in $w$. We will break each self-avoiding walk $w_i$ into three types of successive sub-walks where each sub-walk is exactly one of the following 3 types, and we encode these sub-walks as follows.
 \begin{itemize}

 \item Type 1: hyperedges with corresponding edges  in $G(w)\setminus T(w)$. Given our position in the circuit $w$, we can encode a hyperedge of this type by its right-end vertex. Hyperedges of Type 1 breaks the walk $w_i$ into disjoint sub-walks, and we partition these sub-walks into Type 2 and 3 below. 
	\item  Type 2: sub-walks such that all their hyperedges  correspond to edges of $T(w)$ and have been traversed already  by $w_1,\dots, w_{i-1}$. Each sub-walk is a part of a self-avoiding walk, and it is a path contained in the tree $T(w)$. Given its initial and its end vertices, there will be exactly one such path in $T(w)$. Therefore these walks can be encoded by the end vertices. 
	\item Type 3: sub-walks such that their hyperedges correspond to edges of $T(w)$ and  they are being traversed for the first time. Given the initial vertex of a sub-walk of this type, since it is traversing new edges and knowing in what order the vertices are discovered, we can encode these walks by its length, and from the given length, we know at which vertex the sub-walk ends.
 \end{itemize}
 
We encode any Type 1, Type 2, or Type 3 sub-walk by $0$ if the sub-walk is empty.
Now we can decompose each $w_i$ into sequences characterizing by its sub-walks:
\begin{align}\label{triple}
(p_1,q_1,r_1),(p_2,q_2,r_2),\dots, (p_t,q_t,r_t).
\end{align} 
Here $r_1,\dots r_{t-1}$ are codes from sub-walks of Type 1. From the way we encode such hyperedges, we have $r_i\in \{1,\dots v\}$ for $1\leq i\leq t-1$. Type 2 and Type 3 sub-walks are encoded by $p_1,\dots ,p_t$ and $q_1,\dots, q_t$ respectively. Since Type 1 hyperedges break $w$ into disjoint pieces, we use $(p_t,q_t,r_t)$ to represent the last piece of the sub-walk and make $r_t=0$. Each $p_i$ represents the right-end vertex of the Type 2 sub-walk, and $p_i=0$ if it the sub-walk is  empty, hence $p_i\in \{0,\dots v\}$ for $1\leq i\leq t$. Each $q_i$ represents the length of Type 3 sub-walks, so $q_i\in\{0,\dots l\}$ for $1\leq i\leq t$. From the way we encode these sub-walks, there are at most $(v+1)^2(l+1)$ many possibilities for each triplet $(p_j,q_j,r_j)$.

\begin{figure}
			\includegraphics[width=0.4\linewidth]{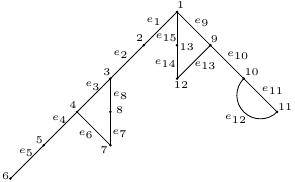}
			\caption{A multigraph $G(w)$ associated to a circuit $w=(w_1,\dots, w_4)$ of length $2kl$ with $k=2,l=5$. $w_1=(1,e_1,2,e_2,3,e_3,4,e_4,5,e_5,6), 
w_2=(6,e_5,5,e_4,4,e_6,7,e_7,8,e_8,3),  w_3=(3,e_2,2,e_1,1,e_9,9,e_{10},10,e_{11},11),w_4=(11,e_{12},10,e_{10},9,e_{13},12,e_{14},13,e_{15},1)$. Edges that are not included in $T(w)$ are $\{e_8, e_{12},e_{15}\}$.  The triplet sequences associated to  the 4 self-avoiding walks $\{w_i\}_{i=1}^4$ are given by $(0,6,0),  \quad (4,2,3), (0,0,0),\quad  (1,3,0)$, \quad  and $ (0,0,10),(9,2,1), (0,0,0)$, respectively.}\label{Gw}
		\end{figure}
		
We now consider how many ways we can concatenate sub-walks encoded by the triplets to form a circuit $w$.
All triples with $r_j\in [v]$ for $1\leq j\leq t-1$ indicate the traversal of an edge not in $T(w)$.  Since we know the number of  edges in $G(w)\setminus T(w)$  is $(h-v+1)$, and within a self-avoiding walk $w_i$, edges on $G(w)$ can be traversed at most once, the length of the triples in \eqref{triple} satisfies  $t-1 \leq h-v+1$, which implies $t\leq h-v+2$. Since each hyperedge can be traversed at most $2k$ many times by $w$ due to the constraint that the circuits $w$ of length $2kl$ are formed by self-avoiding walks, so the number of  triple sequences for fixed $v,h$ is at most
$
	[(v+1)^2(l+1)]^{2k(2+h-v)}.
$

There are multiple $w$ with the same code sequence.  However, they must all have the same  number of vertices and edges, and the positions where vertices and hyperedges are repeated must be the same. The number of ordered sequences of $v$ distinct vertices  is at most $n^v$. Given the vertex sequence, the number of ordered  sequences of $h$ distinct hyperedges in $K_{n,d}$ is at most ${n\choose d-2}^h$. This is because for a hyperedge $e$ between two vertices $i,j$, the number of possible hyperedges containing $i,j$ is at most ${n \choose d-2}$. Therefore, given $v,h$, the number of circuits that share the same triple sequence \eqref{triple} is  at most $n^v{n\choose d-2}^h$. 
% 	We have shown that the number of all possible triple sequences we can have from circuits in the sum \eqref{deltasum} is at most $[(v+1)^2(l+1)]^{2k(2+h-v)}$ and the number of circuits that correspond to the same triple sequence is at most $n^v{n\choose d-2}^h$.
	 Combining all the two estimates, the number of all possible circuits $w$ with fixed $v,h$ in Case (1) is at most 
\begin{align}\label{eq:case1choice}
    	n^v{n\choose d-2}^h[(v+1)^2(l+1)]^{2k(2+h-v)}.
\end{align}

Now we consider the expected weight of each circuit in the sum \eqref{deltasum}. Given $\sigma$, if $i,j\in e$, we have $A_{ij}^e\sim \text{Ber}\left(p_{\sigma(e)}\right)$, where $p_{\sigma(e)}=\frac{a}{{n\choose d-1}}$ if vertices in $e$ have the same $\pm$ spins and $p_{\sigma(e)}=\frac{b}{{n\choose d-1}}$ otherwise. For a given hyperedge appearing in $w$ with multiplicity $m\in\{1,\dots, 2k\}$, the corresponding expectation $\mathbb E_{\mathcal H_n}\left[(A_{ij}^{e}-\overline{A_{ij}^{e}})^m \right]$ is $0$ if $m=1$.  Since $0\leq A_{ij}^e\leq 1$, 
for $m\geq 2$, we have 
\begin{align}\label{sigma}
\mathbb E_{\mathcal H_n}\left[(A_{ij}^e-\overline{A_{ij}^e})^m \mid \sigma\right]\leq \mathbb E_{\mathcal H_n}\left[(A_{ij}^e-\overline{A_{ij}^e})^2 \mid \sigma\right]\leq p_{\sigma(e)}.	
\end{align}
For any hyperedge $e$ corresponding to an edge in $G(w) \setminus T(w)$ we have the upper bound 
\begin{align}\label{cyclebound}
p_{\sigma(e)}\leq \frac{a\vee b}{{{n}\choose{d-1}}}.
\end{align}
 Taking the expectation over $\sigma$   we have 
\begin{align}\label{expbound}
\mathbb E_{\sigma }[p_{\sigma(e)}]=\frac{a+(2^{d-1}-1)b}{2^{d-1}{n\choose d-1}}= \frac{\alpha}{(d-1){{n}\choose{d-1}}}.	
\end{align}
Recall the weight of each circuit in the sum \eqref{deltasum} is given by
$$ 
\mathbb E_{\mathcal H_n}\left[\prod_{t=1}^{2kl} (A_{i_{t-1}i_t}^{e_{i_t}}-\overline{A}_{i_{t-1}i_t}^{e_{i_t}})\right].$$ Conditioned on $\sigma$, $(A_{i_{t-1}i_t}^{e_{i_t}}-\overline{A}_{i_{t-1}i_t}^{e_{i_t}})$ are independent  random variables for distinct hyperedges. Denote these distinct hyperedges by $e_1,\dots e_h$ with multiplicity $m_1,\dots m_h$ and we temporarily order them such that $e_1,\dots e_{v-1}$ are the hyperedges corresponding to edges on $T(w)$. Introduce the random variables  $A^{e_i}\sim  \text{Ber}\left(p_{\sigma(e_i)}\right)$ for $1\leq i\leq h$ and denote $\overline{A^{e_i}}=\mathbb E_{\mathcal H_n}[A^{e_i}\mid \sigma]$. Therefore from \eqref{sigma} we have
\begin{align*}
	&\mathbb E_{\mathcal H_n}\left[\prod_{t=1}^{2kl} (A_{i_{t-1}i_t}^{e_{i_t}}-\overline{A}_{i_{t-1}i_t}^{e_{i_t}})\right]
	=\mathbb E_{\sigma}\left[\mathbb E_{\mathcal H_n}\left[\prod_{t=1}^{2kl} (A_{i_{t-1}i_t}^{e_{i_t}}-\overline{A}_{i_{t-1}i_t}^{e_{i_t}})\mid \sigma \right]\right]\\&=\mathbb E_{\sigma}\left[\mathbb \prod_{i=1}^h E_{\mathcal H_n}\left[(A^{e_i}-\overline{A^{e_i}})^{m_i}\mid \sigma\right] \right]
	\leq  \mathbb E_{\sigma}\left[ \prod_{i=1}^h p_{\sigma(e_i)}  \right]. 
\end{align*}	
We use the bound \eqref{cyclebound} for $p_{\sigma(e_{v})},\dots, p_{\sigma(e_h)}$, which implies
\begin{align}\label{cyclebound2}
		\mathbb E_{\sigma} \left[ \prod_{i=1}^h p_{\sigma(e_i)}  \right]&\leq \left(\frac{a\vee b}{{{n}\choose{d-1}}}\right)^{h-v+1}\mathbb E_{\sigma} \left[ \prod_{i=1}^{v-1} p_{\sigma(e_i)}  \right].
\end{align}

From the second condition for $w$ in  Case (1), any two hyperedges among $\{e_1,\dots e_{v-1}\}$ share at most $1$ vertex,  $p_{\sigma(e_i)},p_{\sigma(e_j)}$ are pairwise independent for all $1\leq i<j\leq v-1$. Moreover, since the corresponding edges of $e_1,\dots e_{v-1}$ forms the spanning tree $T(w)$, taking any $e_j$ such that the corresponding edge in $T(w)$ is attached to some leaf, we know $e_j$ and $\bigcup_{i\not=j, 1\leq i\leq v} e_i$ share exactly one common vertex, therefore $p_{\sigma(e_j)}$ is independent of $  \prod_{1\leq i\leq v-1, i\not=j} p_{\sigma(e_i)}$. We then have
\begin{align}\label{treerecursion}
\mathbb E_{\sigma} \left[ \prod_{i=1}^{v-1} p_{\sigma(e_i)}  \right]=\mathbb E_{\sigma} [p_{\sigma(e_j)}]\cdot \mathbb E_{\sigma} \left[ \prod_{1\leq i\leq v-1, i\not=j} p_{\sigma(e_i)}  \right].
\end{align}
Now the corresponding edges of all hyperedges $\{e_1,\dots e_{v-1}\}\setminus\{e_j\}$ form a tree in $G(w)$ again and the factorization of expectation in \eqref{treerecursion} can proceed as long as  we have some edge attached to leaves. Repeating \eqref{treerecursion} recursively, with \eqref{expbound}, we have 
\begin{align}\label{treebound}
\mathbb E_{\sigma} \left[ \prod_{i=1}^{v-1} p_{\sigma(e_i)}  \right]=\prod_{i=1}^{v-1} \mathbb E_{\sigma}[p_{\sigma(e_i)}]= \left(\frac{\alpha}{(d-1){{n}\choose{d-1}}}\right)^{v-1}. 
\end{align}
% Combining \eqref{cyclebound} and \eqref{treebound} we have for each circuit $w\in W_{2k,l}$ in the sum \eqref{Delta_ijl}, the expected weight it contributes to the sum is bounded by 
% $\left(\frac{a\vee b}{{{n}\choose{d-1}}}\right)^{h-v+1} \left(\frac{\alpha}{(d-1){{n}\choose{d-1}}}\right)^{v-1}. $

 Since every hyperedge in $w$ must be visited at least twice to make its expected weight non-zero, and $w$ is of length $2kl$, we must have $h\leq kl$. In the multigraph $G(w)$, we have the constraint $v\leq h+1\leq kl+1$. Since the first self-avoiding walk in $w$ of length $l$ takes $l+1$ distinct vertices, we also have $v\geq l+1$. So the possible range of $v$ is $l+1\leq v\leq kl+1$ and $h$ satisfies $v-1\leq h\leq kl$. 
 
 Putting all the estimates in \eqref{eq:case1choice}, \eqref{cyclebound2},  and \eqref{treebound} together, for fixed $v,h$, the total contribution of self-avoiding walks from $W_{2k,l}'$ to the sum is bounded by 
 \begin{align*}
 &n^v{{n}\choose{d-2}}^h[(v+1)^2(l+1)]^{2k(2+h-v)}\left(\frac{\alpha}{(d-1){{n}\choose{d-1}}}\right	)^{v-1}\left(\frac{a\vee b}{{{n}\choose{d-1}}}\right)^{h-v+1}.
 \end{align*}
 Denote $S_1$ to be the sum of all contributions from walks in Case (1). Therefore 
\begin{align} \label{expect}
	S_1\leq \sum_{v=l+1}^{kl+1}\sum_{h=v-1}^{kl}n^v\left(\frac{d-1}{n-d+2}\right)^h\left(\frac{\alpha}{d-1}\right	)^{v-1} [(v+1)^2(l+1)]^{2k(2+h-v)}(a\vee b)^{h-v+1}.
	\end{align}
When $l=O(\log n)$, Since $d,k$ are fixed, for sufficiently large $n$,  $\left(\frac{n}{n-d+2}\right)^h\leq 2.$ Then from \eqref{expect}, 
\begin{align*}
S_1	\leq &  \sum_{v=l+1}^{kl+1}\sum_{h=v-1}^{kl} 2 n^{v-h}(d-1)^{h-v+1}[(v+1)^2(l+1)]^{2k(2+h-v)}\alpha^{v-1}\left(a\vee b\right)^{h-v+1} \\
	\leq & 2\sum_{v=l+1}^{kl+1}\sum_{h=v-1}^{kl} n\left[\frac{(a\vee b)(d-1)}{n}\right]^{h-v+1} [(kl+2)^2(l+1)]^{2k(2+h-v)}\alpha^{v-1}. 
\end{align*}
Hence
\begin{align}
\frac{S_1 }{n\alpha^{kl}[(kl+2)^2(l+1)]^{2k}} 
 \leq  &2\sum_{v=l+1}^{kl+1}\alpha^{v-1-kl}\sum_{h=v-1}^{kl}\left[n^{-1}(a\vee b)(d-1)((kl+2)^2(l+1))^{2k}\right]^{h-v+1}.\label{ratio} 	 
\end{align}
Since for fixed $d,k$ and $l=O(\log n)$, $n^{-1}(a\vee b)(d-1)((kl+2)^2(l+1))^{2k}=o(1)$ for $n$ sufficiently large, the leading term in \eqref{ratio} is the term with $h=v-1$. For sufficiently large $n$,   we have 
\begin{align*}
\frac{S_1}{n\alpha^{kl}[(kl+2)^2(l+1)]^{2k}}\leq 3\sum_{v=l+1}^{kl+1}	\alpha^{v-1-kl}=3\cdot\frac{\alpha-\alpha^{(1-k)l}}{\alpha-1}\leq \frac{3\alpha}{\alpha-1}.\end{align*}
It implies that 
$    S_1=O(n \alpha^{kl}\log ^{6k}n). 
$

\textbf{Case (2)}. We now consider $w\in W_{2k,l}$ such that
\begin{itemize}
    \item the number of edges in $G(w)$ is greater than $h$.
    \item vertices in $e_i\setminus \textnormal{end}(e_i)$ are all distinct for $1\leq i\leq h$, and they are not vertices with labels in $V(w)$.
\end{itemize}

 Let $\tilde{h}$ be the number of edges in $G(w)$ with $\tilde{h}\geq h+1$.  Same as in Case (1), the number of triple sequence is at most $[(v+1)^2(l+1)]^{2k(2+\tilde h -v)}$.   Let $s_i, 1\leq i\leq h$ be the size of $\textnormal{end}(e_i)$.  By our definition of $s_i$, we have  $\sum_{i=1}^h (s_i-2)\geq \tilde{h}-h$. We first pick edges with distinct hyperedge labels and label their vertices, then label the remaining vertices from edges with repeated hyperedge labels, each has at most $(d-2)$ choices. The number of all possible circuits  in Case (2) with fixed $v, h, \tilde{h}$  is bounded by
 \begin{align*}
    &[(v+1)^2(l+1)]^{2k(2+\tilde h -v)} n^{v-(\tilde h-h)}(d-2)^{\tilde h-h} \binom{n}{d-s_1}\cdots \binom{n}{d-s_h}.
 \end{align*}
 For large $n$, the quantity above  is bounded by
 \begin{align*}
 2[(v+1)^2(l+1)]^{2k(2+\tilde h -v)}  n^{v-\tilde{h}+h}(d-2)^{\tilde h-h} \left( \frac{d-1}{n}\right)^{\tilde{h}} \binom{n}{d-1}^h.    
 \end{align*}

Now we consider the expected weight of each circuit in Case (2). In the spanning tree $T(w)$, we keep edges with distinct hyperedge labels that  appear first in the circuit $w$ and remove other edges.  This gives us a forest denoted  $F(w)$ inside $T(w)$, with at least $v-1-\tilde{h}+h$ many edges. We temporarily label those edges in the forest as $e_1,\dots, e_q$ with $q\geq v-1-\tilde{h}+h$. Then similar to the analysis of \eqref{treebound} in Case (1),  
 \begin{align*}
	&\mathbb E_{\mathcal H_n}\left[\prod_{t=1}^{2kl} (A_{i_{t-1}i_t}^{e_{i_t}}-\overline{A}_{i_{t-1}i_t}^{e_{i_t}})\right]
	\leq  \mathbb E_{\sigma}\left[ \prod_{i=1}^h p_{\sigma(e_i)}  \right]\leq \left(\frac{a\vee b}{{{n}\choose{d-1}}}\right)^{\tilde{h}-v+1} \left(\frac{\alpha}{(d-1){{n}\choose{d-1}}}\right)^{v-1-\tilde{h}+h}. 
\end{align*}
  Since every distinct hyperedge in $w$ must be visited at least twice, we must have $l\leq h\leq kl$. In the multigraph $G(w)$, we have the constraint $l+1\leq v\leq \tilde{h}+1$.  Therefore we have 
  \begin{align*}
      S_2
      \leq & 2\sum_{h=l}^{kl}\sum_{\tilde{h}=h+1}^{2kl}\sum_{v=l+1}^{\tilde{h}+1}[(v+1)^2(l+1)]^{2k(2+\tilde h -v)}n^{v-{\tilde h}+h} (d-2)^{\tilde{h}-h}\left( \frac{d-1}{n}\right)^{\tilde{h}} \binom{n}{d-1}^h \\
       &\cdot \left(\frac{a\vee b}{{{n}\choose{d-1}}}\right)^{\tilde{h}-v+1} \left(\frac{\alpha}{(d-1){{n}\choose{d-1}}}\right)^{v-1-\tilde{h}+h}
      =O(\alpha^{kl} \log^{6k}(n)).
  \end{align*}

%  (ii) Suppose there are vertices connected to edges with the same hyperedge label. In this case $\textnormal{end}(e_i)$ can be an odd number for some $i$.

 \textbf{Case (3)}.  We now consider $w\in W_{2k,l}$  not included in Cases (1) or Case (2), which satisfies that
 \begin{itemize}
     \item for some $i\not=j$, there are common vertices in $e_i\setminus \textnormal{end}(e_i)$ and $e_j\setminus \textnormal{end}(e_j)$;
     \item or there are  vertices  in $e_i\setminus \textnormal{end}(e_i)$ with labels in $V(w)$.
 \end{itemize}
 
 Let $v,h, \tilde{h}, s_i$ be defined in the same way as in Case (2).
 The number of triple sequence is at most $[(v+1)^2(l+1)]^{2k(2+\tilde h -v)}$. Consider the  forest $F(w)$ introduced in Case (2) as a subgraph of $T(w)$, which has at least $(v-1-\tilde{h}+h)$ many edges with distinct hyperedge labels. We temporarily denote the edges by $e_1,\dots, e_q$, and the ordering is chosen such that $e_1$ is adjacent to a leaf in $F(w)$, and each $e_i, i\leq 2$ is adjacent to a leaf in $F(w)\setminus \{e_1,\dots,e_{i-1}\}$.
 
 For $1\leq i\leq q$, we call $e_i$ a \textit{bad} hyperedge if $e_i\setminus \textnormal{end}(e_i)$ share a vertex with some $e_j\setminus \textnormal{end}(e_j)$ with $j>i$,  or there are  vertices  in $e_i\setminus \textnormal{end}(e_i)$ with labels in $V(w)$.  Suppose among them there are $t$ bad hyperedges with $0\leq t\leq v-1$. Let $\delta_i=1$ if $e_i$ is a bad hyperedge, and $\delta_i=0$ otherwise. Then the number of all possible circuits  in Case (3) with fixed $v, h, \tilde{h}$, and $t$,  is bounded by
 \begin{align*}
    &[(v+1)^2(l+1)]^{2k(2+\tilde h -v)} n^{v-\tilde h +h}(d-2)^{\tilde h-h} \binom{n}{d-s_1-\delta_1}\cdots \binom{n}{d-s_h-\delta_h}\\
 \leq &
    2[(v+1)^2(l+1)]^{2k(2+\tilde h -v)} n^{v-\tilde h+h}(d-2)^{\tilde h-h}\left( \frac{d-1}{n}\right)^{\tilde{h}+t} \binom{n}{d-1}^h. \notag
 \end{align*}
After removing the $t$ edges with bad hyperedge labels from the forest $F(w)$,  we can do the same analysis as in Case (2). The expected weight of each circuit in Case (3) with given $v,h, \tilde{h}, t$ now satisfies
 \begin{align*}
	&\mathbb E_{\mathcal H_n}\left[\prod_{t=1}^{2kl} (A_{i_{t-1}i_t}^{e_{i_t}}-\overline{A}_{i_{t-1}i_t}^{e_{i_t}})\right]
	\leq \left(\frac{a\vee b}{{{n}\choose{d-1}}}\right)^{\tilde{h}-v+1+t}\left(\frac{\alpha}{(d-1){{n}\choose{d-1}}}\right)^{v-1-\tilde{h}+h-t}.
\end{align*}
  Let  $S_3$ be the total contribution from circuits in Case (3) to \eqref{deltasum} . Then 
  \begin{align*}
S_3\leq & \sum_{h=l}^{kl}\sum_{\tilde{h}=h}^{2kl}\sum_{v=l+1}^{\tilde{h}+1}\sum_{t=0}^{v-1}2[(v+1)^2(l+1)]^{2k(2+\tilde h -v)} n^{v-\tilde h +h}(d-2)^{\tilde h-h}\left( \frac{d-1}{n}\right)^{\tilde{h}+t} \binom{n}{d-1}^h\\
 &\cdot \left(\frac{a\vee b}{{{n}\choose{d-1}}}\right)^{\tilde{h}-v+1+t} \left(\frac{\alpha}{(d-1){{n}\choose{d-1}}}\right)^{v-1-\tilde{h}+h-t}= O(n\alpha^{kl} \log^{6k}n).
  \end{align*}

From the estimates on $S_1,S_2$ and $S_3$, Lemma \ref{expectboundmoment} holds.
\end{proof}

With Lemma \ref{expectboundmoment}, we are able to derive \eqref{specrho}. 
 For any fixed $\epsilon>0$, choose $k$ such that  $1-2k\epsilon<0$, using Markov inequality, we have
\begin{align*}
\mathbb P_{\mathcal H_n}(\rho(\Delta^{(l)})\geq n^{\epsilon}\alpha^{l/2})\leq \frac{\mathbb E_{\mathcal H_n}(\rho(\Delta^{(l)})^{2k})}{n^{2k\epsilon} \alpha^{kl}}= O(n^{1-2k\epsilon}\log^{6k}n).
\end{align*} 
This implies \eqref{specrho} in the statement of Theorem \ref{matrixexpansionthm}.  
 
 \subsection{Proof of  \eqref{specrho2} in Theorem \ref{matrixexpansionthm}}
 
 Using a similar argument as in the proof of Lemma \ref{expectboundmoment}, we can prove the following estimate of $\rho(\Gamma^{(l,m)})$. The proof is given in Appendix \ref{sec:appendix_moment}.
 
\begin{lemma}\label{lemma:bound_on_gamma}
For $l=O(\log n)$, fixed $k$, and any $1\leq m\leq l$,  there exists a constant $C>0$ such that 
\begin{align}
    \mathbb E_{\mathcal H_n}[\rho(\Gamma^{(l,m)})^{2k}]\leq C n^{1-2k} \alpha^{k(l+m-2)}\log ^{14k}n.
\end{align}
\end{lemma}

With Lemma \ref{lemma:bound_on_gamma}, we can apply  the union bound and Markov inequality. For any $\epsilon>0$, choose $k>0$ such that $1-2k\epsilon<0$, we have 
\begin{align*} &\mathbb P_{\mathcal H_n}\left(\bigcup_{m=1}^l\left\{\rho(\Gamma^{(l,m)})\geq 
n^{\epsilon-1}\alpha^{(l+m)/2}\right\}\right) 
\leq \sum_{m=1}^l \mathbb P_{\mathcal H_n}\left(\rho(\Gamma^{(l,m)})\geq 
n^{\epsilon-1}\alpha^{(l+m)/2}\right) \\
\leq &\sum_{m=1}^l \frac{\mathbb E_{\mathcal H_n}\rho(\Gamma^{(l,m)})^{2k}}{n^{2k(\epsilon-1)}\alpha^{k(l+m)}}
\leq  \sum_{m=1}^l \frac{C\log^{14k} (n)\cdot n^{1-2k}\alpha^{k(l+m-2)}}{n^{2k(\epsilon-1)}\alpha^{k(l+m)}}= O\left( (\log^{14k}(n)\cdot  n^{1-2k\epsilon}\alpha^{-2k}\right).
\end{align*}
Since $1-2k\epsilon<0$,  this proves 
\eqref{specrho2} in Theorem \ref{matrixexpansionthm}.

\section{Proof of Theorem \ref{quasi}}\label{sec:thm43}

 Let $n^{\pm}$ be the number of vertices with spin $\pm$ respectively. Consider the event
 \begin{align}\label{OmegaEvent}
 \tilde{\Omega}:=\{ |n^{\pm}-\frac{n}{2}|\leq \log(n)\sqrt{n}\}.\end{align}
By Hoeffding's inequality, 
\begin{align}\label{hoffd}
\mathbb P_{\sigma}\left(|n^{\pm}-\frac{n}{2}|\geq \log(n)\sqrt{n}\right)\leq 2\exp(-2\log^2(n)), 	
\end{align} which implies $\mathbb P_{\sigma}(\tilde{\Omega})\geq 1-2\exp(-2\log^2 (n)).$
In the rest of this section we will condition on the event $\tilde{\Omega}$, which will not effect our conclusion and probability bounds, since for any event $A$, if $\mathbb P_{\mathcal H_n}(A\mid \tilde{\Omega})=1-O(n^{-\gamma})$ for some $\gamma>0$, we have 
\begin{align*}
\mathbb P_{\mathcal H_n}(A)=&\mathbb P_{\mathcal H_n}(A\mid \tilde{\Omega})\mathbb P_{\mathcal H_n}(\tilde{\Omega})+\mathbb P_{\mathcal H_n}(A\mid \tilde{\Omega}^c)\mathbb P_{\mathcal H_n}(\tilde{\Omega}^c)=1-O(n^{-\gamma}).
\end{align*}

The following identity from Equation (38) in \cite{massoulie2014community} will be helpful in the proof. 
\begin{lemma}
	For any nonnegative integers $i,j,n$ and nonnegative numbers $a,b$ such that $a/n,  b/n<1$, we have
\begin{align}\label{basic}
	\frac{ai+bj}{n}-\frac{1}{2}\left(\frac{ai+bj}{n}\right)^2\leq 1-(1-a/n)^i(1-b/n)^j\leq \frac{ai+bj}{n}.\end{align}
\end{lemma}

We will also use the following version of Chernoff bound (see \cite{boucheron2013concentration}):
\begin{lemma} \label{ChernoffBound} Let $X$ be a sum of independent random variables taking values in $\{0,1\}$. Let $\mu=\mathbb E[X]$. Then for any $\delta>0$,
we have 
\begin{align}
\mathbb P(X\geq (1+\delta)\mu)&\leq \exp (-\mu h(1+\delta)), \label{onesidedbound}\\
\mathbb P(|X-\mu|\leq \delta \mu) &\geq 1-2\exp (-\mu \tilde{h}(\delta)),\label{twosidedbound}
\end{align}
where
 \begin{align*}
        h(x):&=x\log(x)-x+1, \quad 
\tilde{h}(x):=\min\{(1+x)\log (1+x)-x, (1-x)\log(1-x)+x\}.
\end{align*}
\end{lemma}
For any $t\geq 0$, the number of vertices with spin $\pm$ at distance $t$ (respectively $\leq $) of vertices $i$ is denoted $U_t^{\pm}(i)$ (respectively, $U_{\leq t}^{\pm}(i)$) and we know $S_t(i)=U_t^+(i) + U_t^-(i)$. We will omit index $i$ when considering quantities related to a fixed vertex $i$. Let $n^{\pm}$ be the number of vertices with spin $\pm$ and $\mathcal N^{\pm}$ be the set of vertices with spin $\pm$. For a fixed vertex $i$. Let 
\begin{align}\label{Ftsigma}
 \mathcal F_{t}:=\sigma(U_k^+, U_k^-, k\leq t ,\sigma_i, 1\leq i\leq n)\end{align}
 be the $\sigma$-algebra generated by $\{U_k^+, U_k^-, 0\leq k\leq t\}$ and $\{\sigma_i, 1\leq i\leq n\}$.  In the remainder of
the section we condition on the spins $\sigma_i$ of all $i\in [n]$ and assume $\tilde{\Omega}$ holds. We denote $\mathbb P(\cdot ):=\mathbb P_{\mathcal H_n}(\cdot \mid \tilde{\Omega})$. 

A main difficulty to analyze $U_t^+, U_t^-$ compared to the graph SBM in \cite{massoulie2014community} is that  $U_k^{\pm}$ are no longer independent conditioned on $\mathcal F_{k-1}$. Instead, we can only approximate  $U_k^{\pm}$ by counting subsets connected to $V_{k-1}$. To make it more precise, we have the following definition for connected-subsets.

\begin{definition}\label{S}
A \textit{connected $s$-subset}  in $V_{k}$ for $1\leq s\leq d-1$ is a subset of size $s$ which is contained in some hyperedge $e$ in $H$  and the rest $d-s$ vertices in  $e$ are from $V_{k-1}$ (see Figure \ref{fig:Q} for an example).
	Define $U_{k,s}^{(r)}, 0\leq r\leq s$ to be the number of \textit{connected $s$-subsets} in $V_k$ where exactly $r$ many vertices  have  $+$ spins.  For convenience, we write $U_{k}^{(r)}:=U_{k,d-1}^{(r)}$ for $0\leq r\leq d-1$.  Let $U_{k,s}=\sum_{r=0}^{s}U_{k,s}^{(r)}$ be the number of all connected $s$-subsets in $V_k$.	\end{definition}

\begin{figure}
	\includegraphics[width=0.2\linewidth]{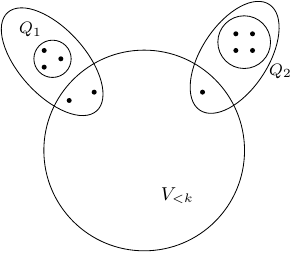}
	\caption{$d=5$, $Q_1$ is a connected $3$-subsets in $V_k$ and $Q_2$ is a connected $4$-subsets in $V_k$.}\label{fig:Q}
\end{figure}

We will show that $\sum_{r=0}^{d-1}rU_k^{(r)}$ is a good approximation of $U_k^+$ and $\sum_{r=0}^{d-1}(d-1-r)U_{k}^{(r)}$ is a good approximation of $U_k^-$, then the concentration of $U_k^{(r)}, 0\leq r\leq d-1$ implies the concentration of $U_k^{\pm}$. 

Since each hyperedge appears independently, conditioned on $\mathcal F_{k-1}$, we know $\{U_{k}^{(r)}, 0\leq r\leq d-1\}$ are independent binomial random variables. 
For $U_k^{(d-1)}$, the number of all possible connected $(d-1)$-subsets with $d-1$ many $+$ signs is  $ {{n^+-U_{\leq k-1}^+}\choose{d-1}}$, and each such subset is included in the hypergraph if and only if it forms a hyperedge with any vertex in $V_{k-1}$. Therefore each such subset is included independently with probability $$ 1-\left(1-\frac{a}{{{n}\choose{d-1}}}\right)^{U_{k-1}^+}\left(1-\frac{b}{{{n}\choose{d-1}}}\right)^{U_{k-1}^-}.$$ 
Similarly, we have the following distributions for $U_{k}^{(r)}, 1\leq r\leq d-1$:

\begin{align}
U_{k}^{(d-1)}&\sim \textnormal{Bin}\left({{n^+-U_{\leq k-1}^+}\choose{d-1}}, 1-\left(1-\frac{a}{{{n}\choose{d-1}}}\right)^{U_{k-1}^+}\left(1-\frac{b}{{{n}\choose{d-1}}}\right)^{U_{k-1}^-}\right)\label{mean1},\\
U_{k}^{(0)}&\sim \textnormal{Bin}\left({{n^--U_{\leq {k-1}}^-}\choose{d-1}}, 1-\left(1-\frac{a}{{{n}\choose{d-1}}}\right)^{U_{k-1}^-}\left(1-\frac{b}{{{n}\choose{d-1}}}\right)^{U_{k-1}^+}\right),\label{mean2}
\end{align}
and for $1\leq r\leq d-2$,
\begin{align}
U_{k}^{(r)}&\sim \textnormal{Bin}\left({{n^+-U_{\leq {k-1}}^+}\choose{r}}{{n^--U_{\leq {k-1}}^-}\choose{d-1-r}}, 1-\left(1-\frac{b}{{{n}\choose{d-1}}}\right)^{S_{k-1}}\right) .\label{mean3}
\end{align}

For two random variable $X, Y$, we denote $X \preceq Y$ if $X$ is stochastically dominant by $Y$, i.e., $\mathbb P(X\leq x)\geq \mathbb P(Y\leq x)$ for any $x\in \mathbb R$. 
 We denote 
 $ U_{k}^{*}:=\sum_{s=1}^{d-2}U_{k,s}$  to be the number of all connected $s$-subsets  in $V_k$ for $1\leq s\leq d-2$. 
 
 For each $1\leq s\leq d-2$, conditioned on $\mathcal F_{k-1}$, the number of possible $s$-subsets is at most $n\choose s$, and each subset is included in the hypergraph independently with probability at most $\left( \frac{a\vee b}{{{n}\choose{d-1}}}{{S_{k-1}}\choose{d-s}}\right)\wedge 1$, so we have
 \begin{align}\label{UKR}
 U_{k,s}&\preceq  \textnormal{Bin}\left({{n}\choose{s}},\frac{a\vee b}{{{n}\choose{d-1}}}{{S_{k-1}}\choose{d-s}}\wedge 1\right).	
 \end{align} 
 
 With the definitions above, we have the following inequality for $U_{k}^{\pm}$ by counting the number of $\pm$ signs from each type of subsets: 
\begin{align}
U_{k}^+ &\leq \sum_{r=0}^{d-1}rU_{k}^{(r)}+(d-2)U_{k}^{*}, \label{upperUT}\\
 U_{k}^- &\leq \sum_{r=0}^{d-1}(d-1-r)U_{k}^{(r)}+(d-2)U_{k}^{*}.	
\end{align}

To obtain the upper bound of $U_{k}^{\pm}$, we will show that $U_k^*$ is negligible compared to the number of $\pm$ signs from $U_{k}^{(r)}$. Since $U_{k}^{(r)}, 1\leq r\leq d-1$ are independent binomial random variables, we can  prove concentration results of these random variables. For the  lower bound of $U_{k}^{\pm}$, we need to show  that only a negligible portion of  $(d-1)$ connected subsets are overlapped, therefore  $U_{k}^+$ is lower bounded by $\sum_{r=0}^{d-1}rU_{k}^{(r)}$ minus some small term, and we can do it similarly for $U_{k}^-$. We will extensively use Chernoff bounds in Lemma \ref{ChernoffBound} to prove the concentration of $U_{k}^{\pm}$ in the following theorem.

\begin{theorem}\label{growthbound}
	Let $\epsilon\in (0,1)$,  and $l=c\log(n)$ with $c\log(\alpha)<1/4$. For any $\gamma \in (0,3/8)$,  there exists some constant $K>0$ and  such that the following holds with probability at least $1-O(n^{-\gamma})$ for all $i\in [n]$.
	\begin{enumerate}
		\item Let $T:=\inf\{t\leq l: S_t\geq K\log n\}$, then  $S_T=\Theta(\log n)$. \label{first}
		\item Let $  \epsilon_t:=\epsilon \alpha^{-(t-T)/2}$ for some $\epsilon>0$ and 
		\begin{align}\label{M}
		M:=\frac{1}{2}\begin{bmatrix}
			\alpha+\beta &\alpha-\beta\\
			\alpha-\beta &\alpha+\beta
		\end{bmatrix}.
		\end{align}
		Then for all $t,t'\in \{T,\dots l\}$, $t>t'$, the vector $\vec{U}_t:=(U_t^+,U_t^-)^{\top}$ satisfies the coordinate-wise bounds:
		\begin{align}\label{concentratecor1}
		U_t^+ &\in \left[\prod_{s=t'}^{t-1} (1-\epsilon_s), \prod_{s=t'}^{t-1} (1+\epsilon_s)\right](M^{t-t'}\vec{U}_{t'})_1,\\
		U_t^- &\in \left[\prod_{s=t'}^{t-1} (1-\epsilon_s), \prod_{s=t'}^{t-1} (1+\epsilon_s)\right](M^{t-t'}\vec{U}_{t'})_2,\label{concentratecor2}
		\end{align}
	where $(M^{t-t'}\vec{U}_{t'})_j$ is the $j$-th coordinate of the vector $M^{t-t'}\vec{U}_{t'}$ for $j=1,2$.
	\end{enumerate}
\end{theorem}

\begin{proof}
In this proof, all constants $C_i$'s, $C, C'$ are distinct for different inequalities unless stated otherwise. By the definition of $T$, $S_{T-1}\leq K\log(n)$. Let $Z_T$ be the number of all hyperedges in $H$ that are incident to at least one vertices in $V_{T-1}$. We have
$
	S_T\leq (d-1)Z_T,
$
and since the number of all possible hyperedges including a vertex in $V_{T-1}$ is at most $ S_{T-1}{n\choose d-1}$, $Z_T$ is stochastically dominated by 
\begin{align*}
	\textnormal{Bin}\left(K\log(n){n\choose d-1}, \frac{a\vee b}{{n\choose d-1}}\right),
\end{align*}
which has mean $(a\vee b) K\log(n)$. Let $K_1=(a\vee b)K$. By \eqref{onesidedbound} in Lemma \ref{ChernoffBound}, we have  for any constant $K_2>0$,
\begin{align}
\mathbb P(Z_T\geq K_2\log(n)|\mathcal F_{T-1})\leq \exp(-K_1\log(n)h(K_2/K_1))
\end{align}

Taking $K_2>K_1$ large enough such that $K_1h(K_2/K_1)\geq 2+\gamma$, we then have
\begin{align}
\mathbb P(Z_T\geq K_2\log(n)|\mathcal F_{T-1})\leq n^{-2-\gamma}.\end{align}
So with probability at least $1-n^{-2-\gamma}$, for a fixed $i\in[n]$, $S_T\leq K_3\log(n)$ with $K_3=(d-2)K_2$. Taking a union bound over $i\in[n]$, part \eqref{first} in Lemma \ref{growthbound} holds. 
We continue to prove \eqref{concentratecor1} and \eqref{concentratecor2} in several steps.

\textbf{Step 1: base case.} For the first step, we prove \eqref{concentratecor1} and \eqref{concentratecor2} for $t=T+1,t'=T$, which is 
\begin{align}\label{eq:818}
	U_{T+1}^{\pm}\in [1-\epsilon,1+\epsilon]\left(\frac{\alpha+\beta }{2}U_{T}^{\pm }+\frac{\alpha-\beta}{2}U_{T}^{\pm}\right).
\end{align}
 This involves a two-sided estimate of $U_{T+1}^{\pm}$.  The idea is to show the expectation of $U_{T+1}^{\pm}$ conditioned on $\mathcal F_{T}$ is closed to $  \frac{\alpha+\beta }{2}U_{T}^{\pm }+\frac{\alpha-\beta}{2}U_{T}^{\pm}$, and  $U_{T+1}^{\pm}$ is concentrated around its mean.
    
 \textbf{(i) Upper bound.} Define the event $\mathcal A_T:=\{S_T\leq K_3\log n\}$. We have just shown for a fixed $i$, 
\begin{align}\label{eq:PAT}
\mathbb P(\mathcal A_T)\geq 1-n^{-2-\gamma}.	
\end{align}
 Recall $|n^{\pm}-n/2|\leq \sqrt{n}\log n$   and conditioned on $\mathcal A_T$, for some constant $C>0$,
 $$U_{\leq T}^+\leq \sum_{t=0}^T S_t\leq 1+TK_3\log n\leq 1+lK_3\log n\leq CK_3\log^2 n.$$
 
  Conditioned on $\mathcal F_T$ and $\mathcal A_T$, for sufficiently large $n$, there exists constants $C_1>0$ such that $${{n^+-U_{\leq T}^+}\choose{d-1}}\geq C_1 {\frac{n}{2}\choose{d-1}}.$$ 
  From  inequality \eqref{basic}, there exists constant $C_2>0$ such that
  \begin{align*} 1-\left(1-\frac{a}{{{n}\choose{d-1}}}\right)^{U_T^+}\left(1-\frac{b}{{{n}\choose{d-1}}}\right)^{U_T^-}&\geq \frac{aU_T^ + + b U_T^-}{{{n}\choose{d-1}}}-\frac{1}{2}\left(\frac{aU_T^ + + b U_T^-}{{{n}\choose{d-1}}}\right)^2 \\
  & \geq \frac{C_2(aU_T^ + + b U_T^-)}{{{n}\choose{d-1}}}\geq \frac{C_2(a\wedge b)K\log n}{{n \choose d-1}}.
  \end{align*}  
  Then from \eqref{mean1},   for some constant $C_3>0$, 
\begin{align*}
\mathbb E[U_{T+1}^{(d-1)} \mid \mathcal F_T,\mathcal A_T]=&{{n^+-U_{\leq T}^+}\choose{d-1}}\left(1-\left(1-\frac{a}{{{n}\choose{d-1}}}\right)^{U_T^+}\left(1-\frac{b}{{{n}\choose{d-1}}}\right)^{U_T^-}\right)\\
\geq & C_1  {\frac{n}{2}\choose{d-1}}\cdot\frac{C_2(a\wedge b)K\log n}{{n \choose d-1}}
\geq  C_3K\log n.
\end{align*}
We can choose $K$ large enough such that
$C_3K\tilde{h}(\epsilon/(2d))\geq 2+\gamma,
$
then from \eqref{twosidedbound} in Lemma \ref{ChernoffBound}, for any given $\epsilon>0$ and $\gamma \in (0,1)$, 
\begin{align*}
&\mathbb P\left(	|U_{T+1}^{(d-1)}-\mathbb E[U_{T+1}^{(d-1)}|\mathcal F_{T}]|\leq \frac{\epsilon}{2d}\mathbb E[U_{T+1}^{(d-1)}|\mathcal F_{T}]  \big| \mathcal F_{T}\right)	\\
\geq &\mathbb P\left(	|U_{T+1}^{(d-1)}-\mathbb E[U_{T+1}^{(d-1)}|\mathcal F_{T}]|\leq \frac{\epsilon}{2d}\mathbb E[U_{T+1}^{(d-1)}|\mathcal F_{T}]  \big| \mathcal F_{T},\mathcal A_T\right)\mathbb P(\mathcal A_T)\\
\geq & \left[ 1-\exp\left(-\mathbb E[U_{T+1}^{(d-1)}|\mathcal F_{T},\mathcal A_T]\tilde{h}(\epsilon/2d)\right)\right] (1-n^{-2-\gamma})
\geq  (1-n^{-2-\gamma})^2\geq 1-2n^{-2-\gamma}.
\end{align*}
From the symmetry of $\pm$ labels, the concentration of $U_{T+1}^{(0)}$ works in the same way. 
Similarly, there exists a constant $C_1>0$ such that $\mathbb E[U_{T+1}^{(r)}\mid \mathcal F_T], 1\leq r\leq d-2$: 
\begin{align*}
\mathbb E[U_{T+1}^{(r)}\mid \mathcal F_T]=&{{n^+-U_{\leq T}^+}\choose{r}}{{n^--U_{\leq T}^-}\choose{d-1-r}}\left( 1-\left(1-\frac{b}{{{n}\choose{d-1}}}\right)^{S_T}\right)
\geq  C_1K\log n.
\end{align*} 
We can choose $K$ large enough such that for all $0\leq r\leq d-1$,
$$\mathbb P\left(	\left|U_{T+1}^{(r)}-\mathbb E[U_{T+1}^{(r)} \mid \mathcal F_{T}] \right|\leq \frac{\epsilon}{2d}\mathbb E[U_{T+1}^{(r)}|\mathcal F_{T}]  \mid \mathcal F_{T}\right)\geq 1-2n^{-2-\gamma}.
$$  

Next, we estimate $ U_{T+1}^*=\sum_{s=1}^{d-2} U_{T+1,s}$. Recall from \eqref{UKR}, we have $
	U_{T+1,s}\preceq Z_{T+1,s}$ where  $$Z_{T+1,s}\sim \textnormal{Bin}\left({{n}\choose{s}},\frac{a\vee b}{{{n}\choose{d-1}}}{{S_{T}}\choose{d-s}}\right).$$
Conditioned on $\mathcal A_T$ we know $K\log n\leq S_T\leq K_3\log n$, and
\begin{align*}
 \mathbb E[Z_{T+1,s}\mid \mathcal A_T,\mathcal F_T]={n\choose s}	\frac{a\vee b}{{{n}\choose{d-1}}}{{S_{T}}\choose{d-s}}\leq C_2\log^{d-s}(n)n^{1+s-d}
\end{align*} 
for some constant $C_2>0$. 
Using the fact that $  h(x)\geq \frac{1}{2}x\log(x)$ for $x$ large enough, from \eqref{onesidedbound}, we have for any constant $\lambda>0$,  $1\leq s\leq d-2$, there exists a constant $C_3>0$ such that for large $n$,
  \begin{align}
&\mathbb P(U_{T+1,s}\geq \lambda  S_T |\mathcal F_T, \mathcal A_T)\leq   \mathbb P(Z_{T+1,s}\geq \lambda  S_T|\mathcal F_T, \mathcal A_T) \notag\\
\leq &\exp\left(-\mathbb E[Z_{T+1,s}\mid \mathcal A_T,\mathcal F_T] h\left(\frac{ \lambda S_T}{\mathbb E[Z_{T+1,s}\mid \mathcal A_T,\mathcal F_T]}\right)\right) \notag\\
\leq &\exp\left(-\frac{1}{2}\lambda S_T\log\left(\frac{\lambda S_T}{\mathbb E[Z_{T+1,s}\mid \mathcal A_T,\mathcal F_T]}\right)\right)
\leq \exp(-\lambda C_3\log^2 n)\leq n^{-2-\gamma}.\label{chernoffonesided}\end{align}
Therefore with \eqref{eq:PAT} and \eqref{chernoffonesided},
\begin{align*}
\mathbb P(U_{T+1,s} < \lambda  S_T |\mathcal F_T)&\geq \mathbb P(U_{T+1,s}<  
\lambda  S_T |\mathcal F_T, \mathcal A_T)	\mathbb P(\mathcal A_T)\geq (1-n^{-2-\gamma})^2\geq 1-2n^{-2-\gamma}.
\end{align*}
Taking $ \lambda=\frac{(\alpha-\beta)\epsilon}{4d^2}$, we have $ U_{T+1,s}\leq \frac{(\alpha-\beta)\epsilon}{4d^2} S_T$ with probability at least $1-2n^{-2-\gamma}$ for any $\gamma\in (0,1)$. 

Taking a union bound over $2\leq r\leq d-1$, it implies 
\begin{align}\label{UTstar}
	 U_{T+1}^*\leq \frac{(\alpha-\beta)\epsilon}{4d}S_T
\end{align} with probability $1-O(n^{-2-\gamma})$ for any $\gamma\in(0,1)$.

Note that $  n^{\pm}=\frac{n}{2}+O(\sqrt{n}\log n)$ and  $U_{\leq T}^{\pm}=\sum_{k=1}^T S_k=O(\log^2( n) )$. From   \eqref{basic},
\begin{align*}
	\left(1-\frac{aU_T^+ + bU_T^-}{2{n\choose d-1}}\right)\frac{aU_T^+ + bU_T^-}{{n\choose d-1}}&\leq 1-\left(1-\frac{a}{{{n}\choose{d-1}}}\right)^{U_T^+}\left(1-\frac{b}{{{n}\choose{d-1}}}\right)^{U_T^-}\leq \frac{aU_T^+ + bU_T^-}{{n\choose d-1}}. 
\end{align*}
It implies that
\begin{align}
	\mathbb E[U_{T+1}^{(d-1)}|\mathcal F_T,\mathcal A_T]&={{\frac{n}{2}+O(\sqrt{n}\log n)}\choose{d-1}}\left(1+O\left(\frac{\log (n)}{ n^{d-1}}\right)\right)\frac{aU_T^+ + bU_T^-}{{n\choose d-1}} \notag\\
	&=\left(\frac{1}{2^{d-1}}+O\left(\frac{\log(n)}{\sqrt{n}}\right)\right)(aU_{T}^+ + bU_T^-).\label{CondE1}
\end{align}
Similarly, for $1\leq r\leq d-2$.
\begin{align*}
\mathbb E[U_{T+1}^{(0)}|\mathcal F_T,\mathcal A_T]&=\left(\frac{1}{2^{d-1}}+O\left(\frac{\log(n)}{\sqrt n}\right)\right)(bU_{T}^+ + aU_T^-),
\\
\mathbb E[U_{T+1}^{(r)}|\mathcal F_T,\mathcal A_T]&=\left(\frac{1}{2^{d-1}}+O\left(\frac{\log(n)}{\sqrt{n}}\right)\right){{d-1}\choose{r}}(bU_{T}^+ + bU_T^-).
\end{align*}
 Therefore from the estimations above, with the definition of $\alpha,\beta$ from \eqref{alphabeta},
\begin{align}
\mathbb E[\sum_{r=0}^{d-1}rU_{T+1}^{(r)}|\mathcal F_T,\mathcal A_T]
=&\left(1+O\left(\frac{\log(n)}{\sqrt n}\right)\right)\frac{1}{2^{d-1}}\left( (d-1)(aU_T^+ +bU_T^-)+\sum_{r=1}^{d-2} r{d-1\choose r}b( U_T^{+} + U_T^-)\right)	\notag\\
=&\left(1+O\left(\frac{\log(n)}{\sqrt n}\right)\right)\left(\frac{\alpha+\beta}{2}U_{T}^{+}+\frac{\alpha-\beta}{2}U_{T}^{-}\right).       \label{ExpConcentrate}
\end{align}

Since we have shown $ \sum_{r=0}^{d-1}U_{T+1}^{(r)}$ concentrated around its mean by $  \frac{\epsilon}{2d}$ with probability at least $1-O(n^{-2-\gamma})$,  conditioned on $\mathcal A_T$, we obtain
\begin{align}
\left|\sum_{r=0}^{d-1}rU_{T+1}^{(r)}-	\mathbb E[\sum_{r=0}^{d-1}rU_{T+1}^{(r)}|\mathcal F_T]\right| 
\leq &\sum_{r=0}^{d-1}r \left| U_{T+1}^{(r)}-\mathbb E[U_{T+1}^{(r)}\mid \mathcal F_T]\right|\notag  \leq \frac{\epsilon}{2d}\sum_{r=1}^{d-1}r\mathbb E[U_{T+1}^{(r)}\mid \mathcal F_T]  \notag\\
\leq &\frac{\epsilon}{4}\left(1+O\left(\frac{\log(n)}{\sqrt n}\right)\right)\left(\frac{\alpha+\beta}{2}U_{T}^{+}+\frac{\alpha-\beta}{2}U_{T}^{-}\right)\label{ConcentrationaroundMean}
\end{align} with probability   $1-O(n^{-2-\gamma})$. Therefore from \eqref{ExpConcentrate}, conditioned on $\mathcal A_T$, for large $n$, with probability $1-O( n^{-2-\gamma})$,
\begin{align}\label{330Concentration}
\sum_{r=0}^{d-1}rU_{T+1}^{(r)}\in \left[1-\frac{\epsilon}{3}, 1+\frac{\epsilon}{3} \right]\left(\frac{\alpha+\beta}{2}U_{T}^{+}+\frac{\alpha-\beta}{2}U_{T}^{-}\right).\end{align}

From \eqref{upperUT}, \eqref{UTstar} and \eqref{330Concentration}, conditioned on $\mathcal A_T$ and $\mathcal F_T$,  with probability  $1-O(n^{-2-\gamma})$,
\begin{align*}
	 U_{T+1}^+&\leq \sum_{r=0}^{d-1}rU_{T+1}^{(r)}+(d-2)U_{T+1}^{*}\leq \sum_{r=0}^{d-1}rU_{T+1}^{(r)}+(d-2)\frac{(\alpha-\beta)\epsilon S_T}{4d}\\
	 & \leq (1+\epsilon)\left(\frac{\alpha+\beta}{2}U_{T}^{+}+\frac{\alpha-\beta}{2}U_{T}^{-}\right).
\end{align*}

 Since $\mathbb P(\mathcal A_T)=1-n^{-2-\gamma}$, and by symmetry of $\pm$ labels,  with probability  $1-O(n^{-2-\gamma})$,
\begin{align}\label{UBDAT}
  U_{T+1}^{\pm}\leq  (1+\epsilon)\left(\frac{\alpha+\beta}{2}U_{T}^{\pm }+\frac{\alpha-\beta}{2}U_{T}^{\pm}\right).\end{align}
 
\textbf{(ii) Lower bound.}
To show \eqref{concentratecor1}, \eqref{concentratecor2} for $t'=T+1,t=T$, we cannot directly bound $U_{T+1}^{\pm}$ from below by $U_{T+1}^{(r)}, 1\leq r\leq d-1$ since from our definition of the connected $(d-1)$-subsets, they  can overlap with each other, which leads to over-counting of the number vertices with $\pm$ labels. 
In the following we show the overlaps between different connected $(d-1)$-sets are small, which gives us the desired lower bound. 

Let $W_{t+1,i}^{\pm}$ be the set of vertices in $ V_{>t}$ with  spin $\pm $ and appear in at least $i$ distinct connected $(d-1)$-subsets in $ V_{>t}$ for $i\geq 1$. Let $W_{t+1,i}=W_{t+1,i}^+\cup W_{t+1,i}^-$. From our definition, $W_{T+1,1}^+$ are the vertices with spin $+$ that appear in at least one connected $(d-1)$-subsets, so  $|W_{T+1,1}^+|\leq U_{T+1}^+$. By counting the multiplicity of  vertices with spin $+$, we have the following relation
\begin{align}\label{Wexpansion}
	\sum_{r=1}^{d-1}rU_{T+1}^{(r)}=|W_{T+1,1}^+|+\sum_{i\geq 2} |W_{T+1,i}^+|\leq U_{T+1}^+ +\sum_{i\geq 2}  |W_{T+1,i}|.
\end{align}
This implies a lower bound on $U_{T+1}^+$:
\begin{align}\label{LowerboundU}
	U_{T+1}^+\geq \sum_{r=1}^{d-1}rU_{T+1}^{(r)}-\sum_{i\geq 2}  |W_{T+1,i}|.
\end{align}

Next we control $|W_{T+1,2}|$. Let $m=n-|V_{\leq T}|$. We enumerate all vertices in $V_{>T}$ from $1$ to $m$ temporarily for the proof of the lower bound. Let $X_i, 1\leq i\leq m$ be  the random variables that $X_i=1$ if $i\in W_{T+1,2}$ and $0$ otherwise, we then have $|W_{T+1,2}|=\sum_{i=1}^m X_i$. A simple calculation yields
\begin{align}\label{PossionMoment}
|W_{T+1,2}|^2-|W_{T+1,2}|=\left(\sum_{i=1}^m X_i\right)^2-\sum_{i=1}^m X_i=2\sum_{1\leq i<j\leq m}X_iX_j. 	
\end{align}
The product $X_{i}X_{j}$ is $1$ if $i,j\in W_{T+1,2}$ and $0$ otherwise. 

We further consider 3 events, $E_{ij}^{s}$ for $s=0,1,2$, where $E_{ij}^{0}$ is the event that all $(d-1)$-subsets in $V_{>T}$ containing $i,j$ are not connected to $V_{T}$,  $E_{ij}^{1}$ is the event that there is only one $(d-1)$-subset in $V_{>T}$ containing $i,j$ connected to $V_{T}$ and $E_{ij}^{2}$ is the event that there are at least two $(d-1)$-subsets in $V_{>T}$ containing $i,j$ connected to $V_{T}$. Now we have 
\begin{align}\label{Eventij}
\mathbb E[X_iX_j \mid \mathcal F_T,\mathcal A_T]&=\mathbb P\left(i,j\in W_{T+1,2} \mid \mathcal F_T,\mathcal A_T\right) \notag\\
&=	\sum_{r=0}^2\mathbb P\left(i,j\in W_{T+1,2} \mid E_{ij}^{r},\mathcal F_T,\mathcal A_T\right)\mathbb P(E_{ij}^{r}\mid \mathcal F_T,\mathcal A_T). 
\end{align}
We estimate the three terms in the sum separately. Conditioned on $E_{ij}^{0}$, $\mathcal F_T$, and $\mathcal A_T$, the two events that $i\in W_{T+1,2}$ and $j\in W_{T+1,2}$ are independent. And the probability that $i\in W_{T+1,2}$ is bounded by 
$${n\choose d-2}^2 \left(\frac{a\vee b}{{n\choose d-1}}\right)^2 S_T^2\leq \frac{C_1\log^2(n)}{n^2}$$
for some constant $C_1>0$. So we have
\begin{align}
&\mathbb P\left(i,j\in W_{T+1,2} \mid E_{ij}^{0},\mathcal F_T,\mathcal A_T\right)\mathbb P(E_{ij}^{0}\mid \mathcal F_T,\mathcal A_T)
\leq 	\mathbb P\left(i,j\in W_{T+1,2} \mid E_{ij}^{0},\mathcal F_T,\mathcal A_T\right)\notag\\
=&\mathbb P\left(i\in W_{T+1,2} \mid E_{ij}^{0},\mathcal F_T,\mathcal A_T\right)\mathbb P\left(j\in W_{T+1,2} \mid E_{ij}^{0},\mathcal F_T,\mathcal A_T\right)\leq \frac{C_1^2\log^4 n}{n^4}.\label{Pterm1}
\end{align}

For the term that involves $E_{ij}^1$, we know for some $C_2>0$,
$$\mathbb P(E_{ij}^1 \mid \mathcal F_T,\mathcal A_T)\leq {n\choose d-3}\frac{a\vee b}{{n\choose d-1}}S_T\leq \frac{C_2\log n}{n^2},
$$
and conditioned on $E_{ij}^{1}$ and $\mathcal F_T,\mathcal A_T$, the two events that $i\in W_{T+1,2}$ and $j\in W_{T+1,2}$ are independent again, since we require $i,j$ to be contained in  at least $2$ connected-subsets. We have 
$$\mathbb P\left(i\in W_{T+1,2} \mid E_{ij}^{1},\mathcal F_T,\mathcal A_T\right)\leq {{n\choose d-2}}S_T\frac{a\vee b}{{n\choose d-1}}\leq \frac{C_2\log n}{n}.
$$ Therefore we have 
\begin{align}
&\mathbb P\left(i,j\in W_{T+1,2} \mid E_{ij}^{1},\mathcal F_T,\mathcal A_T\right)\mathbb P(E_{ij}^{1}\mid \mathcal F_T,\mathcal A_T)\notag\\
= &\mathbb P\left(i\in W_{T+1,2} \mid E_{ij}^{1},\mathcal F_T,\mathcal A_T\right)\mathbb P\left(j\in W_{T+1,2} \mid E_{ij}^{1},\mathcal F_T,\mathcal A_T\right)\mathbb P(E_{ij}^{1}\mid \mathcal F_T,\mathcal A_T)\notag\\
\leq &\frac{C_2^2\log^2 n}{n^2}\cdot  \frac{C_2\log n}{n^2}=\frac{C_2^3\log^3 n}{n^4}.\label{Pterm2}
\end{align}

Conditioned on $E_{ij}^{2}$, $i,j$ have already been included in 2 connected $(d-1)$ subsets, so $$\mathbb P\left(i,j\in W_{T+1,2} \mid E_{ij}^{2},\mathcal F_T,\mathcal A_T\right)=1.$$ We then have for some $C_3>0$,
\begin{align}\label{Pterm3}
&\mathbb P\left(i,j\in W_{T+1,2} \mid E_{ij}^{2},\mathcal F_T,\mathcal A_T\right)\mathbb P(E_{ij}^{2}\mid \mathcal F_T,\mathcal A_T)\notag\\
=&\mathbb P(E_{ij}^{2}\mid \mathcal F_T,\mathcal A_T)\leq {n\choose d-3}^2S_T^2\left(\frac{a\vee b}{{n\choose d-1}}\right)^2\leq \frac{C_3\log^2 n}{n^4}.
\end{align}
Combining \eqref{Pterm1}-\eqref{Pterm3}, we have for some constant $C'>0$,
\begin{align}
\mathbb E[X_iX_j \mid \mathcal F_T,\mathcal A_T]\leq \frac{C'\log^4 n}{n^4}.	
\end{align}  Taking conditional expectation in  \eqref{PossionMoment}, we have  
\begin{align*}
\mathbb E\left[|W_{T+1,2}|^2-|W_{T+1,2}|\mid \mathcal F_T,\mathcal A_T\right]=2\sum_{1\leq i<j\leq m}\mathbb E[X_iX_j\mid \mathcal F_T,\mathcal A_T]\leq \frac{C'\log^4 n}{n^2}.
\end{align*}

By Markov's inequality, there exists a constant $C>0$ such that for any constant $\lambda>0$ and sufficiently large $n$, 
\begin{align}\label{eq:WTTT}
&\mathbb P\left( |W_{T+1,2}|>\lambda S_T \mid \mathcal F_T,\mathcal A_T\right)
 \leq  	\mathbb P\left( |W_{T+1,2}|(|W_{T+1,2}|-1)>\lambda S_T(\lambda S_T-1) \mid \mathcal F_T,\mathcal A_T\right)\\
\leq & \frac{ \mathbb E[|W_{T+1,2}|(|W_{T+1,2}|-1)\mid \mathcal F_T,\mathcal A_T]}{\lambda S_T (\lambda S_T-1)}\leq \frac{C\log ^2 n}{\lambda^2 n^2}, \notag 
\end{align}
where in the last inequality we use the fact that $S_T\geq K\log n$. Taking $\lambda=\frac{(\alpha-\beta)\epsilon}{4}$, we have  for all large $n$ and for any $\gamma\in (0,1)$,
\begin{align}\mathbb P\left( |W_{T+1,2}|>\frac{(\alpha-\beta)\epsilon}{4} S_T \mid \mathcal F_T,\mathcal A_T\right)=O\left(\frac{\log^2 n}{n^2}\right)\leq n^{-1-\gamma}.\label{ProbW2}
\end{align}  

For a fixed vertex $j\in V_{>T}$, the probability that $j\in W_{T+1,i}$ is at most ${n\choose d-2}^iS_T^i \left(\frac{a\vee b}{{n\choose d-1}}\right)^i,$
then  we have for sufficiently large $n$,
\begin{align*}
\mathbb E[|W_{T+1,i}| \mid \mathcal F_T,\mathcal A_T]\leq n {n\choose d-2}^i	S_T^i\left(\frac{a\vee b}{{n\choose d-1}}\right)^i\leq n\left(\frac{C_4\log n}{n}\right)^i
\end{align*}
for some $C_4>0$. For the rest of the terms in \eqref{Wexpansion}, we have for some constant $C>0$,
\begin{align*}
\mathbb E\left[\sum_{i\geq 3}  |W_{T+1,i}|\mathrel{\Big|} \mathcal F_T,\mathcal A_T\right]\leq n\sum_{i=3}^{\infty} \left(\frac{C_4\log n}{n}\right)^i\leq \frac{C\log^3(n)}{n^2}.	
\end{align*}
By Markov's inequality,
\begin{align*}
\mathbb P\left(\sum_{i\geq 3} |W_{T+1,i}|\geq \frac{(\alpha-\beta)\epsilon}{4} S_T \mid \mathcal F_T,\mathcal A_T\right)\leq \frac{C\log^2 (n)}{n^2}\leq n^{-1-\gamma}.	
\end{align*}
Together with \eqref{ProbW2},  we have conditioned on $\mathcal A_T$,
$
	\sum_{i\geq 2} |W_{T+1,2}^+|\leq \frac{(\alpha-\beta)\epsilon}{2}S_T
$
with probability at least $1-2n^{-1-\gamma}$ for any $\gamma\in (0,1)$ and all large $n$.  Note that
$$\frac{(\alpha-\beta)\epsilon}{2}S_T\leq \frac{\epsilon}{2}\left(\frac{\alpha+\beta}{2}U_{T}^{+}+\frac{\alpha-\beta}{2}U_{T}^{-}\right).$$
With  \eqref{330Concentration}, \eqref{LowerboundU}, and \eqref{eq:PAT}, we have 
\begin{align*}
U_{T+1}^+&\geq \sum_{r=1}^{d-1}rU_{T+1}^{(r)}-\frac{\epsilon}{2}\left(\frac{\alpha+\beta}{2}U_{T}^{+}+\frac{\alpha-\beta}{2}U_{T}^{-}\right)\geq (1-\epsilon)\left(\frac{\alpha+\beta}{2}U_{T}^{+}+\frac{\alpha-\beta}{2}U_{T}^{-}\right)
\end{align*} with probability $1-O(n^{-1-\gamma})$.
 By symmetry, the argument works for $U_{T+1}^-$, therefore with probability $1-O(n^{-1-\gamma})$ for any $\gamma\in (0,1)$, we have 
\begin{align}\label{lowerU}
U_{T+1}^{\pm}\geq (1-\epsilon)\left(\frac{\alpha+\beta}{2}U_{T}^{\pm }+\frac{\alpha-\beta}{2}U_{T}^{\mp }\right).
\end{align} 

From  \eqref{UBDAT} and \eqref{lowerU},  we have with probability $1-O(n^{-1-\gamma})$ for any $\gamma\in (0,1)$, \eqref{eq:818} holds. 

\textbf{Step 2: Induction.}
It remains to extend this estimate in Step 1 for all $T\leq t'<t\leq l$. We now define the event 
\begin{align}\label{eventAt}
\mathcal {A}_t:=\left\{ U_t^{\pm}\in [1-\epsilon_{t-1},1+\epsilon_{t-1}]\left(\frac{\alpha+\beta}{2}U_{t-1}^{\pm }+\frac{\alpha-\beta}{2}U_{t-1}^{\pm}\right)\right\}	
\end{align} for $T+1\leq t\leq l$,
and recall  $  \epsilon_t=\epsilon\alpha^{-(t-T)/2},\mathcal A_T=\{S_T\leq K_3\log n\}$. 

From the proof above, we have shown  $\mathcal A_{T+1}$ holds with probability $1-O(n^{-1-\gamma})$. Conditioned on $\mathcal A_T$, $\mathcal A_{T+1},\cdots, \mathcal A_t$ for some fix $t$ with $T+2\leq t\leq l$, the vector $\vec{U}_t=(U_t^+,U_t^-)$ satisfies \eqref{concentratecor1}, \eqref{concentratecor2} for any $T\leq t'<t$. 

Set $t'=T+1$. From \cite{massoulie2014community}, for any integer $k>0$,
$M^k=\frac{1}{2}\begin{bmatrix}
			\alpha^k+\beta^k &\alpha^k-\beta^k\\
			\alpha^k-\beta^k &\alpha^k+\beta^k
		\end{bmatrix}
		$.
\eqref{concentratecor1} implies that 
\begin{align}U_t^{\pm} &\geq \left(\prod_{s=T+1}^{t-1}(1-\epsilon_s)\right) \left(\frac{\alpha^{t-T-1}+\beta^{t-T-1}}{2}U_{T+1}^{\pm}+\frac{\alpha^{t-T-1}-\beta^{t-T-1}}{2}U_{T+1}^{\mp }\right) \notag\\
& \geq (1-O(\epsilon))\frac{\alpha^{t-T-1}}{2}(1-\epsilon)\left(\frac{\alpha+\beta}{2}U_T^{\pm}+\frac{\alpha-\beta}{2}U_T^{\mp}\right)\notag\\
&\geq (1-O(\epsilon))\alpha^{t-T} \frac{(1-\epsilon)(\alpha-\beta)}{4\alpha}S_T
\geq C_1\alpha^{t-T}\log(n), \label{LowerBoundtT}
\end{align} for some constant $C_1>0$. 
For any $t$ with $T\leq t$, conditioned on $\mathcal A_T$, $\mathcal A_{T+1},\cdots, \mathcal A_t$, since $\beta<\alpha$,
\begin{align}
U_{t}^{\pm}\leq	&\left(\prod_{s=T}^{t-1}(1+\epsilon_s)\right) \left(\frac{\alpha^{t-T}+\beta^{t-T}}{2}U_{T}^{\pm}+\frac{\alpha^{t-T}-\beta^{t-T}}{2}U_{T}^{\mp }\right)\notag\\
 &\leq (1+O(\epsilon))\frac{\alpha^{t-T}+\beta^{t-T}}{2}S_T \leq  (1+O(\epsilon))\alpha^{t-T}K_3\log(n)\leq C_2\alpha^{t-T}\log n\label{UpperboundT}
\end{align}
for some $C_2>0$. 
Combining lower and upper bounds on $U_t^{\pm}$, we obtain 
\begin{align}\label{eq:ThetaSt}
S_t=U_t^+ +U_t^-=\Theta(\alpha^{t-T}\log n).	
\end{align}
  We now show by induction that $\mathcal A_{t+1}$ holds with high enough probability conditioned on $\{\mathcal A_j, T\leq j\leq t\}$. 

\textbf{(i) Upper bound.}
Note that $\alpha^l=o(n^{1/4})$, for some constant $C>0$ 
$$U_{\leq t}^+\leq \sum_{i=1}^t S_i\leq C\alpha^{t-T}\log^2 n\leq C\alpha^{l}\log n=o(n^{1/4}\log n).$$ 
Recall $|n^{\pm}-\frac{n}{2}|\leq \sqrt{n}\log n$.
From \eqref{mean1}-\eqref{mean3}, similar to the case for $t=T$, we have
\begin{align*}
&\mathbb E[U_{t+1}^{(d-1)}|\cap_{j=T}^{t}\mathcal A_j,\mathcal F_t]
= {{n^+-U_{\leq t}^+}\choose{d-1}}\left( 1-\left(1-\frac{a}{{{n}\choose{d-1}}}\right)^{U_t^+}\left(1-\frac{b}{{{n}\choose{d-1}}}\right)^{U_t^-}\right)   \notag 	\\ =&\left(\frac{1}{2^{d-1}}+O\left(\frac{\log n}{\sqrt n}\right)\right)(aU_{t}^+ + bU_t^-),\notag 
\end{align*}
and
\begin{align*}
\mathbb E[U_{t+1}^{(0)}|\cap_{j=T}^{t}\mathcal A_j, \mathcal F_t]&=(\frac{1}{2^{d-1}}+O(\frac{\log n}{\sqrt n}))(bU_{t}^+ + aU_t^-),\\
\mathbb E[U_{t+1}^{(r)}|\cap_{j=T}^{t}\mathcal A_j,\mathcal F_t]&=(\frac{1}{2^{d-1}}+O(\frac{\log n}{\sqrt n})){{d-1}\choose{r}}(bU_{t}^+ + bU_t^-), 
\end{align*} for $1\leq r\leq d-2.$
Hence there exists a constant $C_0>0$ such that for all $0\leq r\leq d-1$, \[  \mathbb E[U_{t+1}^{(r)}|\cap_{j=T}^{t}\mathcal A_j,\mathcal F_t]\geq C_0S_t.\]

From \eqref{twosidedbound} in Lemma \ref{ChernoffBound}, for any $0\leq r\leq d-1$, to show 
\begin{align}\label{conditionalChernoff}
&\mathbb P\left(	\left|U_{t+1}^{(r)}-\mathbb E[U_{t+1}^{(r)}\mid \cap_{j=T}^{t}\mathcal A_j,\mathcal F_{t}]\right|\leq \frac{\epsilon}{2d}\mathbb E[U_{t+1}^{(r)} \mid \cap_{j=T}^{t}\mathcal A_j,\mathcal F_{t}]  \mathrel{\Big|} \cap_{j=T}^{t}\mathcal A_j,\mathcal F_{t}\right)
\geq  1-n^{-2-\gamma},
\end{align}
it suffices to have
\begin{align}\label{InductionBound1}
C_0S_t\tilde{h}\left(\frac{\epsilon_t}{2d}\right)\geq (2+\gamma)\log n.\end{align}

From \eqref{twosidedbound}, by a second-order expansion of $\tilde{h}$ around $0$, $\tilde{h}(x)\geq x^2/3$ when $x>0$ is small. For $\gamma\in (0,1)$, the left hand side in \eqref{InductionBound1} is lower bounded by 
\begin{align*}
C_1K\alpha^{t-T}\log (n) \tilde{h}\left(\frac{\epsilon_t}{2d}\right)
\geq & C_2\alpha^{t-T}K\log (n)\epsilon_t^2= C_2K\log n\geq (2+\gamma)\log n,
\end{align*}
by taking $K$ large enough. Therefore \eqref{conditionalChernoff} holds.

 We also have
\begin{align*}
	U_{t+1,s}&\preceq Z_{t+1,s}, \quad Z_{t+1,s}\sim  \textnormal{Bin}\left({{n}\choose{s}},\frac{a\vee b}{{{n}\choose{d-1}}}{{S_{t}}\choose{d-s}}\right),
\end{align*} 
and  $Z_{t+1,s}$ has mean $  {n\choose s}	\frac{a\vee b}{{{n}\choose{d-1}}}{{S_t}\choose{d-s}}=\Theta\left(\frac{\alpha^{(d-s)(t-T)}\log^{d-s}(n)}{n^{d-1-s}}\right).$  
For $1\leq s\leq d-2$, using the fact that $  h(x)\geq \frac{1}{2}x\log(x)$ for $x$ large enough, similar to  \eqref{chernoffonesided}, there are constants $C_1,C_2,C_3, C_4>0$ such that for any $\lambda>0$,
\begin{align*}
&\mathbb P(U_{t+1,s}\geq \lambda  S_t \mid \mathcal \cap_{j=T}^t\mathcal A_j, \mathcal F_t)
\leq   \mathbb P(Z_{t+1,s}\geq \lambda  S_t \mid\mathcal \cap_{j=T}^t\mathcal A_j, \mathcal F_t)\\
\leq &\exp\left(-C_1\lambda \alpha^{t-T}\log(n)\log\left( \frac{C_2\lambda \alpha^{t-T}\log(n)}{C_3\alpha^{(d-s)(t-T)}\log^{d-s}(n)n^{1+s-d}}\right) \right).
\end{align*}
Taking $\lambda=\frac{(\alpha-\beta)\epsilon_t}{4d^2}=\frac{(\alpha-\beta)\epsilon \alpha^{-(t-T)/2}}{4d^2},$ 
we have 
\begin{align*}
&\mathbb P\left(U_{t+1,s}\geq \frac{(\alpha-\beta)\epsilon_t}{4d^2}  S_t \mid \mathcal \cap_{j=T}^t\mathcal A_j, \mathcal F_t\right)\\
\leq & \exp\left(-C_1'\alpha^{(t-T)/2}\log (n)\cdot\log (C_2'\alpha^{(s-d+\frac{1}{2})(t-T)}\log^{1+s-d}(n)n^{d-1-s})\right).	
\end{align*}
Since for some constants $C_4,C_5,C_6>0$,
\begin{align*}&\log (C_2'\alpha^{(s-d+\frac{1}{2})(t-T)}\log^{1+s-d}(n)n^{d-1-s}) \\
\geq &C_4-C_5(t-T)\log (\alpha)+\log(\log^{1+s-d}(n))+(d-1-s)\log n
\geq C_6\log n,
\end{align*}
we have for all $1\leq s\leq d-2$,
\begin{align}\label{CONCENTRATEMEAN}
\mathbb P(U_{t+1,s}\geq \frac{(\alpha-\beta)\epsilon_t}{4d^2}  S_t \mid \mathcal \cap_{j=T}^t\mathcal A_j, \mathcal F_t)\leq \exp\left (-C_1'C_6\log^2 n\right)\leq n^{-2-\gamma}
\end{align}
for any $\gamma\in (0,1)$.  Recall for sufficiently large $n$, 
$$\epsilon_t=\epsilon\alpha^{-(t-T)/2}\geq \epsilon \alpha^{-l/2}>n^{-1/8}.$$ 
Therefore $\frac{\log n}{\sqrt n}=o(\epsilon_t)$. From  \eqref{CONCENTRATEMEAN}, conditioned on $\mathcal A_T,\dots, A_t$ and $\mathcal F_t$,
 $$U_{t+1}^{+}\leq \sum_{r=1}^{d-1}rU_{t+1}^{(r)}+(d-2)U_{t+1}^{*}\leq (1+\epsilon_t)\left(\frac{\alpha+\beta }{2}U_{t}^{+ }+\frac{\alpha-\beta}{2}U_{t}^{-}\right)$$
  with probability at least $1-O(n^{-2-\gamma})$. A similar bound works for $U_{t+1}^-$, which implies conditioned on $\mathcal A_{T},\dots, A_{t}$,
   \begin{align}\label{eq:PUUT}
   U_{t+1}^{\pm }\leq (1+\epsilon_t)\left(\frac{\alpha+\beta }{2}U_{t}^{\pm }+\frac{\alpha-\beta}{2}U_{t}^{\pm}\right)
   \end{align}
 with probability $1-O(n^{-2-\gamma})$ for any $\gamma\in (0,1)$.

 \textbf{(ii) Lower bound.} We need to show that conditioned on $\mathcal A_{T},\dots, \mathcal A_t$, 
 $U_{t+1}^{\pm}\geq  (1-\epsilon_t)\left(\frac{\alpha+\beta}{2}U_{t}^{\pm }+\frac{\alpha-\beta}{2}U_{t}^{\pm}\right)$
 	  with probability $1-O(n^{-1-\gamma})$ for some $\gamma \in (0,1)$.	This part of the proof is very similar to the case for $t=T$. 
 	  Same as \eqref{LowerboundU}, we have the following  lower bound on $U_{t+1}^+$:
	\[  U_{t+1}^+\geq \sum_{r=1}^{d-1}rU_{t+1}^{(r)}-\sum_{i\geq 2}|W_{t+1,i}|.\]
	
 Next we control $|W_{t+1,2}|$. Let $m=n-|V_{\leq t}|$ and we enumerate all vertices in $V_{>t}$ from $1$ to $m$. Let $X_{1},\dots X_m$ be the random variable that $X_i=1$ if $i\in W_{t+1,2}$ and $0$ otherwise. Same as  \eqref{PossionMoment},
\begin{align}\label{PossionMoment2}
|W_{t+1,2}|^2-|W_{t+1,2}|=2\sum_{1\leq i<j\leq m}X_iX_j.	
\end{align}
Let $E_{ij}^{s}$ for $s=0,1,2$, be the similar events as in \eqref{Eventij} before, now we have 
\begin{align*}
&\mathbb E[X_iX_j \mid \cap_{j=T}^t\mathcal A_j,\mathcal F_t]=\mathbb P\left(i,j\in W_{t+1,2} \mid \cap_{j=T}^t\mathcal A_j,\mathcal F_t\right)\\
=&	\sum_{r=0}^2\mathbb P\left(i,j\in W_{t+1,2} \mid E_{ij}^{r},\cap_{j=T}^t\mathcal A_j,\mathcal F_t\right)\mathbb P(E_{ij}^{r}\mid \cap_{j=T}^t\mathcal A_j,\mathcal F_t).
\end{align*}
The three terms in the sum can be estimated separately in the same way as before. By using the upper bound $C\alpha^{t-T}\log n\leq S_t\leq C_0\alpha^{t-T}\log n$ for some $C,C_0>0$, and use the same argument for the case when $t=T$, we have the following three inequalities for some constants $C_1,C_2,C_3>0$:
\begin{align*}
&\mathbb P\left(i,j\in W_{t+1,2} \mid E_{ij}^{0},\mathcal F_t\right)\mathbb P(E_{ij}^{0}\mid \cap_{j=T}^t\mathcal A_j,\mathcal F_t)\leq	\frac{C_1^2\alpha^{4(t-T)}\log^4 n}{n^4},\\
&\mathbb P\left(i,j\in W_{t+1,2} \mid E_{ij}^{1},\mathcal F_t\right)\mathbb P(E_{ij}^{1}\mid \cap_{j=T}^t\mathcal A_j,\mathcal F_t) \leq \frac{C_2^3\alpha^{3(t-T)}\log^3 n}{n^4},\\
&\mathbb P\left(i,j\in W_{t+1,2} \mid E_{ij}^{2},\mathcal F_t\right)\mathbb P(E_{ij}^{2}\mid \cap_{j=T}^t\mathcal A_j,\mathcal F_t) \leq \frac{C_3\alpha^{2(t-T)}\log^2 n}{n^4}.
\end{align*}
This implies
$\mathbb E[X_iX_j \mid \cap_{j=T}^t\mathcal A_j,\mathcal F_t]\leq \frac{C'\alpha^{4(t-T)}\log^4 n}{n^4}	$ for some  $C'>0$. Taking conditional expectation in  \eqref{PossionMoment2}, we have  
\begin{align*}
\mathbb E\left[|W_{t+1,2}|^2-|W_{t+1,2}|\mid \cap_{j=T}^t\mathcal A_j,\mathcal F_t\right]\leq \frac{C'\alpha^{4(t-T)}\log^4 n}{n^2}.
\end{align*}
Then by Markov inequality and \eqref{eq:ThetaSt}, similar to \eqref{eq:WTTT},  there exists a constant $C>0$ such that for any $\lambda=\Omega (\alpha^{-(t-T)})$, 
\begin{align*}
\mathbb P\left( |W_{t+1,2}|>\lambda S_t \mid \cap_{j=T}^t\mathcal A_j, \mathcal F_t\right)& \leq \frac{C\alpha^{2(t-T)}\log ^2 n}{\lambda^2 n^2}.
\end{align*}
Take $\lambda=\frac{(\alpha-\beta)\epsilon_t}{4}$. Since $c\log(\alpha)<1/4$, we have $\alpha^{l}<n^{1/4}$, and \begin{align*}\mathbb P\left( |W_{t+1,2}|>\frac{(\alpha-\beta)\epsilon_t}{4} S_t \mid \cap_{j=T}^t\mathcal A_j,\mathcal F_t\right)\leq \frac{C\alpha^{2(t-T)}\log^2 n}{n^2}\leq n^{-1-\gamma}
\end{align*} for any $\gamma\in (0,1/2)$.

For each $|W_{t+1,i}|$ for $i\geq 3$, we have for sufficiently large $n$, there exists  a constant $C_4>0$
\begin{align*}
\mathbb E[|W_{t+1,i}| \mid \cap_{j=T}^t\mathcal A_j,\mathcal F_t]\leq n {n\choose d-2}^i	S_t^i\left(\frac{a\vee b}{{n\choose d-1}}\right)^i\leq n\left(\frac{C_4\alpha^{t-T}\log n}{n}\right)^i.
\end{align*}
For the rest of the terms, we have for some constant $C_4'>0$,
\begin{align*}
\mathbb E\left[\sum_{i\geq 3}  |W_i|\mid \cap_{j=T}^t\mathcal A_j, \mathcal F_t\right]\leq n\sum_{i=3}^{\infty} \left(\frac{C_4\alpha^{t-T}\log n}{n}\right)^i\leq \frac{C_4'\alpha^{3(t-T)}\log^3(n)}{n^2}.	
\end{align*}
By Markov's inequality,  
\begin{align*}
\mathbb P\left(\sum_{i\geq 3} |W_i|\geq \frac{(\alpha-\beta)\epsilon_t}{4} S_t \mid \cap_{j=T}^t\mathcal A_j, \mathcal F_t\right)\leq \frac{C_5\alpha^{2.5(t-T)}\log^2 (n)}{n^2}\leq n^{-1-\gamma}	
\end{align*}
for any $\gamma\in (0,3/8)$.
Together with the estimate on $W_{t+1,2}$,  we have 
\begin{align*}
	\sum_{i\geq 2} |W_{t+1,2}^+|\leq \frac{(\alpha-\beta)\epsilon_t}{2}S_t\leq \frac{\epsilon_t}{2}\left(\frac{\alpha+\beta}{2}U_{t}^{+}+\frac{\alpha-\beta}{2}U_{t}^{-}\right)
\end{align*}
with probability $1-2n^{-1-\gamma}$ for any $\gamma\in (0,3/8)$.

With \eqref{LowerboundU} and \eqref{330Concentration}, 
$
U_{t+1}^+\geq (1-\epsilon_t)\left(\frac{\alpha+\beta}{2}U_{t}^{+}+\frac{\alpha-\beta}{2}U_{t}^{-}\right)
$ with probability $1-O(n^{-1-\gamma})$. By symmetry, the argument works for $U_{t+1}^-$. Therefore conditioned on $\mathcal A_{T},\dots, \mathcal A_t$, with probability $1-O(n^{-1-\gamma})$ for any $\gamma\in (0,3/8)$,  
\begin{align}\label{lowerU1}
U_{t+1}^{\pm}\geq (1-\epsilon_t)\left(\frac{\alpha+\beta}{2}U_{t}^{\pm }+\frac{\alpha-\beta}{2}U_{t}^{\mp }\right).
\end{align} 
This finishes the proof the lower bound part of Step 2.

Recall \eqref{eventAt}.
With  \eqref{lowerU1} and \eqref{eq:PUUT},  we have shown that conditioned on $\mathcal A_{T},\dots, \mathcal A_t$, with probability $1-O(n^{-1-\gamma})$, $\mathcal A_{t+1}$ holds.
This finishes the induction step. 
 Finally, for fixed $i\in [n]$ and $\gamma \in (0,3/8)$,
\begin{align*}
\mathbb P\left(\bigcap_{t=T}^l \mathcal A_t\right)&=\mathbb P(\mathcal A_T)\prod_{t=T+1}^l \mathbb P(\mathcal A_t \mid \mathcal A_{t-1},\dots , \mathcal A_T)	 \\
&\geq (1-Cn^{-2-\gamma})(1-Cn^{-1-\gamma})^l\geq 1-C_6\log (n)n^{-1-\gamma},
\end{align*}
for some constant $C_6>0$. Taking a union bound over $i\in [n]$, we have shown $\mathcal A_t$ holds for all $T\leq t\leq l$ and all $i\in [n]$ with probability $1-O(n^{-\gamma})$ for any $\gamma\in (0,3/8)$. This completes the proof of Theorem \ref{growthbound}.
\end{proof}

With Theorem \ref{growthbound}, the  rest of the proof of  Theorem \ref{quasi} follows similarly from the proof of Theorem 2.3 in \cite{massoulie2014community}. We include it for completeness.

\begin{proof}[Proof of Theorem \ref{quasi}] Assume all the estimates in statement of Theorem \ref{growthbound} hold.
For $t\leq l$, if $t\leq T$, from the definition of $T$, we  have $S_t,|D_t|=O(\log n)$. For $t>T$,  from \cite{massoulie2014community}, $M$ satisfies
$$M^k=\frac{1}{2}\begin{bmatrix}
			\alpha^k+\beta^k &\alpha^k-\beta^k\\
			\alpha^k-\beta^k &\alpha^k+\beta^k
		\end{bmatrix}.
		$$
Using \eqref{concentratecor1} and \eqref{concentratecor2},  we have for $t>t'\geq T$,
\begin{align}
S_t&\leq \left(\prod_{s=t'}^{t-1}(1+\epsilon_s)\right)(1,1)M^{t-t'}\vec{U}_{t'}\leq \left(\prod_{s=t'}^{t-1}(1+\epsilon_s)\right)\alpha^{t-t'}S_{t'} \label{STT},\\
S_t&\geq \left(\prod_{s=t'}^{t-1}(1-\epsilon_s)\right)(1,1)M^{t-t'}\vec{U}_{t'}\geq \left(\prod_{s=t'}^{t-1}(1-\epsilon_s)\right)\alpha^{t-t'}S_{t'}.	\label{STT2}
\end{align}
Setting $t'=T$ in \eqref{STT}, we obtain
$$S_t\leq \left(\prod_{s=T}^{t-1}(1+\epsilon_s)\right)\alpha^{t-T}S_{T}=O(\alpha^{t-T}\log n)=O(\alpha^t\log n).
$$
Therefore \eqref{claim1} holds. Let $t=l$ in \eqref{STT} and \eqref{STT2}, we have for all $T\leq t'<l$,
$$ \left(\prod_{s=t'}^{l-1}(1-\epsilon_s)\right)\alpha^{l-t'}S_{t'}\leq S_l\leq \left(\prod_{s=t'}^{l-1}(1+\epsilon_s)\right)\alpha^{l-t'}S_{t'}.
$$
And it implies
\begin{align}\label{boundonst}
\left(\prod_{s=t'}^{l-1}(1-\epsilon_s)\right)S_{t'}\leq \alpha^{t'-l}S_l\leq \left(\prod_{s=t'}^{l-1}(1+\epsilon_s)\right)S_{t'}.
\end{align}
Note that 
$$\max\left\{\prod_{s=t'}^{l-1}(1+\epsilon_s)-1,1-\prod_{s=t'}^{l-1}(1-\epsilon_s)\right\}=O(\epsilon_{t'})=O(\alpha^{-t'/2}).
$$
Together with \eqref{boundonst}, we have for all $T\leq t'<l$, 
\begin{align}\label{boundonST}
 	|S_{t'}-\alpha^{t'-l}S_l|\leq O(\alpha^{-t'/2})S_{t'}=O(\alpha^{t'/2}\log n).
 \end{align}
 
On the other hand, for $t\leq T$, we know $S_t=O(\log n )$. Let $t'=T$ in  \eqref{boundonST}, we have
\begin{align}\label{boundonST2}
|S_T-\alpha^{T-l}S_l|=O(\alpha^{T/2}\log n).
\end{align}
So for $1\leq t\leq T$,
\begin{align}
	|S_t-\alpha^{t-l}S_l|&     =O(\log n)+\alpha^{t-T}(S_T+O(\log(n)\alpha^{T/2}))\notag\\
	&=O(\log n)+ O(\alpha^{t-T/2}\log n )=O(\alpha^{t/2}\log n)\label{smallt}.
\end{align}
The last inequality comes from the inequality $t-T/2\leq t/2$.
Combining  \eqref{boundonST} and \eqref{smallt}, we have proved \eqref{claim3} holds for all $1\leq t\leq l$. 

Using \eqref{concentratecor1} and \eqref{concentratecor2}, we have
\begin{align*}
 D_{t+1}=U_{t+1}^+-U_{t+1}^-&\leq \beta(U_t^+ - U_t^-)+\alpha\epsilon_t(U_t^+ +U_t^-)=\beta D_{t}+\alpha\epsilon_t S_t.
\end{align*}
Similarly,
$
\beta D_{t}	-\alpha\epsilon_t S_t \leq D_{t+1}\leq \beta D_{t}+\alpha\epsilon_t S_t .
$
By iterating, we have for $l\geq t>t'\geq T$,
\begin{align}\label{Dt}
	|D_t-\beta^{t-t'}D_{t'}|\leq \sum_{s=t'}^{t-1}\alpha \beta^{t-1-s}\epsilon_sS_s.
\end{align}
Recall $S_s=O(\log(n)\alpha^{s-T})$,
$|D_T|=O(\log n)$, and $\epsilon_s=\alpha^{-(s-T)/2}$. Taking  $t'=T$ in \eqref{Dt},  for $t>T$,
\begin{align*}
|D_t| =O\left(\log (n)\beta^t\right)+O\left(\sum_{s=T}^{t-1}\alpha\beta^{t-1-s}\log (n)\alpha^{(s-T)/2}	\right).
\end{align*}
Since $1<\alpha<\beta^2$, it follows that
\begin{align*}
	\sum_{s=T}^{t-1}\alpha\beta^{t-1-s}\log (n)\alpha^{(s-T)/2}=&\beta^{t-1}\alpha^{1-T/2}\log (n)\sum_{s=T}^{t-1}\left(\frac{\alpha}{\beta^2}\right)^{s/2}\\
	=&\beta^{t-1}\alpha^{1-T/2}\log (n)O(\alpha^{T/2}\beta^{-T})=O(\log( n)\beta^t).
\end{align*}
So we have $|D_t|=O(\log n\beta^t)$. 
The right side of \eqref{Dt} is of order
$$
\sum_{s=t'}^{t-1}\alpha \beta^{t-1-s}\alpha^{(s-T)/2}\log (n)=O(\log(n)\beta^{t-t'}\alpha^{t'/2}).
$$
Thus setting $t=l$ in \eqref{Dt}, for $l>t'\geq T$, we obtain 
$
D_l-\beta^{l-t'}D_{t'}=O(\log (n)\beta^{l-t'}\alpha^{t'/2}).	
$
Therefore
$D_{t'}=\beta^{t'-l}D_l+O(\log(n)\alpha^{t'/2})
$
holds for all $T\leq t'<l$. 
For $t'<T$, we have $D_{t'}=O(\log n)$ and 
\begin{align*}
	|D_{t'}-\beta^{t'-l}D_l|&\leq O(\log n)+\beta^{t'-T}(|D_T|+O(\log(n)\alpha^{T/2}))\notag\\
	&=O(\log n)+ O(\beta^{t'-T}\alpha^{T/2}\log n )=O(\alpha^{t'/2}\log n),
\end{align*}
where the last estimate is because $\beta^{t'-T}<\alpha^{(t'-T)/2}$ under the condition that $t'<T$. Altogether we have shown \eqref{claim4} holds for all $1\leq t'\leq l$. This completes the proof of Theorem \ref{quasi}.
\end{proof}

 \section{Proof of Theorem \ref{thm:endsection}}\label{sec:given}
 
 We first state the following lemma before proving Theorem \ref{thm:endsection}. The proof is included in Appendix \ref{A6}.
 \begin{lemma}\label{Cor412}
	For all $m\in\{1,\dots, l\}$ with $l=c\log n$, $c\log\alpha<1/4$, it holds asymptotically almost surely that
	\begin{align}
\sup_{\|x\|_2=1,x^{\top}B^{(l)}\mathbf{1}=x^{\top}B^{(l)}\sigma=0}\|\mathbf{1}^{\top}B^{(m-1)}x\|_2&={O}(\sqrt{n}\alpha^{(m-1)/2}\log n)\label{bound1},\\
\sup_{\|x\|_2=1,x^{\top}B^{(l)}\mathbf{1}=x^{\top}B^{(l)}\sigma=0}\|\sigma^{\top}B^{(m-1)}x\|_2&={O}(\sqrt{n}\alpha^{(m-1)/2}\log n)	\label{bound2}.	
	\end{align}
\end{lemma}

\begin{proof}[Proof of Theorem \ref{thm:endsection}]
	Using matrix expansion identity \eqref{matrixexpansion} and the estimates in Theorem \ref{matrixexpansionthm}, for any $l_2$-normalized vector $x$ with $x^{\top}B^{(l)}\mathbf{1}=x^{\top}B^{(l)}\sigma=0$, we have for sufficiently large $n$, asymptotically almost surely
	\begin{align}
		\|B^{(l)}x \|_2 &=\left\|\Delta^{(l)}x+\sum_{m=1}^l (\Delta^{(l-m)}\overline{A}B^{(m-1)})x-\sum_{m=1}^l \Gamma^{(l,m)}x\right\|_2 \notag\\
		&\leq \rho(\Delta^{(l)})+\sum_{m=1}^l\rho(\Delta^{(l-m)})\|\overline{A}B^{(m-1)}x\|_2+\sum_{m=1}^l\rho(\Gamma^{(l,m)}) \notag\\
		&\leq 2n^{\epsilon}\alpha^{l/2}+ \sum_{m=1}^ln^{\epsilon}\alpha^{(l-m)/2}\|\overline{A}B^{(m-1)}x\|_2, \label{expanAB}
	\end{align}
	where $\overline{A}=\mathbb E_{\mathcal H_n}\left[ A\mid \sigma\right]$. We have the following expression for entries of  $\overline{A}$. If $i\not=j$ and $\sigma_i=\sigma_j=+1$,
	\begin{align*}
		\overline{A}_{ij}=\frac{a}{{{n}\choose{d-1}}}{{n^{+}-2}\choose{d-2}}+\frac{b}{{{n}\choose{d-1}}}\left({{n-2}\choose{d-2}}-{{n^{+}-2}\choose{d-2}}\right)=:\tilde{a}_n^+.
	\end{align*}
	If $i\not=j$ and $\sigma_i=\sigma_j=-1$,
	\begin{align*}
		\overline{A}_{ij}=\frac{a}{{{n}\choose{d-1}}}{{n^{-}-2}\choose{d-2}}+\frac{b}{{{n}\choose{d-1}}}\left({{n-2}\choose{d-2}}-{{n^{-}-2}\choose{d-2}}\right)=:\tilde{a}_n^-.
	\end{align*}
	If $\sigma_i\not=\sigma_j$,
	\[ 
		\overline{A}_{ij}=\frac{b}{{{n}\choose{d-1}}}{{n-2}\choose{d-2}}:=\tilde{b}_n.
	\] 
	We then have $\tilde{a}_n^{+},\tilde{a}_n^{-},\tilde{b}_n=O(1/n)$.   Conditioned on the event $\{|n^{\pm}-n/2|\leq \log (n)\sqrt n\}$, we obtain
	$$\tilde{a}_n^--\tilde{a}_n^+=\frac{a-b}{{{n}\choose{d-1}}}\left({{n^{-}-2}\choose{d-2}}-{{n^{+}-2}\choose{d-2}}\right)=O\left(\frac{\log n}{n^{3/2}}\right).
	$$
	
	Let $R$ be a $n\times n$ matrix such that 
	$$R_{ij}=\begin{cases}
		1 & \sigma_i=\sigma_j=-1 \text{ and } i\not=j,\\
		0 & \text{otherwise.}
	\end{cases}$$ We then have $\|R\|_2\leq \sqrt{\sum_{ij} R_{ij}^2}\leq  n$. The following decomposition of $\overline{A}$ holds:
	\begin{align}
		\overline{A}&=\tilde{a}_n^{+}\left[\frac{1}{2}(\mathbf{1}
		\cdot \mathbf{1}^{\top}+\sigma\sigma^{\top})-I\right]+\frac{\tilde{b}_n} {2}(\mathbf{1}\cdot \mathbf{1}^{\top}-\sigma\sigma^{\top})+(\tilde{a}_n^{-}-\tilde{a}_n^+)R \label{Abar}\\
		&=\frac{\tilde{a}_n^+ +\tilde{b}_n}{2}  \mathbf{1}\cdot \mathbf{1}^{\top} +\frac{\tilde{a}_n^+ -\tilde{b}_n}{2}  \sigma\sigma^{\top} +\left((\tilde{a}_n^{-}-\tilde{a}_n^+)R-\tilde{a}_n^+I\right).\label{Abar2} 
	\end{align}
	Since 
	\[\|(\tilde{a}_n^{-}-\tilde{a}_n^+)R-\tilde{a}_n^+I\|_2\leq |\tilde{a}_n^{-}-\tilde{a}_n^+|\cdot \|R\|_2+|\tilde{a}_n^+|= O(\log n/\sqrt n),
	\]
	by \eqref{Abar2}, we have
	\begin{align*}
	\|\overline{A}B^{(m-1)}x\|_2
	= &O\left(\frac{1}{n}\right)\|\mathbf{1}\cdot \mathbf{1}^{\top}B^{(m-1)}x \|_2+O\left(\frac{1}{n}\right)\|\sigma\sigma^{\top}B^{(m-1)}x\|_2
	+O\left(\frac{\log n}{\sqrt n}\right)\|B^{(m-1)}x\|_2.
	\end{align*}
	By Cauchy inequality,
	\begin{align*}
	\|\mathbf{1}\cdot \mathbf{1}^{\top}B^{(m-1)}x\|_2\leq \sqrt n	\|\mathbf{1}^{\top}B^{(m-1)}x\|_2, \quad  \|\sigma \sigma^{\top}B^{(m-1)}x\|_2\leq \sqrt n	\|\sigma^{\top}B^{(m-1)}x\|_2.
	\end{align*}
	 Therefore,
	\begin{align*}
		\|\overline{A}B^{(m-1)}x\|_2
		=&  O(n^{-1/2})(\|\sigma^{\top}B^{(m-1)}x\|_2+\|\mathbf{1}^{\top}B^{(m-1)}x\|_2)+O(\log n/\sqrt{n})\|B^{(m-1)}x\|_2.
	\end{align*}
	Using \eqref{bound1} and \eqref{bound2}, the right hand side in the expression above is upper bounded by 
	\begin{align}\label{467}
	O(\alpha^{(m-1)/2}\log n)+O(\|B^{(m-1)}x\|_2\cdot\log n/\sqrt n).
	\end{align}
	
	Since $B^{(m-1)}$ is a nonnegative matrix, the spectral norm is bounded by the maximum row sum (see Theorem 8.1.22 in \cite{horn2012matrix}),   we have that
	\[\|B^{(m-1)}x\|_2\leq \rho(B^{(m-1)})\leq \max_{i}\sum_{j=1}^n B_{ij}^{(m-1)}.
	\]
	By \eqref{claim1}, \eqref{Be} and \eqref{tanglecount}, the right hand side above is $O(\alpha^{m-1}\log n)$.
	Combing \eqref{467} and  noting that $\alpha^{m-1}/\sqrt n =o(n^{-1/4}),$ it implies 
	\begin{align}\label{eq:997}
\|\overline{A}B^{(m-1)}x\|_2=O(\alpha^{(m-1)/2}\log n)+O(\alpha^{m-1}\log^2 n/\sqrt{n})=O(\alpha^{(m-1)/2}\log n).
	\end{align}
	
	 Taking \eqref{eq:997} into \eqref{expanAB}, we have for any $\epsilon>0$, with high probability,
	$\|B^{(l)}x\|_2={O}(n^{\epsilon}\alpha^{l/2}\log^2 n)\leq n^{2\epsilon} \alpha^{l/2}
	$ for $n$ sufficiently large. This completes the proof.
\end{proof}

\section{Proof of Theorem \ref{coupling2}}
The proof in this section is a generalization of the method in  \cite{mossel2015reconstruction} for sparse random graphs. We now prove the case where $\sigma_i=+1$, and the case for $\sigma_i=-1$ can be treated in the same way. Recall the definition of $V_t$ from Definition \ref{DefVt}. Let $A_t$ be the event that no vertex in $V_{t}$ is connected by two distinct hyperedges to $V_{t-1}$. Let $B_t$ be the event that there does not exist  two vertices in $V_t$ that are contained in a hyperedge $  e\subset {V_{t}\choose d}$.

We can construct the multi-type Poisson hypertree $(T,\rho,\tau)$ in the following way.
For a vertex $v\in T$,
Let $Y_v^{(r)}, 0\leq r\leq d-1$ be the number of hyperedges incident to $v$ which among the remaining $d-1$ vertices, $r$ of them have the same spin with $\tau(v)$. We have 
\begin{align*}
Y_v^{(d-1)} \sim \text{Pois}\left(\frac{a}{2^{d-1}}\right),\quad 
Y_v^{(r)} \sim \text{Pois}\left(\frac{{{d-1}\choose{r}}b}{2^{d-1}}\right), 0\leq r\leq d-2	.
\end{align*}
Note that $(T,\rho,\tau)$ can be entirely reconstructed from the label of the root and the sequence $\{Y_v^{(r)}\}$ for $v\in V(T), 0\leq r\leq d-1$.

We define similar random variables for $(H,i,\sigma)$.
For a vertex $v\in V_t$, let $X_v^{(r)}$ be the number of hyperedges incident to $v$, where  all the remaining $d-1$ vertices are in $V_{t+1}$ such that $r$ of them  have spin $\sigma(v)$. 
Then we have
\begin{align*}X_v^{(d-1)} &\sim \text{Bin}\left({|V_{>t}^{\sigma(v)}|\choose d-1},\frac{a}{{{n}\choose{d-1}}}\right), \\
 X_v^{(r)} &\sim \text{Bin}\left({|V_{>t}^{\sigma(v)}|\choose r}{|V_{>t}^{-\sigma(v)}|\choose d-1-r},\frac{b}{{{n}\choose{d-1}}}\right) , \quad 0\leq r\leq d-2
\end{align*} and conditioned on $\mathcal F_{t}$ (recall the definition of $\mathcal F_t$ from \eqref{Ftsigma}) they are independent. Recall Definition \ref{def:spin_preserving}. We have the following lemma on the spin-preserving isomorphism. The proof of Lemma \ref{isom} is given in Appendix \ref{A7}. 

\begin{lemma}\label{isom} Let $(H,i,\sigma)_{t}, (T,\rho,\tau)_t$ be the rooted hypergraph truncated at distance $t$ from $i, \rho$ respectively. 
If \begin{enumerate}
	\item there is a spin-preserving isomorphism $\phi$ such that $(H,i,\sigma)_{t-1}\equiv (T,\rho,\tau)_{t-1}$,
	\item  for every $v\in V_{t-1}$, $X_v^{(r)}=Y_{\phi(v)}^{(r)}$ for $0\leq r\leq d-1$,
	\item $A_t,B_t$ hold,
\end{enumerate}
then $(H,i,\sigma)_t\equiv (T,\rho,\tau)_t$.
\end{lemma}

To make our notation simpler, for the rest of this section, we will identify $v$ with $\phi(v)$. Recall the event
$\Omega_{t}(i)=\{S_{t}(i)\leq C\log (n)\alpha^{t}\}
	$ where the constant $C$ is the same one as in Theorem \ref{quasi}.  Now define a new event
	\begin{align}
	C_t:=\bigcap_{s\leq t}\Omega_s(i).	
	\end{align}
	 From the proof of Theorem \ref{quasi}, for all $t\leq l$,  $\mathbb P_{\mathcal H_n}(C_t)=1-O(n^{-1-\gamma})$ for any $\gamma\in (0,3/8)$. Note that conditioned on $C_t$, there exists $C'>0$ such that
\begin{align}\label{sumbound}
|V_{\leq t}|\leq \sum_{s\leq t}C\log(n)\alpha^t\leq C'\log^2(n)\alpha^{t}.	
\end{align}

We now estimate the probability of event $A_t,B_t$ conditioned on $C_{t}$. The proof is included in Appendix \ref{A8}.
\begin{lemma}\label{AtBt}
	For any $t\geq 1$,
	\begin{align*}
	\mathbb P(A_t |C_{t})\geq 1-o(n^{-1/2}), \quad 
	\mathbb P(B_t |C_{t})\geq 1-o(n^{-1/2}).	
	\end{align*}
\end{lemma}	

Before  proving Theorem \ref{coupling2}, we also need the following bound on the total variation distance between binomial and Poisson random variables, see for example Lemma 4.6 in \cite{mossel2015reconstruction}.
\begin{lemma}\label{coupling} Let $m,n$ be integers and $c$ be a positive constant. The following holds:
	$$ \left\|\textnormal{Bin}\left(m,\frac{c}{n}\right)-\textnormal{Pois}(c)\right \|_{\textnormal{TV}}=O\left(\frac{1\vee |m-n|}{n}\right).$$
\end{lemma}

\begin{proof}[Proof of Theorem \ref{coupling2}]
	Fix $t$ and suppose that $C_{t}$ holds, and $(T,\rho)_t\equiv (H,i)_t$. Then for each $v\in V_t$, recall
	\begin{align*}
	X_v^{(d-1)} \sim \text{Bin}\left({|V_{>t}^{\sigma(v)}|\choose d-1},\frac{a}{{{n}\choose{d-1}}}\right),  \quad 
 X_v^{(r)} \sim \text{Bin}\left({|V_{>t}^{\sigma(v)}|\choose r}{|V_{>t}^{-\sigma(v)}|\choose d-1-r},\frac{b}{{{n}\choose{d-1}}}\right) 
\end{align*}
	and 
	\begin{align*}
Y_{v}^{(d-1)} \sim \text{Pois}\left(\frac{a}{2^{d-1}}\right), \quad 
Y_{v}^{(r)} \sim \text{Pois}\left(\frac{{{d-1}\choose{r}}b}{2^{d-1}}\right), \quad  0\leq r\leq d-2.	
\end{align*}
 Recall $|n^{\pm}-n/2|\leq \sqrt{n}\log n$. We have the following bound for $V_{>t}^{\pm}$: 
\begin{align*}
	 |V_{>t}^{\pm}|&\geq n^{\pm}-|V_{\leq t}|\geq \frac{n}{2}-\sqrt{n}\log(n)-O(\log^2(n)\alpha^{2t})\geq \frac{n}{2}-2\sqrt{n}\log (n),\\
	 |V_{>t}^{\pm}| &\leq n^{\pm}\leq \frac{n}{2}+\sqrt{n}\log(n).
\end{align*}
Therefore $|V_{>t}^{\pm}-\frac{n}{2}|\leq 2\sqrt{n}\log n$. Then from Lemma \ref{coupling},
\begin{align*}
\|X_v^{(d-1)}-Y_{v}^{(d-1)}\|_{\textnormal{TV}} &\leq C\frac{\left|{|V_{>t}^{\sigma(v)}|\choose d-1}-\frac{1}{2^{d-1}}{n\choose d-1}\right|}{\frac{1}{2^{d-1}}{n\choose d-1}}=O(n^{-1/2}\log n),\\
	\|X_v^{(r)}-Y_{v}^{(r)}\|_{\textnormal{TV}} &=O(n^{-1/2}\log n), \quad 0\leq r\leq d-2.
\end{align*}
We can couple $X_v^{(r)}$ with $Y_{v}^{(r)}, 0\leq r\leq d-1$ such that 
$  \mathbb P\left(X_v^{(r)}\not=Y_{v}^{(r)}\right)=O(n^{-1/2}\log n).$ 
Taking a union bound over all $v\in V_{t}$, and  $0\leq r\leq d-1$ and  recall \eqref{sumbound},  we can find a coupling such that with probability at least $$1-O(\log^3(n)\alpha^l n^{-1/2})\geq 1-o(n^{-1/4}),$$ 	$X_v^{(r)}=Y_{v}^{(r)}$ for every $v\in V_t$ and $0\leq r\leq d-1$. 

Lemma \ref{AtBt} implies $A_t,B_t, C_t$ hold simultaneously with probability at least $1-o(n^{-1/4})$.  Altogether we have that  assumptions (2),(3) in Lemma  \ref{isom} hold with probability $1-o(n^{-1/4})$, which can be written as
\begin{align*}
\mathbb P\left((H,i,\sigma)_{t+1}\equiv (T,\rho,\tau)_{t+1}, C_{t+1} \mathrel{\Big|}	(H,i,\sigma)_{t}\equiv (T,\rho,\tau)_{t}, C_{t}\right)\geq 1-o(n^{-1/4}).
\end{align*}
 Since we can certainly couple $i$ with $\rho$ from our construction, we have $
\mathbb P\left((H,i,\sigma)_{0}\equiv (T,\rho,\tau)_{0}, C_0 \right)=1.
$ Therefore  for large $n$,
\begin{align*}
	&\mathbb P((H,i,\sigma)_{l}\equiv (T,\rho,\tau)_{l}) \\
	=&\prod_{t=1}^{l}\mathbb P\left((H,i,\sigma)_{t}\equiv (T,\rho,\tau)_{t}, C_t \mathrel{\Big|}	(H,i,\sigma)_{t-1}\equiv (T,\rho,\tau)_{t-1}, C_{t-1}\right)\cdot \mathbb P\left((H,i,\sigma)_0\equiv (T,\rho,\tau)_{0}, C_{0}\right)\\
	\geq &(1-o(n^{-1/4}))^l\geq 1-n^{-1/5}. 
\end{align*}
 This completes the proof.
\end{proof}

\section{Proof of Theorem \ref{rama3}}\label{sec:rama3}

The proof  of the following Lemma \ref{rama1} follows in a similar way as  Lemma 4.4  in \cite{massoulie2014community}, and we include it in  Appendix \ref{A11}. 

\begin{lemma}\label{rama1}
	For $l=c\log(n), c\log(\alpha)<1/4$, the following hold asymptotically almost surely
	\begin{align}
	\|B^{(l)}\mathbf{1}-\vec{S}_l\|_2&=o(\|B^{(l)}\mathbf{1}\|_2) \label{B1},\\
	\|B^{(l)}\sigma-\vec{D}_l\|_2&=o(\|B^{(l)}\sigma\|_2)\label{B2},\\
	\langle B^{(l)}\mathbf{1}, B^{(l)}\sigma\rangle &=o\left(\|B^{(l)}\mathbf{1}\|_2 \cdot \|B^{(l)}\sigma \|_2\right)\label{B3}.
	\end{align}
\end{lemma}
The next lemma estimate $\|B^{(l)}x\|_2$ when $x=B^{(l)}\sigma$ and $B^{(l)}\mathbf{1}$. The proof of Lemma \ref{rama2} is provided in Appendix \ref{A12}.

\begin{lemma}\label{rama2}
	Assume $\beta^2>\alpha>1$ and $l=c\log(n)$ with $c\log(\alpha)<1/8$. Then for some fixed $\gamma>0$ asymptotically almost surely one has
	\begin{align}
		\Omega(\alpha^l)\|B^{(l)}\mathbf{1}\|_2 &\leq \|B^{(l)}B^{(l)}\mathbf{1}\|_2\leq O(\alpha^l\log n)\|B^{(l)}\mathbf{1}\|_2 \label{Bound1},\\
		\Omega(\beta^l)\|B^{(l)}\sigma\|_2 &\leq \|B^{(l)}B^{(l)}\sigma\|_2
		\leq {O}( n^{-\gamma}\alpha^l)
		\|B^{(l)}\sigma\|_2.
		\label{Bound2}
	\end{align}
\end{lemma}
Together with Lemma \ref{rama1} and Lemma \ref{rama2}, we are ready to prove Theorem \ref{rama3}.

\begin{proof}[Proof of Theorem \ref{rama3}]
	From Theorem \ref{thm:endsection} and Lemma \ref{rama2}, the top two eigenvalues of $B^{(l)}$ will be asymptotically in the span of $B^{(l)}\mathbf{1}$ and $B^{(l)}\sigma$. By the lower bound in \eqref{Bound1} and the upper bound in \eqref{Bound2}, the largest eigenvalue of $B^{(l)}$ will be ${\Theta}(\alpha^l)$ up to a logarithmic factor, and the first eigenvector is asymptotically aligned with  $B^{(l)}\mathbf{1}$.
	
	 From \eqref{B1}, $B^{(l)}\mathbf{1}$ is also asymptotically aligned with $\vec{S}_l$, therefore our statement for the first eigenvalue and eigenvector holds. Since $B^{(l)}\mathbf{1}$ and $B^{(l)}\sigma$ are asymptotically orthogonal from \eqref{B3}, together with \eqref{Bound2}, the second eigenvalue of $B^{(l)}$ is $\Omega(\beta^l)$ and the second eigenvector is asymptotically aligned with $B^{(l)}\sigma$.  
		 From \eqref{B2}, $B^{(l)}\sigma$ is asymptotically aligned with $\vec{D}_l$. So the statement for the second eigenvalue and eigenvector holds. The order  of other eigenvalues follows from Theorem \ref{thm:endsection} and  the Courant minimax principle (see \cite{horn2012matrix}).
\end{proof}

\begin{appendix}
\section{}\label{appB}

\subsection{Proof of Lemma \ref{approxind}}\label{sec:appendix_approx}\label{A3}

\begin{proof}
The two sequences $(U_{k}^{\pm}(i))_{k\leq l}$,  $(U_{k}^{\pm}(j))_{k\leq l}$ are independent conditioned on the event $\{V_{\leq l}(i)\cap  V_{\leq l}(j)=\emptyset\}$. It remains to estimate $\mathbb P_{\mathcal H_n}\left(\{V_{\leq l}(i)\cap  V_{\leq l}(j)=\emptyset\}\right)$.
Introduce the events $$ \mathcal J_k:=\bigcap_{t\leq k} \{S_t(i)\vee S_t(j)\leq C\log(n)\alpha^t \}, \quad \mathcal L_k:= \{ V_{\leq k}(i)\bigcap V_{\leq k}(j)=\emptyset \},$$
where the constant $C$ is the same one as in the statement of Theorem \ref{quasi}.  
For any vertex $v\in [n]\setminus (V_{\leq k}(i)\cup V_{\leq k}(j))$,  Conditioned on  $\mathcal L_k$ and $\mathcal J_k$,  there are  two possible situations where $v$ is included in $V_{k+1}(i)\cap V_{k+1}(j)$: 
\begin{enumerate}[(1)]
	\item There is a hyperedge containing $v$ and a vertex in $V_k(i)$, and a different hyperedge containing $v$ and a vertex in $V_k(j)$.
	\item There is a hyperedge containing $v$, one vertex in $V_{k}(i)$, and another vertex in $V_k(j)$.
\end{enumerate}
There exists a constant $C_1>0$ such that Case (1) happens with probability at most $$S_k(i)S_k(j){n\choose d-2}^2\left(\frac{a\vee b}{{n\choose d-1}}\right)^2\leq C_1 \log^2(n)\alpha^{2k}/n^2,$$ and Case (2) happens with probability at most $$S_k(i)S_k(j){n\choose d-3}\frac{a\vee b}{{n\choose d-1}}=C_1\log^2(n)\alpha^{2k}/n^2. $$ Since $\alpha^{2l}=n^{2c\log\alpha}=o(n^{1/2})$,  we have for large $n$,
$$\mathbb P_{\mathcal H_n}(v\in V_{k+1}(i)\cap V_{k+1}(j) \mid \mathcal J_k, \mathcal L_k)\leq 2C_1\log^2(n)\alpha^{2l}/n^2<n^{-1.5}. 
$$
Taking a union bound over all possible $v$, we have for some constant $C_3>0$,$$\mathbb P_{\mathcal H_n}(V_{k+1}(i)\cap V_{k+1}(j) =\emptyset \mid \mathcal J_k, \mathcal L_k)\geq 1-C_3n^{-1/2}. 
$$

From the proof of Theorem \ref{quasi}, for all $0\leq k\leq l$, $\mathbb P_{H_n}(\mathcal J_k)=1-O(n^{-1-\gamma})$  for any $\gamma\in (0,3/8)$. We then have 
\begin{align*} 
\mathbb P_{\mathcal H_n}(V_{k+1}(i)\cap V_{k+1}(j) =\emptyset \mid  \mathcal L_k)
\geq & \mathbb P_{\mathcal H_n}(V_{k+1}(i)\cap V_{k+1}(j) =\emptyset \mid  \mathcal J_k, \mathcal L_k)~\mathbb P_{\mathcal H_n}(\mathcal J_k)\geq 1-O(n^{-1/2}). 
\end{align*}
Finally, for large $n$,
\begin{align*}\mathbb P_{\mathcal H_n}\left(\{V_{\leq l}(i)\cap  V_{\leq l}(j)=\emptyset\}\right)
 =&
\mathbb P_{\mathcal H_n}(\mathcal L_l) \geq \mathbb P_{\mathcal H_n}(V_l(i)\cap V_l(j)=\emptyset \mid \mathcal L_{l-1})	\mathbb P_{\mathcal H_n}( \mathcal L_{l-1})\\
 \geq & \mathbb P_{\mathcal H_n}(\mathcal L_0)\prod_{k=0}^{l-1}\mathbb P_{\mathcal H_n}(V_{k+1}(i)\cap V_{k+1}(j)=\emptyset \mid  \mathcal L_{k})\\
                 \geq & (1-O(n^{-1/2}))^l\geq 1-n^{-1/3}.
\end{align*}
This completes the proof.
\end{proof}

\subsection{Proof of Lemma \ref{tangle}}\label{A4}
\begin{proof} Consider the exploration process of the neighborhood of a fixed vertex $i$. Conditioned on $\mathcal F_{k-1}$,
	there are two ways to create  new cycles in $V_{\geq k-1}(i)$: 
	\begin{enumerate}
	\item 	Type 1: a new hyperedge $e\subset V_{\geq k-1}(i)$ containing two vertices in $V_{k-1}(i)$ may appear, which creates a cycle including two vertices in $V_{k-1}(i)$.
	\item Type 2: two vertices in $V_{k-1}(i)$ may be connected to the same vertex in $V_{\geq k}(i)$ by two new distinct hyperedges.

	\end{enumerate}
 	Define the event
	\begin{align}\label{event}
	\Omega_{k-1}(i):=\{S_{k-1}(i)\leq C\log (n)\alpha^{k-1}\},
	\end{align}
	where the constant $C$ is the same one as in Theorem \ref{quasi}. From the proof of Theorem \ref{quasi}, $  \mathbb P_{\mathcal H_n}(\Omega_k(i))=1-O(n^{-1-\gamma})$ for some $\gamma\in (0,3/8)$. 	Let $E_k^{(1)}(i)$ be the number of hyperedges of type 1. Conditioned on $\mathcal F_{k-1}$, $E_k^{(1)}(i)$ is stochastically dominated by
	$\textnormal{Bin}\left({{S_{k-1}(i)}\choose{2}}{{n}\choose{d-2}},\frac{a\vee b }{{{n}\choose{d-1}}}\right).$ Then for some constant $C_1>0$, 
	$$\mathbb E_{\mathcal H_n}[E_{k}^{(1)}(i)\mid \Omega_{k-1}(i)]\leq C_1\log^2(n)\alpha^{2k-2}/n\leq C_1\log^2(n)\alpha^{2l}/n.$$
	By Markov's inequality,
	\begin{align*}
	 \mathbb P_{\mathcal H_n}(\{E_k^{(1)}(i)\geq 1\})
	\leq &\mathbb P_{\mathcal H_n}(\{E_k^{(1)}(i)\geq 1\}\mid \Omega_{k-1}(i))+\mathbb P_{\mathcal H_n}(\Omega_{k-1}^c(i))\\
	\leq &\mathbb E_{\mathcal H_n}[E_{k}^{(1)}(i)\mid \Omega_{k-1}(i)]+O(n^{-1-\gamma})=O(\log^2(n)\alpha^{2l}/n).
	\end{align*}
 Taking the union bound, the probability that there is  a type 1 hyperedge in the $l$-neighborhood of $i$ is 
 \begin{align*}
 \mathbb P_{\mathcal H_n}\left(\bigcup_{k=1}^{l}\{E_k^{(1)}(i)\geq 1\}\right)
 &\leq \sum_{k=1}^l\mathbb P_{\mathcal H_n}(\{E_k^{(1)}(i)\geq 1\})=O(\log^3(n)\alpha^{2l}/n).
 \end{align*}

 The number of hyperedge pair $(e_1,e_2)$ of Type 2 is stochastically dominated by 
	$$\textnormal{Bin}\left(nS_{k-1}^2{{n}\choose{d-2}}^2,\left(\frac{a\vee b }{{{n}\choose{d-1}}}\right)^2\right), $$  
	which conditioned on $\Omega_{k-1}(i)$ has expectation $O(\log^2(n)\alpha^{2l}/n)$. By a Markov's inequality and a union bound, in the same way as the proof for Type 1, we have the probability there is a type 2 hyperedge pair in the $l$-neighborhood of $i$ is $O(\log^2(n)\alpha^{2l}/n)$. Altogether the probability that there are at least one cycles within the  $l-$neighborhood of $i$  is $O(\log^3(n)\alpha^{2l}/n)$.

	Let $Z_i$ be the random variable such that $Z_i=1$ if $l$-neighborhood of $i$ contains one cycle and $Z_i=0$ otherwise.
	From the analysis above, we have $\mathbb E[Z_i]=O(\log^3(n)\alpha^{2l}/n).$ 
By Markov's inequality, 
	\begin{align*}
	&\mathbb P_{\mathcal H_n}\left(\sum_{i\in [n]} {Z_i} \geq \alpha^{2l}\log^4(n)\right)\leq \frac{\sum_{i}\mathbb E[Z_i]}{\log^4(n)\alpha^{2l}}
		= \frac{O(\log^3(n)\alpha^{2l})}{\alpha^{2l}\log^4(n)}=O(\log^{-1}(n)).
	\end{align*} 
	Then asymptotically almost surely the number of vertices whose $l$-neighborhood contains one cycle at most $\log^4(n)\alpha^{2l}$. 
It remains to show $H$ is $l$-tangle free asymptotically almost surely.  For a fixed vertex $i\in[n]$, there are several possible cases where there can be two cycles  in $V_{\leq l}(i)$.

(1) There is  one hyperedge of Type 1 or a hyperedge pair of Type 2 which creates more than one cycles. We discuss in the following cases conditioned on  the event $\cap_{t=1}^l \Omega_t(i)$.
		\begin{enumerate}[(a)]
			\item The number of hyperedge of the first type which connects to more than two vertices in $V_{k-1}$ is stochastically dominated by 
	$\textnormal{Bin}\left({{S_{k-1}}\choose{3}}{{n}\choose{d-3}},\frac{a\vee b }{{{n}\choose{d-1}}}\right).$ 
	The expectation is at most $O(\alpha^{3l}\log^3(n)/n^2)$.
	\item  If the intersection of the hyperedge pair of Type 2 contains  $2$ vertices in $V_{\geq k}$, it will create two cycles. The number of such hyperedge pairs is stochastically dominated by 		
		$\textnormal{Bin}\left({n\choose 2}S_{k-1}^2{{n}\choose{d-3}}^2,\left(\frac{a\vee b }{{{n}\choose{d-1}}}\right)^2\right)$ 
		with mean $O(\log^2 (n)\alpha^{2l}/n^2)$.
		\end{enumerate}
		
Then by Markov's inequality and a union bound, asymptotically almost surely, there is no $V_{\leq l}(i)$ such that its neighborhood contains Type 1 hyperedges or Type 2 hyperedge pairs which create more than one cycles. 

(2) The remaining case is that there is a $V_{\leq l}(i)$ where  two cycles are created by two Type 1 hyperedges or two Type 2  hyperedge pairs or one Type 1 hyperedge and another hyperedge pairs. By the same argument, under the event $\cap_{t=1}^l \Omega_t(i)$, the probability that  such event happens is $O(\log^6(n)\alpha^{4l}/n^2)$. 
Since $\alpha^{4l}=o(n)$, by taking a union bound over $i\in [n]$, we have $H$ is $l$-tangle-free asymptotically almost surely.
\end{proof}

\subsection{Proof of Lemma \ref{TangleCount}}\label{sec:TangleCount}\label{A5}
\begin{proof}
	Let $i\not\in \mathcal B$ whose $l$-neighborhood contains no cycles. For any $k\in[n]$ and any $m\leq l$, there is a unique self-avoiding walk of length $m$ from $i$ to $k$ if and only if $d(i,k)=m$, so we have $B_{ik}^{(m)}=\mathbf{1}_{d(i,k)=m}.
	$ For such $i$ we have
	$$
	(B^{(m)}\mathbf 1)_i=S_m(i),\quad 
	(B^{(m)}\sigma)_i=D_m(i).	
	$$
 Then \eqref{Be}, \eqref{Bs} follows from Theorem \ref{quasi}. By Lemma \ref{tangle}, asymptotically almost surely all vertices in $\mathcal B$ have only one cycle in $l$-neighborhood. For any $m\leq l, i\in \mathcal B$, since 
$(B^{(m)}\mathbf 1)_i=\sum_{k\in [n]}B_{ik}^{(m)}
$, and 
only vertices at distance at most $m$ from $i$ can be reached by a self-avoiding walk of length $m$ from $i$, which will  be counted in $(B^{(m)}\mathbf{1})_i$. Moreover, for any $k\in [n]$ with $B_{ik}^{(m)}\not=0$, since the $l$-neighborhood of $i$ contains at most one cycle, there are at most 2 self-avoiding walks of length $m$ between $i$ and $k$. Altogether we know $$\sum_{k\in [n]}B_{ik}^{(m)}\leq 2\sum_{t=0}^m S_t(i)=O(\alpha^m \log n)$$ asymptotically almost surely.  Then \eqref{tanglecount} follows.
\end{proof}

\subsection{Proof of Lemma \ref{lem:martingale}}
\begin{proof}
Recall the definitions of $\alpha,\beta$ from \eqref{alphabeta}. From \eqref{Pois1}-\eqref{Pois3},
\begin{align*}
	\mathbb E(W_{t+1}^{+}|\mathcal G_t)&=\sum_{r=0}^{d-1}r\mathbb E(W_{t+1}^{(r)}|\mathcal G_t)=\sum_{r=1}^{d-2} r\left(\frac{b{{d-1}\choose{r}} }{2^{d-1}}(W_{t}^{-}+W_{t}^{+})\right)+(d-1)\left(\frac{a}{2^{d-1}}W_{t}^{+} + \frac{b}{2^{d-1}}W_{t}^{-}\right)\\
	&=\frac{\alpha+\beta}{2} W_t^+  +\frac{\alpha-\beta}{2}W_t^-=\frac{\alpha^{t+1}}{2}M_t+\frac{\beta^{t+1}}{2}\Delta_t.
\end{align*} 
Similarly,
$ 
	\mathbb E[W_{t+1}^{-}|\mathcal G_t]=\frac{\alpha^{t+1}}{2}M_t-\frac{\beta^{t+1}}{2}\Delta_t.
$
Therefore 
\begin{align*}
\mathbb E[M_{t+1} \mid \mathcal G_t]= \alpha^{-t-1} \mathbb E[W_{t+1}^+ +W_{t+1}^- \mid \mathcal G_t]=M_t,\\
	\mathbb E[\Delta_{t+1} \mid \mathcal G_t]= \beta^{-t-1} \mathbb E[W_{t+1}^+ -W_{t+1}^- \mid \mathcal G_t]=\Delta _t.
\end{align*}
It follows that $\{M_t\}, \{\Delta_t\}$ are martingales with respect to $\mathcal G_t$. 
From \eqref{Pois1}-\eqref{Pois5},
\begin{align*}
	\textnormal{Var}(M_t|\mathcal G_{t-1})&=\textnormal{Var}(\alpha^{-t}(W_t^++W_t^- )|\mathcal G_{t-1})=\alpha^{-2t}\textnormal{Var}\left((d-1)\sum_{r=0}^{d-1}W_t^{(r)}|\mathcal G_{t-1}\right)\\
	 &=(d-1)^2\alpha^{-2t}\cdot \frac{\alpha}{d-1} (W_{t-1}^+ +W_{t-1}^-)=(d-1)\alpha^{-t}M_{t-1}.
\end{align*} 
Sine $\mathbb EM_0=1$,  by conditional variance formula,
\begin{align*}
\textnormal{Var}(M_t)&=\textnormal{Var}(\mathbb E[M_t|\mathcal G_{t-1}])+\mathbb E\textnormal{Var}(M_t|\mathcal G_{t-1})=\textnormal{Var}(M_{t-1})+(d-1)\alpha^{-t}.
\end{align*}
Since $\textnormal{Var}(M_0)=0$, we have for $t\geq 0$, $ 
	\textnormal{Var}(M_t)=(d-1)\frac{1-\alpha^{-t}}{\alpha-1}
$.
So  $\{M_t\}$ is uniformly integrable for $\alpha>1$. Similarly, 
\begin{align*}
	\textnormal{Var}(\Delta_t|\mathcal G_{t-1})&=\textnormal{Var}(\beta^{-t}(W_t^+-W_t^- )|\mathcal G_{t-1})=\beta^{-2t}\sum_{r=0}^{d-1}(2r-d+1)^2\textnormal{Var}(W_t^{(r)}|\mathcal G_{t-1}) \\
	 &=(\alpha/\beta^2)^t M_{t-1}(d-1)\alpha^{-1}\cdot  \frac{(d-1)a+(2^{d-1}+1-d)b}{2^{d-1}}=:\kappa(\alpha/\beta^2)^t M_{t-1},
\end{align*}
where $  \kappa:=\frac{(d-1)(a-b)+2^{d-1}b}{a+(2^{d-1}-1)b}$. And we also have the following recursion: 
\begin{align*}
\textnormal{Var}(\Delta_t)&=\textnormal{Var}(\mathbb E[\Delta_t|\mathcal G_{t-1}])+\mathbb E\textnormal{Var}(\Delta_t|\mathcal G_{t-1})=\textnormal{Var}(\Delta_{t-1})+\kappa \beta^{-2t}\alpha^t.
\end{align*}
Since $\textnormal{Var}(\Delta_{0})=0$, we have for $t>0$, 
\begin{align}\label{Var}
\textnormal{Var}(\Delta_{t})=\kappa\cdot \frac{1-(\beta^2/\alpha)^{-t}}{\beta^2/\alpha-1}.	
\end{align}
So  $\{\Delta_t\}$ is uniformly integrable if $\beta^2>\alpha$.  
From the martingale convergence theorem, $\mathbb E \Delta_{\infty}=\Delta_0=1$, $\textnormal{Var}(\Delta_{\infty})=\frac{\kappa}{\beta^2/\alpha-1}$, and \eqref{l2} holds. This finishes the proof.
\end{proof}

\subsection{Proof of Lemma \ref{thresholding}} \label{sec:thresh}
\begin{proof}
	From Theorem \ref{coupling2}, For each $i\in [n]$, there exists a coupling such that with probability $1-O(n^{-\epsilon})$ for some positive $\epsilon$,  $  \beta^{-l}\sigma(i)D_l(i)=\Delta_{l}$ and we denote this event by $\mathcal C$. When the coupling fails, by Theorem \ref{quasi}, $\beta^{-l}\sigma(i)D_l(i)=O(\log(n))$ with probability $1-O(n^{-\gamma})$ for some $\gamma>0$. 	Recall the event
	\begin{align}	\Omega_{k-1}(i):&=\{S_{k-1}(i)\leq C\log (n)
	\alpha^{k-1}\}.
	\end{align} 
	We define 
	$  \Omega:=\bigcap_{i=1}^n \Omega(i),  \Omega(i):=\bigcap_{k\leq l}\Omega_k(i).$ We have 
	\begin{align}\label{mean}
		\mathbb E\left(\frac{1}{n}\sum_{i=1}^n \beta^{-2l}D_l^2(i)\mid \Omega \right)=O(\log^2(n))n^{-\epsilon}+\mathbb E(\Delta_l^2\mathbf{1}_{\mathcal C} \mid \Omega).
	\end{align}
	Moreover,
\begin{align}
|\mathbb E(\Delta_l^2\mathbf{1}_{\mathcal C}|\Omega)-\mathbb E(\Delta_{\infty}^2)|
=&\left| \frac{\mathbb E(\Delta_l^2\mathbf{1}_{\mathcal C}-\mathbb E(\Delta_l^2\mathbf{1}_{\mathcal C}\mathbf{1}_{\overline{\Omega}})-\mathbb P(\Omega)\mathbb E(\Delta_{\infty}^2)}{\mathbb P(\Omega)}\right| \notag\\
\leq &\frac{|\mathbb E(\Delta_l^2-\Delta_{\infty}^2)|}{\mathbb P(\Omega)}+\frac{1-\mathbb P(\Omega)}{\mathbb P(\Omega)}\mathbb E(\Delta_{\infty}^2)+\frac{|\mathbb E(\Delta_l^2\mathbf{1}_{\overline{\mathcal C}})-\mathbb E(\Delta_l^2\mathbf{1}_{\mathcal C\cap \overline{\Omega}})|}{\mathbb P(\Omega)}.\label{3sum}
\end{align}

Since we know $\mathbb P(\Omega\cap \mathcal C)\to 1$ and  \eqref{l2},  the first two terms in \eqref{3sum} converges to $0$. The third term also converges to $0$ by dominated convergence theorem. So we have
\begin{align*}
		\mathbb E\left(\frac{1}{n}\sum_{i=1}^n \beta^{-2l}D_l^2(i)\mid \Omega \right)\to \mathbb E(\Delta_{\infty}^2).
	\end{align*}
	
	We then estimate the second moment. Note that
\begin{align} 
\mathbb E\left(\frac{1}{n}\sum_{i=1}^n \beta^{-2l}D_l^2(i)\mid \Omega \right)^2 
=&\frac{1}{n^2}\mathbb E\left(\sum_{i=1}^n\beta^{-4l}D_l^4(i) \mid \Omega\right)+\frac{2}{n^2}\sum_{i<j}\beta^{-4l}\mathbb E(D_l(i)^2D_l^2(j) \mid \Omega), \label{varianceSUM}
\end{align}
and from Theorem \ref{quasi}, the first term   is $O(\log^4(n)/n)=o(1)$. Next, we show the second  term  satisfies
\begin{align}
\frac{2}{n^2}\sum_{i<j}\beta^{-4l}\mathbb E(D_l(i)^2D_l^2(j) \mid \Omega)=\frac{2}{n^2}\sum_{i<j}\beta^{-4l}\frac{1}{\mathbb P(\Omega)}\mathbb E(\mathbf{1}_{\Omega}D_l(i)^2D_l^2(j))=o(1).\label{varianceSUM2}
\end{align}  
Since $\mathbb P(\Omega)=1-O(n^{-\gamma})$, it suffices to show

$$\frac{2}{n^2}\sum_{i<j}\beta^{-4l}\mathbb E(\mathbf{1}_{\Omega}D_l(i)^2D_l^2(j))=o(1).
$$

Consider 
$ 
\beta^{-4l}\mathbb E(\mathbf{1}_{\Omega(i)\cap\Omega(j)}D_l^2(i)D_l^2(j)).
$
From Lemma \ref{approxind}, when $i\not=j$, $D_l(i), D_l(j)$ are asymptotically independent. On the event that the coupling with independent copies fails (recall the failure probability is $O(n^{-\gamma})$), we bound $D_l^2(i)D_l^2(j)$ by $O(\beta^{4l}\log^4(n))$. When the coupling succeeds,  $$  \beta^{-4l}\mathbb E(\mathbf{1}_{\Omega(i)\cap\Omega(j)}D_l(i)^2D_l^2(j))=  \beta^{-4l}\mathbb E(\mathbf{1}_{\Omega(i)}D_l(i)^2)\mathbb E(\mathbf{1}_{\Omega(j)}D_l(j)^2).$$ Then from \eqref{inprob},
\begin{align}
 \frac{2}{n^2}	\sum_{i<j}\beta^{-4l}\mathbb E(\mathbf{1}_{\Omega(i)\cap\Omega(j)}D_l(i)^2D_l^2(j))\
=&O\left(\frac{1}{n^2}\sum_{i<j}\beta^{-4l}\mathbb E(\mathbf{1}_{\Omega(i)}D_l(i)^2)\mathbb E(\mathbf{1}_{\Omega(j)}D_l(j)^2)+O(n^{-2\gamma}\log^4n)\right) \notag\\
=&O\left((\mathbb E(\Delta_{\infty}^2))^2 \right)=O(1).\label{varianceSUM3}
\end{align}

Therefore from \eqref{varianceSUM}, \eqref{varianceSUM2}, and \eqref{varianceSUM3}, $$\mathbb E\left(\frac{1}{n}\sum_{i=1}^n \beta^{-2l}D_l^2(i)\mid \Omega \right)^2=O(1).$$

With \eqref{mean}, by Chebyshev's inequality, conditioned on $\Omega$, in probability we have
$$\lim_{n\to\infty}\frac{1}{n}\sum_{i=1}^n \beta^{-2l}D_l^2(i)=\mathbb E(\Delta_{\infty}^2).
$$ Since $\mathbb P(\Omega)\to 1$, \eqref{inprob} follows. 

We now establish \eqref{indicator}. Without loss of generality, we discuss the case of $+$ sign. Since $\tau$ is a continuous point of the distribution of $\Delta_{\infty}$, for any fixed $\delta>0$, we can find two bounded $K$-Lipschitz function $f,g$ for some constant $K>0$ such that 
$$f(x)\leq (\mathbf{1}_{x\geq\tau})\leq g(x), x\in\mathbb R,\quad 
0\leq \mathbb E(g(\Delta_{\infty})-f(\Delta_{\infty}))\leq \delta.
$$
Consider the empirical sum $  \frac{1}{n}\sum_{i\in \mathcal N^+}f(x^{(n)}_i\sqrt{n\mathbb E(\Delta_{\infty}^2})$, we have 
\begin{align*}
	&\left|\frac{1}{n}\sum_{i\in \mathcal N^+}f(x_i^{(n)}\sqrt{n\mathbb E\Delta_{\infty}^2})-\frac{1}{n}\sum_{i\in \mathcal N^+}f(\beta^{-l}D_l(i))\right|\\
	\leq &\frac{K}{n}\sum_{i\in\mathcal N^+}|(x_i^{(n)}-y_i^{(n)})\sqrt{n\mathbb E\Delta_{\infty}^2}|+\frac{K}{n}\sum_{i\in\mathcal N^+}|y_i^{(n)}\sqrt{n\mathbb E\Delta_{\infty}^2}-\beta^{-l}D_l(i)|.
\end{align*}
The first term converges to $0$ by the assumption that $\|x-y\|_2\to 0$ in probability. The second term converges to $0$ in probability from \eqref{inprob}. Moreover,
$\frac{1}{n}\sum_{i\in \mathcal N^+}f(\beta^{-l}D_l(i))$ converges in probability to $\frac{1}{2}\mathbb Ef(\Delta_{\infty})$. So we have 
\begin{align*}
\lim_{n\to\infty}\frac{1}{n}\sum_{i\in \mathcal N^+}f(x_i^{(n)}\sqrt{n\mathbb E\Delta_{\infty}^2})=\frac{1}{2}\mathbb Ef(\Delta_{\infty}), 
\end{align*} and the same holds for $g$. If follows that 
\begin{align*}
\limsup_{n\to\infty}\left|\frac{1}{n}\sum_{i\in[n]:\sigma_i=+}\mathbf{1}_{\left\{x_i^{(n)}\geq\tau/\sqrt{n\mathbb E[\Delta_{\infty}^2]}\right\}}-\frac{1}{2}\mathbb P(\Delta_{\infty}\geq \tau)\right|\leq \delta
\end{align*}
for any $\delta>0$. Therefore \eqref{indicator} holds.
\end{proof}

\subsection{Proof of Lemma \ref{lemma:bound_on_gamma}}\label{sec:appendix_moment}
\begin{proof}
     For any $n\times n$ real matrix $M$, we have $\rho(M)^{2k}\leq \textnormal{tr}[(MM^{\top})^k]$, therefore
\begin{align}\label{gammalm}
\mathbb E_{\mathcal H_n}[\rho(\Gamma^{(l,m)})^{2k}] &\leq \mathbb E_{\mathcal H_n}\left[\textnormal{tr}\left(\Gamma^{(l,m)}{\Gamma^{(l,m)}}^{\top}\right)^k\right] \\
 &=\sum_{i_1,\dots, i_{2k}\in [n]}\mathbb E_{\mathcal H_n} \left[\Gamma_{i_1 i_2}^{(l,m)}\Gamma_{i_3 i_2}^{(l,m)}\dots \Gamma_{i_{2k-1}i_{2k}}^{(l,m)}\Gamma_{i_{1}i_{2k}}^{(l,m)}\right]. \notag
\end{align}

 Recall the definition of $\Gamma^{(l,m)}_{ij}$ from \eqref{Gammaij}, the sum in \eqref{gammalm} can be expanded to be the sum over all circuits $w=(w_1,\dots w_{2k})$ of length $2kl$ which are obtained by concatenation of $2k$ walks of length $l$, and each $w_i, 1\leq i\leq 2k$ is a concatenation of two self-avoiding walks of length  $l-m$ and $m-1$. The weight that each hyperedge in the circuit contributes  can be either $ A_{ij}^e-\overline{A^e_{ij}}, \overline{A^e_{ij}}$ or $A_{ij}^e$. For all circuits $w$ in \eqref{gammalm} with nonzero expected weights, there is an extra constraint that  each $w_i$ intersects with some other $w_j$, otherwise the expected weight that $w_i$ contributes to the sum \eqref{gammalm} will be $0$. We want to bound the number of such circuits with nonzero expectation.

Let $v,h$ denoted the number of distinct vertices and hyperedges traversed by the circuit. Here we don't count the hyperedges that are weighted by $\overline{A_{ij}^e}$. We associate a multigraph $G(w)$ for each $w$ as before, but the hyperedges with weight $\overline{A_{ij}^e}$ are not included. Since $\mathbb E_{\mathcal H_n}[\Gamma_{ij}^{(l,m)}]=0$ for any $i,j\in[n]$, if the expected weight of $w$ is nonzero, the corresponding graph  $G(w)$ must be connected.

We detail the proof for circuits in Case (1), where \begin{itemize}
    \item each hyperedge label in $\{e_i\}_{1\leq i\leq h}$ appears exactly once on $G(w)$; 
    \item vertices in $e_i\setminus \textnormal{end}(e_i)$ are all distinct for $1\leq i\leq h$, and they are not vertices with labels in $V(w)$,
\end{itemize} 
and the cases from other circuits follow similarly following the proof of Lemma \ref{expectboundmoment}.

 Let $m$ be fixed. For each circuit $w$, there are $4k$  self-avoiding walks, and each $w_i$ is broken into two   self-avoiding walks of length $m-1$ and $l-m$ respectively. We adopt the way of encoding each self-avoiding walk as before, except that we must also include the labels of the endpoint $j$  after the traversal of an edge $e$ with weight from $\overline{A_{ij}^e}$, which gives us the initial vertex of the self-avoiding walk of length $l-m$ within each $w_i$. These extra labels tell us how to concatenate the two self-avoiding walks of length $m-1$ and $l-m$ into the walk $w_i$ of length $l$. For each $w_i$,  label is encoded by a number from $\{1,\dots, v\}$. So all possible such labels can be bounded by $v^{2k}$. 
 Then the upper bound on the number of valid triplet sequences with extra labels for fixed $v,h$ is now given by 
$
 v^{2k}[(v+1)^2(l+1)]^{4k(2+h-v)}$.

 The total number of circuits that have the same triplet sequences with extra labels is at most $ n^v{{n}\choose{d-2}}^{h+2k}$
    where $h+2k$ is the total number of distinct hyperedges we can have in $w$, including the hyperedges with weights from $\overline{A_{ij}^e}$. 
    Combining the two estimates above, the number of all circuits $w$ with given $v,h$ is upper bounded by 
  $
    	 n^v{{n}\choose{d-2}}^{h+2k}v^{2k}[(v+1)^2(l+1)]^{4k(2+h-v)}.
   $
   
   We also need to bound the possible range of $v,h$. There are overall $2k(l-1)$ hyperedges traversed in $w$  (remember we don't count the edges with weights from $\overline{A_{ij}^e}$). Out of these, $2k(l-m)$ hyperedges (with multiplicity) with weights coming from $A_{ij}^e-\overline{A_{ij}^e}$ must be at least doubled for the expectation not to vanish.  Then the number of distinct hyperedges in $w$  excluding the hyperedge weighted by some $\overline{A_{ij}^e}$,  satisfies 
$h\leq k(l-m)+(2k(l-1)-2k(l-m))=k(l+m-2).$

We have  $v\geq \max\{m,l-m+1\}$ since each self-avoiding walk of length $m-1$ or $l-m$ has distinct vertices. 
Moreover, since $G(w)$ is connected, $h\geq v-1$, so we have $v-1\leq h\leq k(l+m-2)$. And the range of $v$ is then given by  $ \max\{m,l-m+1\}\leq v\leq k(l+m-2)+1$. 

The expected weight that a circuit contributes can be  estimated similarly as before. From \eqref{treebound}, the expected weights from $v-1$ many hyperedges that corresponds to edges on $T(w)$ is bounded by $ \left(\frac{\alpha}{(d-1){{n}\choose{d-1}}}\right	)^{v-1}$. Similar to \eqref{cyclebound}, the expected weights from $h-v+1+2k$ many hyperedges that corresponds to edges on $G(w)\setminus T(w)$ together with hyperedges whose weights are from $\overline{A_{ij}^e}$ is bounded by $ \left(\frac{a\vee b}{{{n}\choose{d-1}}}\right)^{h-v+1+2k}$. 
Putting all estimates together, for fixed $v,h$, the total contribution to the sum is bounded by 
\begin{align*}
	&n^v{{n}\choose{d-2}}^{h+2k} v^{2k}[(v+1)^2(l+1)]^{4k(2+h-v)} \left(\frac{\alpha}{(d-1){{n}\choose{d-1}}}\right	)^{v-1}\left(\frac{a\vee b}{{{n}\choose{d-1}}}\right)^{h-v+1+2k}\\
	=&n^v\left(\frac{\alpha}{d-1}\right	)^{v-1}\left(\frac{d-1}{n-d+2}\right)^{h+2k}  v^{2k}~ Q(k,l,v,h),	   
\end{align*}
where $Q(k,l,v,h):=[(v+1)^2(l+1)]^{4k(2+h-v)}\left(a\vee b \right)^{h-v+1+2k}.$ Let $S_1$ be the contribution of circuits in Case (1) to the sum in \eqref{gammalm}.
Summing over all possible $v$ and $h$, we have
\begin{align}\label{eq:A10}
S_1\leq &\sum_{v=m\vee (l-m+1)}^{k(l+m-2)+1}\sum_{h=v-1}^{k(l+m-2)} n^v\left(\frac{\alpha}{d-1}\right	)^{v-1}\left(\frac{d-1}{n-d+2}\right)^{h+2k}  v^{2k}~ Q(k,l,v,h). 	
\end{align}
Taking $l=O(\log n)$, similar to the discussion in \eqref{ratio},   the leading term in \eqref{eq:A10} is given by the term with $h=v-1$. So  for any $1\leq m\leq l$, and sufficiently large $n$, there are constants $C_1,C_2>0$ such that 
\begin{align*} 
	S_1 
	\leq & \sum_{v=m\vee (l-m+1)}^{k(l+m-2)+1}2n^{1-2k}((d-1)v)^{2k}[(v+1)^2(l+1)]^{4k}\alpha^{v-1}\left(a\vee b \right)^{2k}\\
	\leq & 2\sum_{v=m\vee (l-m+1)}^{k(l+m-2)+1}n^{1-2k}\alpha^{v-1}[(v+1)^5(l+1)^2(d-1)(a\vee b)]^{2k}
	\\
	\leq & C_1\log^{14k} (n)\cdot n^{1-2k}\sum_{v=m\vee (l-m+1)}^{k(l+m-2)+1}\alpha^{v-1}
	 \leq C_2\log^{14k} (n)\cdot n^{1-2k}\alpha^{k(l+m-2)}.
		\end{align*}	
For circuits not in Case (1), similar to the proof of Lemma 	\ref{expectboundmoment}, their total contribution is bounded by $C_2'  n^{1-2k}\alpha^{k(l+m-2)}\log^{14k}n$ for a constant $C_2'>0$. This completes the proof of Lemma \ref{lemma:bound_on_gamma}.
 \end{proof}

 \subsection{Proof of Lemma \ref{Cor412}}
\label{A6}

\begin{proof}
Let $\mathcal B$ be the set of vertices such that their $l$-neighborhood contains a cycle. Let $x$ be a normed vector such that $x^{\top}B^{(l)}\mathbf{1}=0$. We then have
\begin{align}
\mathbf{1}^{\top}B^{(m-1)}x=&\sum_{i\in [n]}x_i(B^{(m-1)}\mathbf{1})_i
              =\sum_{i\not\in \mathcal B} x_iS_{m-1}(i)+\sum_{i\in\mathcal B}x_i(B^{m-1}\mathbf{1})_i \notag\\
              =&\sum_{i\in [n]}x_i(\alpha^{m-1-l}(B^{(l)}\mathbf{1})_i+{O}(\alpha^{\frac{m-1}{2}}\log n)) \notag\\
              &-\sum_{i\in \mathcal B}x_i(\alpha^{m-1-l}(B^{(l)}\mathbf{1})_i+{O}(\alpha^{\frac{m-1}{2}}\log n)) 
                          +\sum_{i\in \mathcal B}x_i(B^{(m-1)}\mathbf{1})_i.\label{lastterm}
\end{align}

Since we have $\mathbf{1}^{\top}B^{(l)}x=0$, the first term in \eqref{lastterm} satisfies 
\begin{align*}
 \left|\sum_{i\in [n]}x_i(\alpha^{m-1-l}(B^{(l)}\mathbf{1})_i+ {O}(\alpha^{\frac{m-1}{2}}\log n))\right|
=&\left|\sum_{i\in[n]}x_i {O}(\alpha^{\frac{m-1}{2}}\log n)\right|= {O}(\sqrt{n}\alpha^{\frac{m-1}{2}}\log n),
\end{align*} where the last inequality above is from Cauchy inequality.  From Lemma \ref{tangle}, $|\mathcal B|={O}(\alpha^{2l}\log^4 n)$.
For the second term in \eqref{lastterm},   recall from \eqref{tanglecount}, for $m\leq l$, $|(B^{(m)}\mathbf{1})_i|=O(\alpha^{m}\log n)$, then by Cauchy inequality
\begin{align*}
	\left|\sum_{i\in \mathcal B}x_i(\alpha^{m-1-l}(B^{(l)}\mathbf{1})_i+ {O}(\alpha^{\frac{m-1}{2}}\log n))\right|
	\leq & \sqrt{|\mathcal B|}{O}(\alpha^{m-1}\log n)={O}(\alpha^{l+m-1}\log^3 n).
\end{align*}

Similarly, the third term satisfies 
\begin{align*}
|\sum_{i\in \mathcal B}x_i(B^{(m-1)}\mathbf{1})_i|={O}(\alpha^{l+m-1}\log^3 n).	
\end{align*}
Note that $\alpha^{l+m-1}=o(n^{1/2}),$ altogether we have 
\begin{align}
|\mathbf{1}^{\top}B^{(m-1)}x|={O}(\sqrt{n}\alpha^{\frac{m-1}{2}}\log n+\alpha^{l+m-1}\log^3 n)={O}(\sqrt{n}\alpha^{\frac{m-1}{2}}\log n).	
\end{align} \eqref{bound1} then follows. Using the property $x^{\top}B^{(l)}\sigma=0$ instead of $x^{\top}B^{(l)}\mathbf{1}=0$ and following the same argument, \eqref{bound2} holds.
\end{proof}

\subsection{Proof of Lemma \ref{isom}}\label{A7}
\begin{proof}
Conditioned  on $(H,i,\sigma)_{t-1}\equiv (T,\rho,\tau)_{t-1}$, if $A_t$ holds, it implies that hyperedges generated from vertices in $V_{t-1}$ do not overlap (except for the parent vertices in $V_{t-1}$). If $B_t$ holds, vertices in $V_t$ that are in different hyperedges generated from $H_{t-1}$ do not connect to each other. If both $A_t$ $B_t$ holds, $(H,i,\sigma)_t$ is still a hypertree. Since $X_v^{(r)}=Y_{\phi(v)}^{(r)}$ for $v\in V_{t-1}$, we can extend the hypergraph isomorphism $\phi$ by mapping the children of $v\in V_t$ to the corresponding vertices in the $t$-th generation of children of $\rho$ in $T$, which keeps the hypertree structure and the spin of each vertex.
\end{proof}

\subsection{Proof of Lemma \ref{AtBt}}\label{A8}
\begin{proof}
	First we fix $u,v\in V_{t}$. For any $w\in V_{>t}$, the probability that $(u,w), (v,w)$ are both connected is $O(n^{-2})$. We know $|V_{>t}|\leq n$ and $|V_{\leq t}|=O(\log^2(n)\alpha^t)$ conditioned on $C_{t}$. Since $\alpha^{2t}\leq \alpha^{2l}=o(n^{1/2})$, taking a union bound over all $u,v,w$ we have 
	\begin{align}
	\mathbb P(A_t |C_{t})\geq 1-O(\log^4(n)\alpha^{2t}n^{-1})=1-o(n^{-1/2}).	
	\end{align}
	
 For the second claim, the probability of having an edge between $u,v\in V_t$ is $O(n^{-1})$. Taking a union bound over all pairs of $u,v\in V_{t}$ implies
 \begin{align}
 \mathbb P(B_t |C_{t})\geq 1-O(\log^4(n)\alpha^{2t}n^{-1})=1-o(n^{-1/2}).	
 \end{align}
\end{proof}

\subsection{Proof of Lemma \ref{rama1}} \label{A11}
\begin{proof}
	In \eqref{B1}, the coordinates of two vectors on the left hand side agree at $i$ if the $l$-neighborhood of $l$ contains no cycle. Recall $\mathcal B$ is the set of vertices whose $l$-neighborhood contains a cycle, from Lemma \ref{tangle}, and \eqref{tanglecount}, we have asymptotically almost surely,
	\begin{align}\label{on}
	\|B^{(l)}\mathbf{1}-\vec{S}_l\|_2\leq \sqrt{|\mathcal B|}O(\log (n)\alpha^l)=O(\log^3(n)\alpha^{2l})=o(\sqrt{n}).	
	\end{align}
	
	From \eqref{inprob} we have 
	\begin{align}\label{DL}
	\|\vec{D}_l\|_2=\Theta(\sqrt{n}\beta^l)\end{align} asymptotically almost surely, and  $\|B^{(l)}\mathbf{1}\|_2\geq \|\vec{D}_l\|_2$, therefore \eqref{B1} follows.  Similar to \eqref{on}, we have 
	\begin{align}
		\|B^{(l)}\sigma-\vec{D}_l\|_2=o(\sqrt{n}),\quad 
		 \|B^{(l)}\sigma\|_2 = \|\vec{D}_l\|_2 + o(\sqrt{n})=\Theta(\sqrt{n}\beta^l).\label{Blsigma}
	\end{align}
	Then \eqref{B2} follows. It remains to show \eqref{B3}. Using the same argument as in Theorem \ref{thresholding}, we have the following convergence in probability 
	\begin{align}
	\lim_{n\to\infty}\frac{1}{n}\sum_{i\in[n]}\alpha^{-2l}S_l^2(i)&=\mathbb EM_{\infty}^2,	\label{Sl}
	\end{align}
where $M_{\infty}$ is the limit of the martingale $M_t$. Similarly, the following convergences in probability hold
\begin{align*}
\lim_{n\to\infty}\frac{1}{n}\sum_{i\in[n]}\alpha^{-l}\beta^{-l}S_l(i)D_l(i)
=&\lim_{n\to\infty}\frac{1}{n}\sum_{i\in\mathcal N^+}\alpha^{-l}\beta^{-l}S_l(i)D_l(i)+\lim_{n\to\infty}\frac{1}{n}\sum_{i\in\mathcal N^-}\alpha^{-l}\beta^{-l}S_l(i)D_l(i)\\
=&\frac{1}{2}\mathbb EM_{\infty}D_{\infty}-\frac{1}{2}\mathbb EM_{\infty}D_{\infty}=0.		
\end{align*}
Thus $\langle \vec{S}_l, \vec{D}_l\rangle=o(n\alpha^l\beta^l)$ asymptotically almost surely.  From \eqref{Sl} we have 
\begin{align}\label{Size}
\|\vec{S}_l\|_2=\Theta(\sqrt{n}\alpha^l),	
\end{align}
therefore together with \eqref{DL}, we have
$\|\vec{S}_l\|_2 \cdot \|\vec{D}_l\|_2=\Theta(n\alpha^l\beta^l).$
With \eqref{B1} and \eqref{B2},  \eqref{B3} holds. 
\end{proof}

\subsection{Proof of Lemma \ref{rama2}}\label{A12}
\begin{proof}
	For the lower bound in \eqref{Bound1}, note that $B^{(l)}$ is symmetric, we have
	\begin{align}
	\|B^{(l)}\mathbf{1}\|_2^2=\langle B^{(l)}\mathbf{1},B^{(l)}\mathbf{1}\rangle =\langle \mathbf{1}, B^{(l)}B^{(l)}\mathbf{1}\rangle\leq \|\mathbf{1}\|_2 \|B^{(l)}B^{(l)}\mathbf{1}\|_2.
	\end{align}
Therefore from \eqref{Size}  and \eqref{B1},
\begin{align}
	\|B^{(l)}B^{(l)}\mathbf{1}\|_2\geq \frac{\|B^{(l)}\mathbf{1}\|_2^2}{\|\mathbf{1}\|_2}=\Theta(\alpha^l)\|B^{(l)}\mathbf{1}\|_2.	
\end{align}

For the upper bound in \eqref{Bound1}, from  \eqref{claim1} and  \eqref{tanglecount},  the maximum row sum of $B^{(l)}$ is  $O(\alpha^l\log n)$. Since $B^{(l)}$ is nonnegative, the spectral norm  $\rho(B^{(l)})$ is bounded by the maximal row sum, hence \eqref{Bound1} holds. 
The lower bound in \eqref{Bound2} can be proved similarly as in \eqref{Bound1},  from the inequality 
$  \|B^{(l)}\sigma\|_2^2\leq \|\sigma\|_2 \|B^{(l)}B^{(l)}\sigma\|_2
$ together with  \eqref{DL}  and \eqref{B2}. 
Recall $\mathcal B$ is the set of vertices whose $l$-neighborhood contains  cycles. Let $\overline{\mathcal B}=[n]\setminus \mathcal B$.
Since $$\left(B^{(l)}B^{(l)}\sigma\right)_i=\sum_{j\in [n]}B_{ij}^{(l)}(B^{(l)}\sigma)_j,$$
we can decompose the vector $B^{(l)}B^{(l)}\sigma$ as a sum of three vectors $z+z'+z''$, where
\begin{align*}
z_i &:=\mathbf{1}_{\overline{\mathcal B}}(i)\sum_{j: d(i,j)=l}D_l(j)\mathbf{1}_{\overline{\mathcal B}}(j),	\quad 
z_i' :=\mathbf{1}_{\overline{\mathcal B}}(i)\sum_{j: d(i,j)=l}O(\alpha^l\log n)\mathbf{1}_{\mathcal B}(j),\\
z_i'' &:=\mathbf{1}_{{\mathcal B}}(i)O(\alpha^{2l}\log ^2n).  \end{align*}
The decomposition above depends on whether $i,j\in\mathcal B$ and the estimation follows from \eqref{tanglecount}.
From Lemma \ref{tangle}, $\mathcal B=O(\alpha^{2l}\log^4(n))$ asymptotically almost surely, so one has 
\begin{align*}
\|z'\|_2^2&=\sum_{i=1}^n (z_i')^2=\sum_{i\in \overline{\mathcal B}}\sum_{j: d(i,j)=l}\sum_{j': d(i,j')=l} O(\alpha^{2l}\log^2 n)\mathbf{1}_{\mathcal B}(j)\mathbf{1}_{\mathcal B}(j')\\
&=\sum_{j\in\mathcal B}\sum_{j'\in\mathcal B}\sum_{\substack{i\in\overline{\mathcal B}\\ d(i,j)=d(i,j')=l}}O(\alpha^{2l}\log^2 n)=\sum_{j,j'\in\mathcal B}O(\alpha^{3l}\log^3 n)=O(\alpha^{7l}\log^{11} n),\end{align*}
which implies  $\|z'\|_2=O(\alpha^{7l/2}\log^{11/2}n).$ And similarly $\|z''\|_2=O( \alpha^{3l}\log^2n).$
We know from \eqref{Blsigma}, $\|B^{(l)}\sigma\|_2=\Theta(\beta^l\sqrt n)$, and since $c\log\alpha<1/8$, we have $\alpha^{5l/2}=n^{-\gamma'}\sqrt n$ for some $\gamma'>0$, therefore
\begin{align}\label{zzprime}
	\|z'+z''\|_2=O(\alpha^{7l/2}\log^{11/2} n)=o(\alpha^{5l/2}\beta^{2l})=O(n^{-\gamma'}\beta^{l}\|B^{(l)}\sigma\|_2). 
\end{align}  

It remains to upper bound $\|z\|_2$. Assume the $2l$-neighborhood of $i$ is cycle-free, then the $i$-th entry of $B^{(l)}B^{(l)}\sigma$, denoted by $X_i$, can be written as 
\begin{align}
X_i:&=(B^{(l)}B^{(l)}\sigma)_i=\sum_{k=1}^n B_{ik}^{(l)}(B^{(l)}\sigma)_k
=\sum_{k=1}^n \mathbf{1}_{d(i,k)=l}\sum_{j=1}^n \mathbf{1}_{d(j,k)=l}\sigma_j \notag\\
&=\sum_{h=0}^l	\sum_{j: d(i,j)=2h}\sigma_j|\{k: d(i,k)=d(j,k )=l\}|.\label{expansionX}
\end{align}

We control the magnitude of $X_i$ in the corresponding hypertree growth process. Since $2l=2c\log n$ and $2c\log(\alpha)<1/4$, the coupling result in Theorem \ref{coupling2} can apply. 
Let $\mathcal C_i$ be the event that coupling between $2l$-neighborhood of $i$ with the Poisson Galton-Watson hypertree has succeeded and $n^{-\epsilon}$ be the failure probability of the coupling. When the coupling succeed, $z_i=X_i$, therefore
\begin{align}
\mathbb E (\|z\|_2^2\mid \Omega)&=\sum_{i\in [n]}n^{-\epsilon}O(\alpha^{2l}\beta^{2l}\log^2n)+\sum_{i\in [n]}\mathbb E (X_i^2\mathbf{1}_{\mathcal C_i}\mid \Omega)\notag\\
&=n^{1-\epsilon}O(\alpha^{2l}\beta^{2l}\log^2n)+\sum_{i\in [n]}\mathbb E (X_i^2\mathbf{1}_{\mathcal C_i}\mid \Omega).\label{expz}
\end{align}

For any $i,j\in [n], t\in [l]$, define 
$D_{i,j}^{(t)}:=|\{k: d(i,k)=d(j,k )=t\}|.$
From \eqref{expansionX}, we have 
\begin{align}\label{expandX2}
X_i^2=\sum_{h,h'=0}^l \sum_{j: d(i,j)=2h}\sum_{j':d(i,j')=2h'}\sigma_j\sigma_{j'}D_{i,j}^{(l)}D_{i,j'}^{(l)}.
\end{align}

We further classify the pair $j,j'$ in \eqref{expandX2} according to their distance. Let $d(j,j')=2(h+h'-\tau)$ for $\tau=0,\dots, 2(h\wedge h')$.  This yields
\begin{align*}
X_i^2
&=\sum_{h,h'=0}^l \sum_{\tau=0}^{2(h\wedge h')}\sum_{j:d(i,j)	=2h}\sum_{j':d(i,j')	=2h'}\mathbf{1}_{d(j,j')=2(h+h'-\tau)}\sigma_j \sigma_j'  D_{i,j}^{(l)}D_{i,j'}^{(l)}.\end{align*}

Conditioned on $\Omega$  and $\mathcal C_i$, similar to the analysis in Appendix H in \cite{massoulie2014community}, we have the following holds 
\begin{align}
|\{k: d(i,k)=d(j,k )=l\}| &=O(\alpha^{l-h}\log n),\label{estimatedistance1}\\	
|\{k': d(i,k') =d(j',k' )=l\}|&=O(\alpha^{l-h'}\log n),\\
|\{j: d(i,j) =2h\}|&=O(\alpha^{2h}\log n),\\
|\{j': d(i,j') =2h', d(j,j')=2(h+h'-\tau)\}|&=O(\alpha^{2h'-\tau }\log n).\label{estimatedistance4}
\end{align}

We claim that \begin{align}\label{keyclaim}
 	\mathbb E [\sigma_j \sigma_{j'}|\mathcal C_i]\leq \left(\frac{\beta}{\alpha}\right)^{d(j,j')-1},
 \end{align}
and  prove \eqref{keyclaim}  in Cases (a)-(d).

(a) Assume $j$ is the parent of $j'$ in the hypertree growth process. Then $d(j,j')=1$. Let $\mathcal T_r$ be the event that the hyperedge containing $j'$ is of type $r$. Given $\mathcal T_r$, by our construction of the hypertree process, the spin of $j'$ is assigned to be $\sigma_j$ with probability $\frac{r}{d-1}$ and $-\sigma_j$ with probability $\frac{d-1-r}{d-1}$, so we have 
\begin{align*}
\mathbb E[\sigma_j\sigma_{j'}\mid \mathcal C_i]&=\sum_{r=0}^{d-1}\mathbb E[\sigma_j\sigma_j'\mid \mathcal T_r, C_i]\mathbb P[\mathcal T_r\mid \mathcal C_i]=	\sum_{r=0}^{d-1}\left( \frac{r}{d-1}-\frac{d-1-r}{d-1}\right)\mathbb P[\mathcal T_r\mid \mathcal C_i].
\end{align*}
Recall $\mathbb P[\mathcal T_{d-1}\mid \mathcal C_i]=\frac{(d-1)a}{\alpha 2^{d-1}}$ and   $\mathbb P[\mathcal T_{r}\mid \mathcal C_i]=\frac{(d-1)b{ d-1 \choose r}}{\alpha 2^{d-1}}$ for $0\leq r\leq d-2$.
 A simple calculation implies 
$\mathbb E[\sigma_j\sigma_{j'}\mid \mathcal C_i]=\frac{\beta}{\alpha}\leq 1.
$

(b) Suppose $d(j,j')=t$ and there is a sequence of vertices $j, j_1,\dots, j_{t-1}, j'$ such that $j_1$ is a child of $j$, $j_i$ is a child of $j_{i-1}$ for $1\leq i\leq t$, and $j'$ is a child of $j_{t-1}$. We  show by induction that for $t\geq 1$, 
$\mathbb E[\sigma_j\sigma_{j'}\mid \mathcal C_i]= \left(\frac{\beta}{\alpha}\right)^t.$
When $t=1$ this has been approved in part (a). Assume it is true for all $j,j'$ with distance $\leq t-1$. Then when $d(j,j')=t$, we have 
\begin{align*}
\mathbb E[\sigma_j\sigma_{j'}\mid \mathcal C_i]=&
\mathbb E[\sigma_j\sigma_{j'}\mid \sigma_{j_1}=\sigma_j, \mathcal C_i]	\mathbb P(\sigma_{j_1}=\sigma_j\mid \mathcal C_i)+\mathbb E[\sigma_j\sigma_{j'}\mid \sigma_{j_1}=-\sigma_j, \mathcal C_i]	\mathbb P(\sigma_{j_1}=-\sigma_j\mid \mathcal C_i) \\
=&\left(\frac{\beta}{\alpha}\right)^{t-1}\mathbb P(\sigma_{j_1}=\sigma_j\mid \mathcal C_i)-\left(\frac{\beta}{\alpha}\right)^{t-1}\mathbb P(\sigma_{j_1}=-\sigma_j\mid \mathcal C_i)\\
=&\left(\frac{\beta}{\alpha}\right)^{t-1}\frac{\alpha+\beta}{2\alpha}-\left(\frac{\beta}{\alpha}\right)^{t-1}\frac{\alpha-\beta}{2\alpha}=\left(\frac{\beta}{\alpha}\right)^{t}.
\end{align*}
Therefore $\mathbb E[\sigma_j\sigma_{j'}\mid \mathcal C_i]\leq \left(\frac{\beta}{\alpha}\right)^{d(j,j')}\leq \left(\frac{\beta}{\alpha}\right)^{d(j,j')-1}$.
This completes the proof for part (b).

(c) Suppose $j,j'$ are not in the same hyperedge and there exists a vertex $k$ such that $j,k$ satisfies the assumption in Case (b) with $d(j,k)=t_1$, and $j',k$ satisfy the assumption in Case (b) with $d(j',k)=t_2$. Conditioned on $\sigma_k$, we know $\sigma_j$ and $\sigma_j'$ are independent. Then  we have 
\begin{align*}
\mathbb E[\sigma_j\sigma_{j'}\mid \mathcal C_i]&=\mathbb E[\mathbb E[\sigma_j\sigma_{j'}\sigma_k^2\mid \sigma_k, \mathcal C_i]\mid \mathcal C_i]=\mathbb E\left[\mathbb E[\sigma_j\sigma_k \mid \sigma_k,\mathcal C_i]\cdot \mathbb E[\sigma_{j'}\sigma_k \mid \sigma_k,\mathcal C_i]	\mid \mathcal C_i\right]\\
&= \left(\frac{\beta}{\alpha}\right)^{t_1+t_2}\leq \left(\frac{\beta}{\alpha}\right)^{d(j,j')-1},
\end{align*}
where the last line follows from the triangle inequality $d(j,k)+d(j',k)\geq d(j,j')$ and the condition $\beta< \alpha$.

(d) If $j,j'$ are in the same hyperedge, then $d(j,j')=1$ and \eqref{keyclaim} holds trivially.

Combining Cases (a)-(d), \eqref{keyclaim} holds. 
From \eqref{keyclaim} and \eqref{estimatedistance1}-\eqref{estimatedistance4}, we have
\begin{align}
\mathbb E[X_i^2\mathbf{1}_{\Omega}\mid \mathcal C_i] 
\leq &  \sum_{h,h'=0}^l \sum_{\tau=0}^{2(h\wedge h')}\sum_{j:d(i,j)	=2h}\sum_{j':d(i,j')	=2h'}\mathbf{1}_{d(j,j')=2(h+h'-\tau)}\mathbb E[\sigma_j \sigma_j'\mid \mathcal C_i]  R_{i,j}^{(l)}R_{i,j'}^{(l)} \notag\\
\leq &\sum_{h,h'=0}^l \sum_{\tau=0}^{2(h\wedge h')}\sum_{j:d(i,j)	=2h}O(\alpha^{2h'-\tau}\log n)\left(\frac{\beta}{\alpha}\right)^{2(h+h'-\tau)-1} \cdot  O(\alpha^{2l-h-h'}\log^2 n)\notag\\
= & \sum_{h,h'=0}^l\sum_{\tau=0}^{2(h\wedge h')}O(\alpha^{2l+h+h'-\tau}\log^4 n)\left(\frac{\beta}{\alpha}\right)^{2(h+h'-\tau)-1}\notag\\
=&\sum_{h,h'=0}^l\sum_{\tau=0}^{2(h\wedge h')} O(\alpha^{2l}\log^4 n)\cdot (\beta^2/\alpha)^{h+h'-\tau}=O(\beta^{4l}\log^4 n).	\label{Xsquare}
\end{align}
From \eqref{expz} and \eqref{Xsquare}, we have for some $\epsilon>0$,
\begin{align*}
	\mathbb E(\|z\|_2^2 \mid \Omega)=n^{1-\epsilon}O(\alpha^{2l}\beta^{2l}\log^2 n)+O(n\beta^{4l}\log^2n).
\end{align*}
Then by Chebyshev's inequality, asymptotically almost surely, 
\begin{align*}\|z\|_2= & O(n^{1/2-\epsilon/2}\alpha^l \beta^l\log^2 n)+O(n^{1/2}\beta^{2l}\log^2 n)
=(\sqrt{n}\beta^l\log^2n)\cdot O(\beta^l \vee \alpha^l n^{-\epsilon/2} ).
\end{align*}

Recall $l=c\log n$. We have $\beta^{l}=n^{c\log\beta}, \alpha^{l}=n^{c\log \alpha}.$ So $\beta^{l}=n^{-\epsilon'}\alpha^l$ for some constant $\epsilon'>0$.
 Since from \eqref{Blsigma}, $\|B^{(l)}\sigma\|_2=\Theta(\sqrt{n}\beta^{l}),$ we have
 \begin{align}\label{eq:z2}
 	\|z\|_2=O(n^{-\gamma''}\alpha^{l}\|B^{(l)}\sigma\|_2) 
 \end{align}
 for some constant $\gamma''>0$. Combining \eqref{zzprime} with \eqref{eq:z2}, it implies for some constant $\gamma>0$,
 \[\|B^{(l)}B^{(l)}\sigma\|_2=\|z+z'+z''\|_2=O(n^{-\gamma}\alpha^l)\|B^{(l)}\sigma\|_2.
 \]
 Then the upper bound on $\|B^{(l)}B^{(l)}\sigma\|_2$ in \eqref{Bound2}  holds.
\end{proof}

\end{appendix}

 \bibliographystyle{plain}
\bibliography{globalref.bib}

\end{document}